\documentclass[11pt,a4paper]{siamltex1213}

\usepackage{amssymb}
\usepackage{amsmath}
\usepackage{graphicx}
\usepackage{color}

\newtheorem{remark}[theorem]{Remark}

\def\mm#1{\vspace*{-0.5ex}\begin{align}#1\end{align}\vspace*{-0.5ex}\noindent}

\def\nn{\nonumber\\}

\newcommand{\softd}{{\leavevmode\setbox1=\hbox{d}%
          \hbox to 1.05\wd1{d\kern-0.4ex{\char039}\hss}}}

\newcommand{\abs}[1]{\lvert#1\rvert}
\newcommand{\norm}[1]{\lVert#1\rVert}
\newcommand{\D}{\partial}

\newcommand{\dt}{\Delta t}
\newcommand{\dx}{\Delta x}
\newcommand{\dy}{\Delta y}

\newcommand{\mbf}{\mathbf}

\def\bold#1{\mbox{\boldmath $#1$}}
\newcommand{\uu}[1]{\bold{#1}}

\newcommand{\mbb}{\mathbb}
\newcommand{\veps}{\varepsilon}
\newcommand{\NL}{\mathrm{NL}}
\newcommand{\en}{\rho E}
\newcommand{\rf}{\mathrm{ref}}
\newcommand{\Id}{\mbf{Id}} 
\newcommand{\sign}{\mathrm{sign}}
\def\rhohat{\widehat{\rho}}
\def\ub{{\mbf u}}
\def\uhat{\widehat{\mbf u}}
\def\phat{\widehat{p}}
\def\momhat{\widehat{\rho \mbf u}}
\def\enhat{\widehat{\rho E}}
\def\nuhat{\hat\nu}

\def\cstab{{\color{black}c_{\mathbf{stab}}}}

\title{A Weakly Asymptotic
Preserving Low Mach Number Scheme for the Euler Equations of Gas Dynamics
\thanks{S.N. is supported by the DFG under grants number NO 361/3-2 and GSC 111. 
A.B. and M.L. are supported by the DFG under grant number LU 1470/2-2.
K.R.A. was supported by the Alexander-von-Humboldt Foundation
through a postdoctoral fellowship (2011-12).
C.-D.M. is supported by the DFG in the Cluster of Excellence
"Simulation Technology" and SFB TRR 40}}

\author{S. Noelle
 \thanks{Institut f{\"u}r Geometrie und Praktische Mathematik,
  RWTH-Aachen,  Templergraben 55,  D-52056 Aachen,  Germany.
  {\tt noelle@igpm.rwth-aachen.de}}
  \and G. Bispen \thanks{Institut f{\"u}r Mathematik,
  Johannes Gutenberg-Universit{\"a}t Mainz,
 Staudingerweg 9,
   D-55099 Mainz,
   Germany.{\tt bispeng@mathematik.uni-mainz.de } } 
   \and K. R. Arun \thanks{School of Mathematics, Indian Institute of Science Education and 
   Research Thiruvananthapuram,
 India.{\tt arun@iisertvm.ac.in}}
 \and M. Luk\'a\v{c}ov\'a-Medvi\v{d}ov\'a \thanks {Institut f{\"u}r Mathematik,
  Johannes Gutenberg-Universit{\"a}t Mainz,
   Staudingerweg 9,
  D-55099 Mainz,
   Germany. {\tt lukacova@mathematik.uni-mainz.de}}
   \and C.-D. Munz \thanks {Institut f{\"u}r Aerodynamik und Gasdynamik,
   Universit{\"a}t Stuttgart,
  Pfaffenwaldring 21,
  D-70550 Stuttgart,
  Germany. {\tt munz@iag.uni-stuttgart.de} }}

\begin{document}

\maketitle

\begin{AMS}
{[2010] Primary 35L65, 76N15, 76M45; Secondary 65M08, 65M06}
\end{AMS}

\date{26. August 2014} 


\pagestyle{myheadings}
\thispagestyle{plain}
\markboth{Noelle, Bispen, Arun, Luk\'a\v{c}ov\'a, Munz}{A Weakly Asymptotic Preserving Low Mach number scheme}

\begin{abstract}
 We propose a low Mach number,
 Godunov-type finite volume scheme for the numerical solution of the
 compressible Euler equations of gas dynamics. The scheme combines
 Klein's non-stiff/stiff decomposition of the fluxes
 (J.\ Comput.\ Phys.\ 121:213-237, 1995) with an explicit/implicit
 time discretization (Cordier et al.\/, J.\ Comput.\ Phys.\ 231:5685-5704, 2012)
 for the split fluxes. This results in a scalar
 second order partial differential equation (PDE) for the
 pressure, which we solve by an iterative approximation.
 Due to our choice of a crucial reference pressure, the stiff
 subsystem is hyperbolic, and the
 second order PDE for the pressure is elliptic.
 The scheme is also uniformly asymptotically consistent.
 Numerical experiments show that the scheme needs to be stabilized
 for low Mach numbers. Unfortunately, this affects the asymptotic consistency,
 which becomes non-uniform in the Mach number, and requires an unduly fine grid in the
 small Mach number limit. On the other hand, the CFL number
 is only related to the non-stiff characteristic speeds, independently of the Mach number.
 Our analytical and numerical results stress the importance of further studies of asymptotic
 stability in the development of AP (asymptotic preserving) schemes.
\end{abstract}

\section{Introduction}
\label{sec:intro}

We consider the non-dimensionalised compressible Euler equations
for an ideal gas, which may be written as a system of conservation
laws in $d=1,2$ or $3$ space dimensions,
\begin{equation}
\label{eq:euler_system}
  U_t + \nabla \cdot F(U)=0,
\end{equation}
where $t>0$ and $\mbf{x} \in \mbb R^d$ are the time and space
variables, $U \in \mbb R^{d+2}$ is the vector of conserved
variables and $F(U) \in \mbb R^{(d+2)\times d}$, the flux vector.
Here,
\begin{equation}
\label{eq:U-F}
 U
=
 \begin{pmatrix}
  \rho\\
  \rho\mbf{u}\\
  \rho E
 \end{pmatrix},
\
 F(U)=
 \begin{pmatrix}
  \rho\mbf{u}\\
  \rho\mbf{u}\otimes\mbf{u}+\frac{p}{\veps^2} \, \Id \\
  (\rho E+p)\mbf{u}
 \end{pmatrix},
\end{equation}
with density $\rho$, velocity $\mbf u$, momentum $\rho \mbf u$,
total specific energy $E$, total energy $\en$ and pressure $p$.
The operators $\nabla,\nabla\cdot,\otimes$ are respectively the
gradient, divergence and tensor product in $\mbb R^d$. The
parameter $\veps$ is the reference Mach number and is usually
given by
\begin{equation}
 \veps
:=
 \frac{u_{\rf}}{\sqrt{p_{\rf}/\rho_{\rf}}},
\label{eq:mach}
\end{equation}
where the basic reference values $\rho_{\rf}, u_{\rf}$, $p_{\rf}$
and length $x_{\rf}$ are problem dependent characteristic numbers.
It has to be noted that the reference Mach number $\veps$ is a
measure of compressibility of the fluid. Throughout this paper we
follow the convention that $0<\veps\leq 1$ and $\veps=1$
corresponds to the fully compressible regime. On the other hand,
if the values of the reference parameters are prescribed in such a
way that $\veps>1$, then we redefine them to get $\veps=1$. This
can be achived, e.g.\ by setting $u_\rf=p_\rf=\rho_\rf=1$.

The system \eqref{eq:euler_system} is closed by the dimensionless
equation of state

\begin{equation}
\label{eq:state_ndim}
 p
 =
 (\gamma-1)
 \left(
 \en - \frac{\veps^2}{2} \rho \norm{\mbf{u}}^2
 \right),
\end{equation}
with $\gamma>1$ being the ratio of specific heats. As long as the
pressure $p$ remains positive, \eqref{eq:euler_system} is
hyperbolic and the eigenvalues in direction $\mbf{n}$ are
\begin{equation}
  \label{eq:eig_euler}
  \lambda_1 = \mbf u \cdot \mbf n - \frac c \veps, \quad
  \lambda_2 = \mbf u \cdot \mbf n, \quad
  \lambda_3 = \mbf u \cdot \mbf n + \frac c \veps,
\end{equation}
where $c = \sqrt{(\gamma p)/\rho}$.  Therefore, the sound speed
$c/\veps$ becomes very large, $\mathcal O(\veps^{-1})$, as the
Mach number $\veps$ becomes small, while the advection speed $\mbf
u$ remains finite. The spectral condition number of the flux
Jacobian becomes proportional to $1/\epsilon$  and numerical
difficulties will occur. If acoustic effects can be neglected,
then the physical time scale is given by
$t_{ref}=x_{\rf}/u_{\rf}$. For an explicit  time discretization,
the Courant-Friedrichs-Lewy condition (CFL number) imposes a
numerical time step $\Delta t$ that  has to be proportional to
$\veps$. Hence, the numerical time step has to be much
smaller  than the physical one and becomes intolerable for $\veps
\ll 1$.

There is an abundance of literature on the computation of low Mach
number weakly compressible and incompressible flows. A pioneering
contribution is due to Chorin \cite{chorin} who introduced the
projection method for incompressible flows via an artificial
compressibility. In \cite{turkel}, Turkel introduced the well
known pre-conditioning approach to reduce the stiffness of the
problem. In the case of unsteady weakly compressible problems,
Klein \cite{klein} derives a multiple pressure variables approach
using low Mach number asymptotic expansions. Klein's work was
later extended to both inviscid and viscous flow problems by Munz
and collaborators, see, e.g.\ \cite{munz-etal,park}. The review
paper \cite{klein-etal} gives a good survey of the recent develops
on the asymptotic based low Mach number approximations and further
references on this subject.

This paper aims to contribute to two well known, challenging
issues: First, we propose and analyse a scheme which is based upon
Klein's non-stiff/stiff splitting \cite{klein}. We combine this
with an explicit/implicit time discretization due to Cordier,
Degond and Kumbaro \cite{cordier}, which transforms the stiff
energy equation into a second order nonlinear PDE for the
pressure.
Due to our choice of reference pressure, both the non-stiff and
the stiff subsystems are hyperbolic. The implicit terms in the stiff
system are reduced to a second order elliptic PDE for
the pressure and is solved at every time step by an iterative approximation.
This procedure allows a
CFL number which is only bounded by the advection velocity (i.e.\ it is
independent of the Mach number $\veps$).

The second issue is related to Klainerman and Majda's celebrated
result that solutions to the compressible Euler equations converge
to the solutions of the incompressible Euler equations as the Mach
number tends to zero. Jin \cite{jin-ap} introduced a related
consistency criterion for numerical schemes: a scheme is
asymptotic preserving (AP), if its lowest order multiscale
expansion is a consistent discretization of the incompressible
limit; see the equations \eqref{mass0}-\eqref{div0}. An AP scheme
for the low Mach number limit of the Euler equations should serve
as a consistent and stable discretization, independent the Mach
number $\veps$. In the fully compressible regime, i.e.\ when
$\veps=\mathcal{O}(1)$, the AP scheme should have the desirable
features of a compressible solver, such as non-oscillatory
solution profiles and good resolution of shock type
discontinuities. On the other hand, the AP scheme should yield a
consistent discretization of the incompressible equations when
$\veps\to 0$. We refer the reader to
\cite{berthon-turpault,bijl-wesseling,cordier,degond,haack} and
the references cited therein for some AP schemes for the Euler
equations and other related applications. In Section \ref{sec:ap}
we prove that our scheme possess the AP property.

While Klein's paper motivated strongly such a pressure-based
approach, in which the pressure is used as a primary variable,
Guillard and Viozat \cite{guillard} 
proposed a further development of the pre-conditioning technique
based on the low Mach number asymptotic analysis. The question
with respect to efficiency of this density-based approach in
comparison to the pressure-based has no clear answer. The
pressure-based solver has its advantage especially in the low Mach
number regime due to the splitting of stiff and non-stiff terms.
All the implicit treatment of the stiff terms is reduced to a
scalar elliptic equation for the pressure. For small Mach numbers
such a pressure-based approach is usually used in commercial
codes. In the fully compressible regime the pressure-based
approach has the restriction of the implicit treatment of the
sonic terms even for $M>1$, which may be no longer necessary. The
density-based approach has the advantage that it starts from a
typical compressible solver as used in aerodynamics to get a
steady state. Hence, we expect that it will be the best approach
for larger Mach numbers. At small Mach numbers, the
preconditioning reduces the sound speed of the system to avoid the
stiffness in the equations. By a dual time stepping technique it
may be extended to unsteady flows. Hence, for very small Mach
numbers the density-based approach should be more difficult - the
$M=0$ case can not be handeled.

The rest of this paper is organised as follows.
Section~\ref{sec:prelim} contains several preliminaries: in
Section~\ref{sec:multiscale} we briefly recall the asymptotic
analysis as well as the incompressible limit due to Klainerman and
Majda \cite{klainerman-majda}. In Section~\ref{sec:splitting} we
recall Klein's flux splitting  and introduce our choice of
reference pressure. In Section~\ref{sec:well_prepared} we prove
that the non-stiff subsystem respects the structure of
divergence-free velocity and constant pressure fields. This
property helps to avoid spurious initial layers for low Mach
number computations.

In Section~\ref{sec:time-discr} we introduce our time discretization and prove
the AP property: in Section~\ref{sec:first-order}, we propose the first order
scheme IMEX scheme. In Section~\ref{sec:elliptic_equation} we derive the
nonlinear elliptic pressure equation and discuss an iterative linearization to
solve it. In Section~\ref{sec:ap} we prove the AP property for this scheme. In
Section~\ref{sec:second-order} we define a second-order time discretization
based on the Runge-Kutta Crank-Nicolson (RK2CN) scheme. At this point our scheme
is not uniformly asymptotically stable with respect to $\veps$. Indeed, for some
test problems, the scheme is only stable if $\dt = O(\veps\dx)$. Therefore, we
introduce a fourth order pressure stabilization in
Section~\ref{sec:stabilization}. We prove asymptotic consistency in
Theorem~\ref{theorem:dtdx23}. While the ratio $\dt/\dx$ is now independent of
$\veps$, we need to restrict both $\dt$ and $\dx$ as the Mach number goes to
zero as $\dt = O(\veps^{2/3}) = \dx$. 

Section~\ref{sec:discrete} describes the fully discrete scheme
with spatial reconstruction, numerical fluxes and a linear system
solver. In Section~\ref{sec:numerics} we present the numerical
experiments.

Finally, we draw some conclusions in Section~\ref{sec:conclusion}. In
particular, we discuss possible ways to analyze and improve asymptotic
stability.




\section{Asymptotic Analysis and Flux Splitting}
\label{sec:prelim}

In this section we review the asymptotic analysis presented in
\cite{klainerman-majda,klein}, from which the incompressible limit
equations are derived (Subsection~\ref{sec:multiscale}), a
splitting of the fluxes into non-stiff and stiff parts with a
technique to guarantee the hyperbolicity of the stiff part
(Subsection~\ref{sec:splitting}) and the importance of
well-prepared initial data which eliminate spurious initial layers
(Subsection~\ref{sec:well_prepared}).

\subsection{Asymptotic Analysis and Incompressible Limit}
\label{sec:multiscale}

In this section we consider the low Mach number limit of the Euler
equations \eqref{eq:euler_system}-\eqref{eq:U-F}, obtained via a
formal multiscale asymptotic analysis \cite{klainerman-majda}.
Following \cite{klein,meister} we use a three-term asymptotic
ansatz
\begin{equation}
  \label{eq:ansatz}
  f(\mbf{x},t)=f^{(0)}(\mbf{x},t)+\veps
  f^{(1)}(\mbf{x},t)+\veps^2f^{(2)}(\mbf{x},t)
\end{equation}
for all the flow variables. The ansatz \eqref{eq:ansatz} was
introduced by Klainerman and Majda \cite{klainerman-majda}. A
drawback of this ansatz is that it cannot resolve long wave
phenomena, particularly those related to acoustic waves. However,
our focus is on resolving slow convective wave, e.g.\ vortices. In
order to resolve long wavelength one has to consider multiple
space scales in \eqref{eq:ansatz} and multiple pressure variables
as \cite{klein,park}.

A multiscale analysis consists of inserting the ansatz
\eqref{eq:ansatz} into the Euler system
\eqref{eq:euler_system}-\eqref{eq:U-F} and balancing the powers of
$\veps$. This leads to a hierarchy of asymptotic equations which
shows the behaviour of the different order terms. In the
following, we briefly review the results of multiscale analysis
presented in \cite{klein,meister}.

The leading order terms in the conserved variables do not give a
completely determined system of equations. Even though the leading
order terms in density and velocity form a coupled system of
equations, the presence of second order pressure term makes the
system incomplete. Both the leading order and second order
pressure terms influence the leading order density and velocity
fields. The pressure $p(\mbf{x},t)$ admits the multiscale
representation \cite{klein}
\begin{equation}
  \label{eq:press}
  p(\mbf{x},t)=p^{(0)}(t)+\veps^2p^{(2)}(\mbf{x},t).
\end{equation}
Here, the leading order pressure term $p^{(0)}$ allows only
temporal variations and $p^{(0)}$ is a thermodynamic variable
satisfying the equation of state, i.e.\
\begin{equation}
  \label{eq:p0}
  p^{(0)}=(\gamma-1)(\rho E)^{(0)}.
\end{equation}
As a result of the compression or expansion at the boundaries, the
pressure $p^{(0)}$ changes in time and vice-versa, according to
the relation
\begin{equation}
  \label{eq:p0-eqn}
  \frac{1}{\gamma
    p^{(0)}}\frac{dp^{(0)}}{dt}=-\frac{1}{\abs{\Omega}}\int_{\D\Omega}
  \mbf{u}^{(0)}\cdot\mbf{n}d\sigma.
\end{equation}
As a consequence of \eqref{eq:p0-eqn}, it can be inferred that the
leading order velocity $\mbf{u}^{(0)}$ cannot be arbitrary; an
application of the Gauss theorem to the right hand side yields a
divergence constraint on $\mbf{u}^{(0)}$. In the limit $\veps\to
0$, $p^{(0)}$ becomes a constant and which gives the standard
divergence-free condition of incompressible flows.

The first order pressure $p^{(1)}$ is also function of time,
independent of $\mbf{x}$. However, it can admit long scale
variations, say $\uu{\xi}=\veps\mbf{x}$ and $p^{(1)}$ can be
thought of as the amplitude of an acoustic wave. Therefore, in
order to include the effect of $p^{(1)}$, instead of
\eqref{eq:ansatz}, one would require a multiple space-scale ansatz
in both $\mbf{x}$ and $\uu{\xi}$, cf.\ \cite{klein}. However, in
the limit $\veps\to 0$, the large scale becomes infinite and
$p^{(1)}$ becomes constant in space and time; see also
\cite{klainerman-majda}. Therefore, in this work we consider only
a single space-scale ansatz \eqref{eq:ansatz} and thus $p^{(1)}$
is constant in space.

In order to interpret the meaning of the second order term
$p^{(2)}$, we write the classical zero Mach number limit
equations, i.e.\
\begin{align}
  \rho^{(0)}_t+\nabla\cdot\left(\rho^{(0)}\mbf{u}^{(0)}\right)&=0,
  \label{mass0}\\
  \left(\rho^{(0)}\mbf{u}^{(0)}\right)_t+\nabla\cdot\left(\rho^{(0)}\mbf{u}^{(0)}
    \otimes\mbf{u}^{(0)}\right)+\nabla p^{(2)}&=0,\label{mom0}\\
  \nabla\cdot\mbf{u}^{(0)}&=-\frac{1}{\gamma
    p^{(0)}}\frac{dp^{(0)}}{dt}. \label{div0}
\end{align}
We notice that the above zero Mach number equation system has a
mixed hyperbolic-elliptic character in contrast to the hyperbolic
Euler equations \eqref{eq:euler_system}-\eqref{eq:U-F}. Here,
$p^{(2)}$ survives as the incompressible pressure which is a
Lagrange multiplier for the divergence-free constraint
\eqref{div0}.


\subsection{Flux-splitting into non-stiff and Stiff Parts}
\label{sec:splitting}

Based on the as\-ymptotic structure presented in the previous
paragraph, Klein \cite{klein} proposed a new splitting of the
Euler fluxes and a novel Godunov-type scheme for the low Mach
number regime. We refer to \cite{hoffmann} for a detailed
derivation. The flux-splitting reads
\begin{equation}
\label{eq:F_split}
 F(U)
=
 \hat{F}(U) + \tilde{F}(U),
\end{equation}
where
\begin{equation}
\label{eq:vars}
 \hat{F}(U)
=
 \begin{pmatrix}
  \rho\mbf{u} \\
  \rho \mbf{u} \otimes \mbf{u} +p \, \Id \\
  (\rho E + \Pi) \mbf{u}
 \end{pmatrix},
\
 \tilde{F}(U)
=
 \begin{pmatrix}
  0 \\
  \frac{1-\veps^2}{\veps^2} p \, \Id  \\
  (p - \Pi) \mbf{u}
 \end{pmatrix}.
\end{equation}
Here, $\Id$ denotes the $d\times d$ identity matrix and
$\Pi=\Pi(\mbf{x},t)$ is an auxiliary pressure variable which
should satisfy
\begin{equation}
  \label{eq:Pi_limits}
  \lim\limits_{\veps\to0} \Pi(\mbf{x},t) = p^{(0)}(t),
  \quad \lim\limits_{\veps\to1} \Pi(\mbf{x},t) = p(\mbf{x},t).
\end{equation}
This is achieved by defining
\begin{equation}
 \label{eq:pi}
 \Pi(\mbf{x},t)
:=
 \veps^2 p(\mbf{x},t) + (1-\veps^2){p}_{\infty}(t),
\end{equation}
where the reference pressure $p_\infty(t)$ has to satisfy
\begin{equation}
  \label{eq:p_infty_limits}
  \lim\limits_{\veps\to0} p_\infty(t) = p^{(0)}(t).
\end{equation}

\begin{remark}
  We give our choice of $p_\infty$ in \eqref{eq:p_infty_def} below.
  \begin{enumerate}
  \item [(i)] The limits \eqref{eq:Pi_limits} respectively
    \eqref{eq:p_infty_limits} and in particular the choice of the
    reference pressure $p_\infty$ play an important role in  assuring
    the correct asymptotic behaviour in the incompressible and
    compressible regimes respectively; see subsection~\ref{sec:ap}
    below.
  \item [(ii)] Klein chooses
    \begin{equation}
      p_{\infty}(t)
      =
      \frac1{\vert\Omega\vert}\int_{\Omega}p(\mbf{x},t)d\mbf{x}.
      \label{eq:p_infty_klein}
    \end{equation}
    The reference pressure $p_\infty$ is related to his so-called
    `nonlocal' pressure $p_{\NL}$ via
    \begin{align}
      p_{\NL}(t)
      &=
      (1-\veps^2){p}_{\infty}(t) \label{pNL2}.
    \end{align}
  \item [(iii)] For future reference we note that
    \begin{align}
      p-\Pi
      =
      (1-\veps^2) (p-p_\infty)
      \label{eq:p-Pi}
    \end{align}
    and hence
    \begin{equation}
      \label{eq:tilde_F}
      \tilde{F}(U)
      =
      (1-\veps^2) \,
      \begin{pmatrix}
        0 \\[0.5ex]
        \frac{p}{\veps^2} \, \Id \\[0.5ex]
        (p-p_{\infty})\ub
      \end{pmatrix}.
    \end{equation}
  \end{enumerate}
\end{remark}

The eigenvalues of the Jacobian of $\hat{F}$, e.g.\ in 1-D case,
are
\begin{equation}
  \label{eq:hat_eig}
  \hat{\lambda}_1=u-c^{*}, \ \hat{\lambda}_2=u, \ \hat{\lambda}_3=u+c^{*},
\end{equation}
where the `pseudo' sound speed is
\begin{equation}
  \label{eq:cstar}
  c^{*}=\sqrt{\frac{p+(\gamma-1)\Pi}{\rho}}.
\end{equation}
Similarly, the eigenvalues of $\tilde{F}$ are
\begin{equation}
  \label{eq:til_eig}
  \tilde{\lambda}_1=-\frac{(1-\veps^2)}{\veps}\sqrt{\frac{(\gamma-1)
      (p-p_{\infty})}{\rho}}, \ \tilde{\lambda}_2=0, \
  \tilde{\lambda}_3=\frac{(1-\veps^2)}{\veps}\sqrt{\frac{(\gamma-1)
      (p-p_{\infty})}{\rho}}.
\end{equation}
Note that the eigenvalues of $\hat{F}$ are $\mathcal{O}(1)$,
whereas those of $\tilde{F}$ are $\mathcal{O}(\veps^{-1})$.
Therefore, $\tilde{F}$ represents the `stiff' part of the original
flux $F$ and $\hat{F}$ is the corresponding `non-stiff' part.

At this point it might also be noted that the quantity under the
square root in the expression for the eigenvalues in
\eqref{eq:til_eig} need not remain positive, which will destroy
the hyperbolicity of the fast system. The crucial point is how to
gauge, or equilibrate, the auxiliary pressure $\Pi$, which is
defined only up to a constant.  Therefore, instead of taking the
average of the pressure to gauge $\Pi$ as in
\eqref{eq:p_infty_klein}, we propose to use the infimum condition
\begin{equation}
  \label{eq:Pi_infty}
  \inf_{\mbf{x}}\Pi(\mbf{x},t)=\inf_{\mbf{x}}p(\mbf{x},t).
\end{equation}
Using \eqref{eq:pi} and \eqref{pNL2}, this leads to
\begin{equation}
  \label{eq:p_infty_def}
  p_{\infty}(t)=\inf_{\mbf{x}}p(\mbf{x},t)=\inf_{\mbf{x}}\Pi(\mbf{x},t)
\end{equation}
which will ensure \eqref{eq:Pi_infty}, \eqref{eq:Pi_limits} and
that the eigenvalues \eqref{eq:til_eig} are all real. As a
preview, we would like to mention that \eqref{eq:p_infty_def} will
also guarantee the ellipticity of the pressure equation
\eqref{eq:elliptic_pressure} in Section \ref{sec:time-discr}.

\begin{remark}
  To summarise, we describe some advantages of Klein's splitting
  \eqref{eq:F_split}, \eqref{eq:vars}, \eqref{eq:pi} with the new
  reference pressure \eqref{eq:p_infty_def}.
  \begin{enumerate}
  \item [(i)] Under condition \eqref{eq:p_infty_def}, both subsystems are
    hyperbolic systems of conservation laws. Following \cite{haack},
    \eqref{eq:F_split} may be called a hyperbolic splitting.
  \item [(ii)] In the limit $\veps\to1$, the stiff flux $\tilde{F}$
    vanishes identically and the non-stiff flux $\hat{F}$ tends to the
    full Euler flux $F$.
\end{enumerate}
\end{remark}

\subsection{Well Prepared Initial Data}
\label{sec:well_prepared}

In \cite{klainerman-majda}, the authors have observed the
importance of well prepared initial data, e.g.\ divergence-free
velocity fields and spatially homogeneous pressure in passing to
the low Mach number limit. The initial data also plays a crucial
role in the loss of accuracy of a numerical scheme in the low Mach
number regimes; see \cite{dellacherie} for details. In the
numerical algorithm developed in the next section we first solve
the auxiliary system
\begin{equation}
  \label{eq:aux_sys}
  U_t+\nabla\cdot\hat{F}(U)=0
\end{equation}
with data $(\rho,\mbf{u},p)$ at time $t$. Using the asymptotics
presented in \cite{klein,meister} it can easily be seen that the
energy equation of \eqref{eq:aux_sys}, i.e.\
\begin{equation}
  \label{eq:aux_engy}
  (\en)_t+\nabla\cdot((\en+\Pi)\mbf{u})=0
\end{equation}
should lead to the divergence-free condition \eqref{div0}. In
fact, the pressure coming from \eqref{eq:aux_engy}, say
$\hat{p}^{n+1}$, should converge to the spatially homogeneous
pressure $p^{(0)}$. Otherwise, the splitting algorithm creates
spurious initial layers. The following proposition states that for
well prepared initial data (a notion introduced in
\cite{klainerman-majda}), $\nabla\cdot\mbf{u}$ and $\nabla p$ grow
at most linearly in time. This is also confirmed by our numerical
experiments, where we do not observe such spurious initial layers.
\begin{proposition}
  \label{prop:well_prepared_data}
  For a well prepared initial data, i.e.\ $({\rho},{\mbf{u}},{p})$
  with
  \begin{equation}
    \label{eq:aux-ic}
    \nabla\cdot{\mbf{u}}(\mbf{x},t)=0, \ \nabla{p}(\mbf{x},t)=0,
  \end{equation}
  the solution of the auxiliary system \eqref{eq:aux_sys} at time
  $t+\dt$ satisfy
  \begin{align}
    \nabla\cdot{\mbf{u}}(\mbf{x},t+\dt)&=\mathcal{O}(\dt),
    \label{aux-div}\\
    \nabla{p}(\mbf{x},t+\dt)&=\mathcal{O}\left(\dt^2\right).
    \label{aux-p}
  \end{align}
\end{proposition}
\begin{proof}
  Since
  ${\mbf{u}}(\mbf{x},t+\dt)={\mbf{u}}(\mbf{x},t)+\mathcal{O}(\dt)$,
  the relation \eqref{aux-div} follows very easily. To establish
  \eqref{aux-p}, we write the non-conservation form of the energy
  equation in \eqref{eq:aux_sys}, i.e.\
  \begin{equation}
    \label{eq:aux-engy}
    {p}_t+\nabla\cdot({p}{\mbf{u}})+(\gamma-1)\Pi\nabla\cdot\mbf{u}=0.
  \end{equation}
  Therefore,
  \begin{align}
    {p}(\mbf{x},t+\dt)&={p}(\mbf{x},t)+\dt{p}_t(\mbf{x},t)
    +\mathcal{O}\left(\dt^2\right),\nonumber\\
    &={p}(\mbf{x},t)-\dt\left\{
      \nabla\cdot({p}{\mbf{u}})
      +(\gamma-1){\Pi}\nabla\cdot{\mbf{u}}\right\}(\mbf{x},t)
    +\mathcal{O}\left(\dt^2\right).\label{aux_p_update}
  \end{align}
  The terms in the curly braces vanish due the hypotheses
  \eqref{eq:aux-ic} and the relation \eqref{aux-p} follows by taking
  gradient.
\end{proof}



\section{Asymptotic Preserving Time Discretization}
\label{sec:time-discr}

In this section we present the asymptotic preserving time
discretization of the Euler equations
\eqref{eq:euler_system}-\eqref{eq:U-F}. Let
$0=t^0<t^1<\cdots<t^n<\cdots$ be an increasing sequence of times
with uniform timestep $\dt = t^{n+1}-t^n$. We shall denote by
$f^n(\mbf{x})$, the value of any component of the approximate
solution at time $t^n$.

\subsection{A First Order Semi-implicit Scheme }
\label{sec:first-order}

The canonical first order semi-implicit time discretization is
\begin{equation}
  U^{n+1}:= U^n - \dt \, \nabla \cdot \hat{F} \left( U^n \right)
  - \dt \, \nabla \cdot \tilde{F} \left( U^{n+1} \right),
  \label{first_update}
\end{equation}
where the non-stiff fluxes are treated explicitly and the stiff
fluxes implicitly. Equivalently, this may be rewritten as a two
step scheme,
\begin{align}
  \hat{U}&:= U^n - \dt \, \nabla \cdot \hat{F} \left( U^n \right),
  \label{eq:update_Uhat} \\ 
  U^{n+1}
  &=\hat{U}-\dt\, \nabla \cdot \tilde{F} \left( U^{n+1} \right).
 \label{eq:update_Utilde} 
\end{align}
Using \eqref{eq:F_split} we write \eqref{eq:update_Uhat} -
\eqref{eq:update_Utilde} in componentwise form
\mm{
  \rhohat
 &=
   \rho^n - \dt \, \nabla \cdot (\rho \mbf{u})^n
\label{eq:update_rhohat}\\
  \momhat
 &=
  (\rho \mbf{u})^n
  - \dt \, \nabla \cdot (\rho \mbf{u} \otimes \mbf{u} + p \, \Id)^n
\label{eq:update_momhat} \\
  \enhat
 &=
  (\en)^n - \dt \, \nabla \cdot
  \left(
   (\en+\Pi)^n \mbf{u}^{n}
  \right)
\label{eq:update_enhat}. }
and, using  \eqref{eq:p-Pi},
\mm{
  \rho^{n+1}
 &=
  \rhohat
\label{eq:update_mass} \\
  (\rho \mbf{u})^{n+1}
 &=
  \momhat
  - \frac{1-\veps^2}{\veps^2} \, \dt \, \nabla \cdot (p^{n+1}  \, \Id)
\label{eq:update_mom} \\
  (\en)^{n+1}
 &=
  \enhat
  - (1-\veps^2) \, \dt \, \nabla \cdot
  \left(
   (p-p_\infty) \mbf{u}
  \right)^{n+1}.
\label{eq:update_en}}

\subsection{Elliptic Pressure Equation}
\label{sec:elliptic_equation}

In this section we start from the semi-implicit, semi-discrete
scheme \eqref{eq:update_rhohat} - \eqref{eq:update_en} and derive
a nonlinear elliptic equation for the pressure. This is motivated
by \cite{cordier}. We approximate the pressure equation
numerically by an iteration scheme, which solves a linearized
elliptic equation in each step.

Using the notation $\uhat := \momhat / \rhohat $ and dividing
\eqref{eq:update_mom} by $\rhohat$, we obtain
\mm{
  \mbf{u}^{n+1}
 &=
  \uhat
    - \frac{1-\veps^2}{\veps^2} \, \frac{\dt}{\rhohat} \, \nabla p^{n+1}.
}
Plugging this into \eqref{eq:update_en} we obtain
\mm{
  (\en)^{n+1}
 &=
  \enhat
  - (1-\veps^2) \, \dt \, \nabla \cdot
  \left(
   (p-p_\infty)^{n+1} \uhat
  \right)
\nonumber \\ &
  + \frac{(1-\veps^2)^2}{\veps^2} \, \dt^2 \, \nabla \cdot
  \left(
   \frac{(p-p_\infty)^{n+1}}{\rhohat}
   \, \nabla p^{n+1})
  \right).
\label{eq:elliptic_1}}
To eliminate the remaining term containing $\mbf{u}^{n+1}$ from
$\en^{n+1}$, we compute
\mm{
 & \quad
  (\en)^{n+1} - \enhat
\nonumber \\ &=
  \frac{p^{n+1}-\phat}{\gamma-1} + \frac{\veps^2 \rhohat}2
  \left(
   \|\mbf{u}^{n+1}\|^2 - \|\uhat\|^2
  \right)
\nonumber \\ &=
  \frac{p^{n+1}-\phat}{\gamma-1} + \frac{\veps^2 \rhohat}2
  \left(
   - 2 \frac{1-\veps^2}{\veps^2} \, \frac{\dt}{\rhohat}
   \, \uhat \cdot \nabla p^{n+1}
   + \frac{(1-\veps^2)^2}{\veps^4} \, \frac{\dt^2}{\rhohat^2}
   \, \|\nabla p^{n+1}\|^2
  \right)
\nonumber \\ &=
  \frac{p^{n+1}-\phat}{\gamma-1} -
   (1-\veps^2) \, \dt
   \, \uhat \cdot \nabla p^{n+1}
   + \frac{(1-\veps^2)^2}{2\veps^2} \, \frac{\dt^2}{\rhohat}
   \, \|\nabla p^{n+1}\|^2
\label{eq:elliptic_2} }
Using \eqref{eq:elliptic_2}, we can rearrange
\eqref{eq:elliptic_1} and obtain the following
\begin{lemma}\label{lemma:IMEX_FirstOrder}
The first order IMEX time discretization consists of the explicit,
non-stiff system \eqref{eq:update_rhohat} -
\eqref{eq:update_enhat}, the implicit mass and momentum equations
\eqref{eq:update_mass} - \eqref{eq:update_mom}, and the nonlinear
elliptic pressure equation
\mm{
 & \quad
  - \frac{(1-\veps^2)^2}{\veps^2} \, \dt^2 \, \nabla \cdot
  \left(
   \frac{(p-p_\infty)^{n+1}}{\rhohat}
   \, \nabla p^{n+1})
  \right)
  + \frac{p^{n+1}}{\gamma-1}
\nonumber \\ &=
  - \frac{(1-\veps^2)^2}{2\veps^2} \, \frac{\dt^2}{\rhohat}
  \, \|\nabla p^{n+1}\|^2
  - (1-\veps^2) \, \dt
   (p-p_\infty)^{n+1} \, \nabla \cdot \uhat
  + \frac{\phat}{\gamma-1}
\label{eq:elliptic_pressure}}
\end{lemma}
\begin{remark}
(i) Note that for $\veps=1$, there is no stiff update, since \eqref{eq:update_mass} - \eqref{eq:update_mom} and \eqref{eq:elliptic_pressure} collapse to
$U^{n+1} = \hat U$.

(ii) Ellipticity degenerates at points where $p$ assumes its
infimum.

(iii) The leading order term is nonlinear.

(iv) The term $p_\infty$ is nonlocal.
\end{remark}
Several of these problems can be avoided by linearization. We
approximate the pressure update $p^{n+1}$ by a sequence
$\left(p_k\right)_{k\in\mathbb{N}}$: $p_1 := \phat$, and given
$p_k$, $p = p_{k+1}$ solves the linearized elliptic equation
\mm{
 & \quad
  - \frac{(1-\veps^2)^2}{\veps^2} \, \dt^2 \, \nabla \cdot
  \left(
   \frac{(p-p_\infty)_{k}}{\rhohat}
   \, \nabla p_{k+1}
  \right)
  + \frac{p_{k+1}}{\gamma-1}
\nonumber \\ &=
  - \frac{(1-\veps^2)^2}{2\veps^2} \, \frac{\dt^2}{\rhohat}
  \, \|\nabla p_{k+1}\|^2
  - (1-\veps^2) \, \dt
   (p-p_\infty)_{k} \, \nabla \cdot \uhat
  + \frac{\phat}{\gamma-1}
\label{eq:elliptic_5}}
The leading order term is now linear, and the term $p^\infty_k$ is
now a constant and hence local. Ellipticity still degenerates at
points where $p_k$ assumes its infimum, but we can leave this
slight degeneracy to the linear algebra solver. We also have the
option to enforce uniform ellipticity by lowering $(p_\infty)_{k}$
slightly.

We measure the convergence in the iterative scheme by computing the distance of $p_k$ and $p_{k+1}$ either in the $W^{1,1}$ norm, or in the weighted $H_1$ norm
\mm{
  \|| p_{k+1} \|| :=
  \frac{1-\veps^2}{\veps} \, \dt \,
  \bigg( \int_\Omega
 \left(
  \frac{|p_{k+1}|^2}{\gamma-1} + \frac{(p-p_\infty)_{k}}{\rhohat}
  \, \|\nabla p_{k+1}\|^2
 \right)
  dx \bigg)^{1/2}
\label{eq:elliptic_6}}
This norm arises naturally when trying to prove that the iterative
scheme is a contraction. We display the contraction constants of
the iteration in the numerical experiments in
Section~\ref{sec:numerics}.



\subsection{Asymptotic Preserving Property}
\label{sec:ap}
We now show that the scheme
\eqref{eq:update_rhohat}-\eqref{eq:update_mom},
\eqref{eq:elliptic_pressure} possesses the AP property, in the
sense that it leads to a discrete version of the limit equations
\eqref{mass0}-\eqref{div0} as $\veps\to0$. The proof of the AP
property uses an asymptotic analysis as in the continuous case;
see also \cite{cordier,degond}. Let us consider the asymptotic
expansions
\begin{align}
  \rho^n(\mbf{x})
 &=
  \rho^{n,(0)}(\mbf{x})+\veps\rho^{n,(1)}(\mbf{x})
  +\veps^2\rho^{n,(2)}(\mbf{x})
  + {\mathcal O} (\veps^3),
\label{rho_asym} \\
  \mbf{u}^n(\mbf{x})
 &=
  \mbf{u}^{n,(0)}(\mbf{x})+\veps\mbf{u}^{n,(1)}(\mbf{x})
  +\veps^2\mbf{u}^{n,(2)}(\mbf{x})
  + {\mathcal O} (\veps^3),
\label{u_asym} \\
  p^n(\mbf{x})
  &=
  p^{n,(0)}(\mbf{x})+\veps p^{n,(1)}(\mbf{x})
  +\veps^2p^{n,(2)}(\mbf{x})
  + {\mathcal O} (\veps^3).
\label{p_asym}
\end{align}

The total energy at time $t^n$ can be expanded as
\begin{align}
  (\en)^n
  &=
  \frac{p^n}{\gamma - 1} + \frac{\veps^2 \rho^n}2 |u^n|^2
  \\&=
  \frac {p^{n,(0)}} {\gamma - 1}
  + \veps \frac {p^{n,(1)}} {\gamma - 1}
  + \veps^2
  \left(
    \frac {p^{n,(2)}} {\gamma - 1}
    + \frac {\rho^{n,(0)}} 2 |u^{n,(0)}|^2
  \right)
  + {\mathcal O} (\veps^3)
  \\&=: (\en)^{n,(0)} + \veps (\en)^{n,(1)} + \veps^2 (\en)^{n,(2)}
  + {\mathcal O} (\veps^3)
  \label{eq:en_expansion}
\end{align}

The auxiliary pressure $\Pi^n$ has to satisfy
\begin{align}
  \Pi^n
  &=
  p_\infty^n + \veps^2 ( p^n - p_\infty^n )
  \\&=
  p_\infty^{n,(0)}
  + \veps p_\infty^{n,(1)}
  + \veps^2
  \left(
    p_\infty^{n,(2)} + p^{n,(0)} - p_\infty^{n,(0)}
  \right)
  + {\mathcal O} (\veps^3)
  \\&=:
  \Pi^{n,(0)}
  + \veps \Pi^{n,(1)}
  + \veps^2 \Pi^{n,(2)}
  + {\mathcal O} (\veps^3)
  \label{eq:Pi_expansion}
\end{align}
Hence, the non-stiff flux $\hat F$ (see \eqref{eq:vars}) can be
expanded as
\begin{align}
  \hat F^n
 &=
  \begin{pmatrix}
    \rho^n\mbf{u}^n \\
    \rho^n \mbf{u}^n \otimes \mbf{u}^n +p^n \, \Id \\
    (\rho E^n + \Pi^n) \mbf{u}^n
   \end{pmatrix},
 \\ &=
  \begin{pmatrix}
    \rho^{n,(0)} \mbf{u}^{n,(0)} \\
    \rho^{n,(0)} \mbf{u}^{n,(0)}\otimes\mbf{u}^{n,(0)} + p^{n,(0)} \, \Id \\
    ((\en)^{n,(0)}+\Pi^{n,(0)})\mbf{u}^{n,(0)}
  \end{pmatrix}
  + {\mathcal O} (\veps)
  \\[0.5ex]&=
  \begin{pmatrix}
    \rho^{n,(0)} \mbf{u}^{n,(0)} \\
    \rho^{n,(0)} \mbf{u}^{n,(0)}\otimes\mbf{u}^{n,(0)} + p^{n,(0)} \, \Id \\
    \left(
      \frac {p^{n,(0)}} {\gamma - 1}
      + p_\infty^{n,(0)}
    \right)
    \mbf{u}^{n,(0)}
  \end{pmatrix}
  + {\mathcal O} (\veps)
  \\&=:
  \hat F^{n,(0)}
  + {\mathcal O} (\veps)
  \label{eq:Fhat_expansion}
\end{align}
Analogously, the stiff flux $\tilde F$ can be expanded as
\begin{align}
  \tilde F^{n+1}
  &=
  \begin{pmatrix}
    0 \\[0.5ex]
    \frac{(1 - \veps^2)}{\veps^2} \, p^{n+1} \, \Id \\[0.5ex]
    (1-\veps^2) (p^{n+1}-p_{\infty}^{n+1})\ub^{n+1}
  \end{pmatrix}
  \nonumber\\[0.5ex] &=
  \veps^{-2}
  \begin{pmatrix}
    0 \\
    p^{n+1,(0)} \, \Id \\
    0
  \end{pmatrix}
  + \veps^{-1}
  \begin{pmatrix}
    0 \\
    p^{n+1,(1)} \, \Id \\
    0
  \end{pmatrix}
  \nonumber \\[0.5ex] &
  + \veps^0
  \begin{pmatrix}
    0 \\
    \left(
      p^{n+1,(2)} \, - p^{n+1,(0)}
    \right)
    \, \Id \\
    (p^{n+1,(0)}-p_\infty^{n+1,(0)})\ub^{n+1}
  \end{pmatrix}
  + {\mathcal O} (\veps)
  \nonumber \\[0.5ex] &=:
  \veps^{-2} \tilde F^{n+1,(-2)}
  + \veps^{-1} \tilde F^{n+1,(-1)}
  + \tilde F^{n+1,(0)}
  + {\mathcal O} (\veps).
  \label{eq:Ftilde_expansion}
\end{align}
We summarize this in the following lemma:
\begin{lemma}
As $\veps\to0$, the scheme
\eqref{eq:update_rhohat}-\eqref{eq:update_en} is consistent with
\begin{align}
  \begin{pmatrix}
  \rho^{n+1,(0)} \\
  (\rho \mbf{u})^{n+1,(0)} \\
  \frac{p^{n+1,(0)}}{\gamma - 1}
 \end{pmatrix}
 &=
  \begin{pmatrix}
  \rho^{n,(0)} \\
  (\rho \mbf{u})^{n,(0)} \\
  \frac{p^{n,(0)}}{\gamma - 1}
 \end{pmatrix}
  - \dt \nabla \cdot \left\{
  \begin{pmatrix}
    \rho^{n,(0)} \mbf{u}^{n,(0)} \\
    \rho^{n,(0)} \mbf{u}^{n,(0)}\otimes\mbf{u}^{n,(0)} + p^{n,(0)} \, \Id \\
    \left(
      \frac {p^{n,(0)}} {\gamma - 1}
      + p_\infty^{n,(0)}
    \right)
    \mbf{u}^{n,(0)}
  \end{pmatrix}
  \right.
 \nonumber \\ & \quad \left. +
  \veps^{-2}
  \begin{pmatrix}
    0 \\
    p^{n+1,(0)} \, \Id \\
    0
  \end{pmatrix}
  + \veps^{-1}
  \begin{pmatrix}
    0 \\
    p^{n+1,(1)} \, \Id \\
    0
  \end{pmatrix} \right.
  \nonumber \\[0.5ex] &
  + \left.
  \begin{pmatrix}
    0 \\
    \left(
      p^{n+1,(2)} \, - p^{n+1,(0)}
    \right)
    \, \Id \\
    (p^{n+1,(0)}-p_\infty^{n+1,(0)})\ub^{n+1,(0)}
  \end{pmatrix}
  \right\} +
   {\mathcal O} (\veps).
 \label{eq:ap_1}
\end{align}
\end{lemma}
Now we assemble the asymptotic numerical scheme. To the two
leading orders the expansion of the momentum equation yields
spatially constant pressures
\begin{align}
 p^{n+1,(0)}(x) \equiv p^{n+1,(0)}, \quad
 p^{n+1,(1)}(x) \equiv p^{n+1,(1)}.
\end{align}
As in \cite{klein}, we absorb $\veps p^{n+1,(1)}$ into
$p^{n+1,(0)}$ by assuming that $ p^{n+1,(1)} \equiv 0 $. With this
the expansion of the reference pressure $p_\infty$ simplifies,
\mm{
  p_{\infty}^{n+1}
 &= \inf_{\mbf{x}}p^{n+1}(\mbf{x})
 = p^{n+1,(0)}
  + \veps^2 \inf_{\mbf{x}} p^{n+1,(2)}(x)
  + {\mathcal O}(\veps^3)
\nonumber \\ &=  
  p^{n+1,(0)} + \veps^2 p_{\infty}^{n+1,(2)}(x)
  + {\mathcal O}(\veps^3),
\label{eq:p_infty_sect3}}
and the leading order auxiliary pressure $\Pi$  and energy are constant and given by
\begin{align}
 \Pi^{n+1,(0)} &= p_\infty^{n+1,(0)}
 \\
  \label{eq:rhoE_Pi}
  (\en+\Pi)^{n+1,(0)}& =\frac{\gamma}{\gamma-1} \,
  p^{n+1,(0)}_\infty.
\end{align}
In particular, the divergence of these terms vanishes in the other
equations. Now we assemble the ${\mathcal O} (\veps^0)$ terms of
expansion \eqref{eq:ap_1}:

\begin{lemma}
(i) The leading order terms mass update is given by
\begin{align}
 \rho^{n+1,(0)}
&=
 \rho^{n,(0)}
 - \dt \nabla \cdot \left(\rho^{n,(0)}\mbf{u}^{n,(0)}\right)
\label{mass_d0}
\end{align}
which is a consistent first order time discretization of the mass
equation \eqref{mass0} in the incompressible limit system.

(ii) The leading order momentum update is given by
\begin{align}
 (\rho\mbf{u})^{n+1,(0)}
&=
 (\rho\mbf{u})^{n,(0)}
 - \dt \nabla \cdot
 \left(
 (\rho \mbf{u} \otimes \mbf{u})^{n,(0)} + p^{n+1,(2)} \, \Id
 \right),
\label{mom_d0}
\end{align}
which is consistent with \eqref{mom0}.
\end{lemma}

Next, we study the elliptic pressure equation
\eqref{eq:elliptic_pressure}, which we slightly rearrange as
follows:
\mm{
   & \quad
   \frac1{\gamma-1} \frac{p^{n+1}-p^n}{\dt}
 =
  \frac1{\gamma-1} \frac{\phat-p^n}{\dt}
  - (1-\veps^2) \, (p-p_\infty)^{n+1} \, \nabla \cdot \uhat
\nn &
  + \frac{(1-\veps^2)^2}{\veps^2} \, \dt \, \nabla \cdot
  \left(
   \frac{(p-p_\infty)^{n+1}}{\rhohat}
   \, \nabla p^{n+1}
  \right)
  - \frac{(1-\veps^2)^2}{2\veps^2} \, \frac{\dt}{\rhohat}
  \, \|\nabla p^{n+1}\|^2
 \label{eq:ap_2}
}
The pressure differences on the RHS of \eqref{eq:ap_2} are
\mm{
  (p-p_\infty)^{n+1}
 & = \veps^2 p^{n+1,(2)} + {\mathcal O}(\veps^3)
\label{eq:ap_2a} \\
  \nabla p^{n+1}
 & = \veps^2 \nabla p^{n+1,(2)}
 + {\mathcal O}(\veps^3)
\label{eq:ap_2b}}
Therefore, to leading order \eqref{eq:ap_2} becomes
\mm{
  \frac{p^{n+1,(0)}-p^{n,(0)}}{(\gamma-1)\dt}
 &= \frac{\phat^{(0)}-p^{n,(0)}}{(\gamma-1)\dt}
 \label{eq:ap_3}
}
It remains to expand $\rhohat$, $\phat$ and $\uhat$:
\mm{
  \rhohat^{(0)}
 &= \rho^{n,(0)} - \dt \nabla\cdot
  (\rho^{n,(0)}\ub^{n,(0)})
 \nn
  \uhat^{(0)} &= \ub^{n,(0)} - \dt \nabla\cdot
  (\rho^{n,(0)} \ub^{n,(0)} \otimes \ub^{n,(0)}
   + p^{n,(0)} \, \Id)
}
so
\mm{
  \uhat^{(0)} - \frac{\dt}{\rhohat}\,\nabla p^{n+1,(2)}
 &=
  \ub^{n,(0)} - \dt \nabla\cdot (\rho \ub \otimes \ub)^{n,(0)}
   - \frac{\dt\,\nabla p^{n+1,(2)}}
    {\rho^{n,(0)} - \dt \nabla\cdot
  (\rho\ub)^{n,(0)}}.
}
\mm{
  \frac{\phat^{(0)}-p^{n,(0)}}{(\gamma - 1)\dt}
 &=
  - \frac {\gamma}{\gamma - 1} p_\infty^{n,(0)}
  \, \nabla \cdot \mbf{u}^{n,(0)}
}
Plugging this into \eqref{eq:ap_3} we immediately obtain the
following
\begin{lemma}\label{lemma:ap_pressure}
To leading order, the semi-discrete elliptic pressure equation
\eqref{eq:elliptic_pressure} is given by

\mm{
  \frac{p^{n+1,(0)}-p^{n,(0)}}{\dt}
 &=
  - \gamma \, p^{n,(0)}
  \nabla \cdot \mbf{u}^{n,(0)},
 \label{eq:ap_5}
}
which is a consistent discretization of the divergence
constraint \eqref{div0}.
\end{lemma}
%
%
%

Integrating \eqref{eq:ap_5} over the space domain $\Omega$ yields
\begin{equation}
  \label{eq:div_thm}
  \frac{p^{n+1,(0)}-p^{n,(0)}}{\gamma \, p^{n,(0)} \, \dt}
  =
  - \int_{\D\Omega} \mbf{u}^{n,(0)} \cdot \mbf{n} \, d\sigma.
\end{equation}
Under several reasonable boundary conditions, such as periodic,
wall, open, etc.\ the integral on the right hand side vanishes;
see \cite{haack}. This yields $p^{n+1,(0)}=p^{n,(0)}$, i.e.\
$p^{n,(0)}$ is a constant in space and time. This in turn leads to
the pointwise divergence constraint in incompressible flows.

Together, this yields the following theorem:
\begin{theorem}
\label{thm:ap}
  The time-discrete scheme
  \eqref{eq:update_rhohat}-\eqref{eq:update_mom},
\eqref{eq:elliptic_pressure} is asymptotic
  preserving in the following sense: the leading order asymptotic
  expansion of the numerical solution is a consistent approximation of the incompressible Euler
  equations \eqref{mass0}-\eqref{div0}.
\end{theorem}



\subsection{Second Order Extension}
\label{sec:second-order}
In this section we extend the semi-discrete scheme to second order
accuracy in time. Our approach is along the lines of \cite{park},
where the authors design a second order scheme using a combination
of second order Runge-Kutta and Crank-Nicolson time stepping
strategies.

We begin by discretizing \eqref{first_update} and obtain the
semi-discrete scheme
\begin{align}
  U^{n+\frac{1}{2}}&=U^n-\frac{\dt}{2}{\nabla\cdot\hat{F}(U^n)
    -\frac{\dt}{2}\nabla\cdot\tilde{F}\left(U^{n+\frac{1}{2}}\right)},
  \label{eq:step1}\\
  U^{n+1}&=U^n-\dt\nabla\cdot\hat{F}\left(U^{n+\frac{1}{2}}\right)
  -\frac{\dt}{2}\nabla\cdot\left(\tilde{F}\left(U^n\right)
    +\tilde{F}\left(U^{n+1}\right)\right). \label{eq:step2}
\end{align}
We notice that in the first timestep \eqref{eq:step1}, the
non-stiff flux $\hat{F}$ is treated explicitly and the stiff flux
$\tilde{F}$ is treated implicitly. In the second timestep
\eqref{eq:step2}, the non-stiff flux is treated by the midpoint
rule and the stiff flux by the trapezoidal or Crank-Nicolson rule.

We proceed as follows. First,
the predictor step is carried out exactly as in the first order
scheme, i.e.\
\begin{align}
  \rho^{n+\frac{1}{2}} & = \rho^n-\frac{\dt}{2}\nabla\cdot(\rho\mbf{u})^n,
  \label{mass_update2}\\
  (\rho\mbf{u})^{n+\frac{1}{2}} & =(\rho\mbf{u})^n
  -\frac{\dt}{2}\nabla\cdot(\rho\mbf{u}\otimes\mbf{u}+p \, \Id)^n
  -\frac{\dt}{2}\frac{(1-\veps^2)}{\veps^2}\nabla
  p^{n+\frac{1}{2}},
  \label{mom_update2}\\
  (\en)^{n+\frac{1}{2}} & = (\en)^n-\frac{\dt}{2}\nabla\cdot\left((\en+\Pi)^n
    \mbf{u}^{n}\right)
  -\frac{\dt}{2}(1-\veps^2)\nabla\cdot((p-p_\infty)\mbf{u})^{n+\frac{1}{2}}
   \label{energy_update2}.
\end{align}

The corrector step is,
\begin{align}
  \rhohat &= \rho^n-\dt\nabla\cdot(\rho\mbf{u})^{n+\frac{1}{2}}, \\
  \uhat   & = \frac{1}{\rhohat} \left\{
   (\rho\mbf{u})^n-\dt\nabla\cdot(\rho\mbf{u}\otimes\mbf{u}
  +p \Id)^{n+\frac{1}{2}} \right\}
\\
  \enhat  & = (\en)^n-\dt\nabla\cdot\left\{(\en+\Pi)^{n+\frac{1}{2}}
    \mbf{u}^{n+\frac{1}{2}}\right\}\\
  \rho^{n+1}&= \rhohat,
  \label{mass_update3} \\
  (\rho\mbf{u})^{n+1}&= \rhohat\uhat-\frac{\dt}{2}\frac{(1-\veps^2)}{\veps^2}\nabla
  \left(p^n+p^{n+1}\right),
  \label{mom_update3}\\
  (\en)^{n+1}&= \enhat -\frac{\dt}{2}(1-\veps^2)\nabla\cdot
  \left\{((p-p_\infty)\mbf{u})^n+((p-p_\infty)\mbf{u})^{n+1}\right\}.
  \label{energy_update3}
\end{align}

As in Section \ref{sec:elliptic_equation}, we obtain the pressure equation for the corrector step. We divide the momentum equation (\ref{mom_update3}) by the explicitly known $\rhohat$ to get the velocity update, that we plug in the energy equation (\ref{energy_update3}). Rewriting $(\en)^{n+1}$ with the state equation (\ref{eq:state_ndim}) we get
\begin{align}
  &\quad
   \frac{p^{n+1}}{\gamma-1}
\nonumber \\ & =
  \enhat - \frac{\veps^2}{2} \rho^{n+1} \norm{\mbf{u}^{n+1}}^2
   -\frac{\dt}{2}(1-\veps^2)\nabla\cdot
     \left\{((p-p_\infty)\mbf{u})^n+((p-p_\infty)\mbf{u})^{n+1}\right\}
\nonumber \\
  \label{eq:ell_sec}
  & = \enhat
\nn &-
  \frac{\veps^2}{2} \rho^{n+1} \bigg\{
    \norm{\uhat}^2 - \frac{\dt}{\rhohat}
    \frac{1-\veps^2}{\veps^2} \uhat \cdot \nabla (p^n + p^{n+1})
   + \frac{\dt^2}{4 \rhohat^2} \frac{(1-\veps^2)^2}{\veps^4}
    \norm{\nabla (p^n + p^{n+1})}^2\bigg\}
  \\ & \quad
    - \frac{\dt}{2}(1-\veps^2)\nabla\cdot \bigg\{((p-p_\infty)\mbf{u})^n
   +(p-p_\infty)\left( \uhat
    - \frac{\dt}{2 \rhohat} \frac{1-\veps^2}{\veps^2}\nabla (p^n + p^{n+1})
   \right) \bigg\} \nonumber,
\end{align}
that is equivalent to
\begin{align}
 &\quad
   \frac{p^n+p^{n+1}}{\gamma-1} - \left( \frac{\dt (1-\veps^2)}{2 \veps} \right)^2
   \nabla\cdot
   \Big(\frac{(p-p_{\infty})^{n+1}}{\rhohat} \nabla (p^n+p^{n+1})\Big)
\nonumber \\
  &= \frac{\phat + p^n}{\gamma-1} - \left(\frac{\dt (1-\veps^2)}{2\veps}\right)^2
     \frac{1}{2\rhohat}\norm{\nabla(p^n+p^{n+1})}^2 \label{eq:p_sec}
 \\&\quad
  - \frac{\dt}{2}(1-\veps^2) \left\{(\mbf{u}^n-\uhat)\cdot\nabla p^n
   + (p-p_{\infty})^n \nabla \cdot \mbf{u}^n
    + (p-p_{\infty})^{n+1} \nabla\cdot \uhat \right\}
 \nonumber.
\end{align}

The derived pressure equation (\ref{eq:p_sec}) is solved by the fixed point iteration
\begin{align}
    & \quad \frac{p^n+p_{k+1}}{\gamma-1} - \left( \frac{\dt (1-\veps^2)}{2 \veps} \right)^2 \nabla\cdot
    \left(\frac{(p-p_{\infty})_{k}}{\rhohat} \nabla (p^n+p_{k+1})\right) \nonumber \\
    &= \frac{\phat+p^n}{\gamma-1} - \left(\frac{\dt (1-\veps^2)}{2\veps}\right)^2 \frac{1}{2\rhohat}\norm{\nabla(p^n+p_k)}^2 \label{eq:fp_sec}\\
    &- \frac{\dt}{2}(1-\veps^2) \left\{ (\mbf{u}^n-\uhat)\cdot\nabla p^n + (p-p_{\infty})^n \nabla \cdot \mbf{u}^n + (p-p_{\infty})_k \nabla\cdot \uhat \right\} \nonumber
\end{align}
with initial value
\begin{equation}
    p_0:=\phat = (\gamma-1)(\enhat - \frac{\veps^2}{2}\rhohat\norm{\uhat}^2)
\end{equation}

Analogously to Section \ref{sec:ap} we show the asymptotic preserving property of the second order scheme (\ref{eq:step1}), (\ref{eq:step2}). 
Let us begin with the equations (\ref{mass_update3}),(\ref{mom_update3})
\begin{align}
 \label{eq:ap21}
   \rho^{n+1,(0)}           
& =
   \rhohat^{(0)} = \rho^{n,(0)} - \dt \nabla \cdot (\rho \mbf{u})^{n+\frac{1}{2},(0)},
 \\ \label{eq:ap22}
   (\rho \mbf{u})^{n+1,(0)}
  &=
   \rhohat^{(0)} \uhat^{0} - \frac{\dt}{2} \nabla (p^{n,(2)} + p^{n+1,(2)}) 
 \nn &
   = (\rho \mbf{u})^{n,(0)} - \dt \nabla \cdot (\rho \mbf{u} \otimes \mbf{u})^{n+\frac{1}{2},(0)} 
                            - \frac{\dt}{2} \nabla (p^{n,(2)} + p^{n+1,(2)}),
\end{align}
where we used $\nabla p^{k,(0)}, \nabla p^{k,(1)} = 0$ for $k = n, n + 1/2, n+1$.
The elliptic equation (\ref{eq:ell_sec}) combined with the leading order state equation $p^{(0)} = (\gamma-1) (\rho E)^{(0)}$ and $\Pi^{0}=p^{0}$ leads to 
\begin{equation} \label{eq:ap23}
 p^{n+1,(0)} = \phat^{(0)} = p^{n,(0)} - \gamma p^{n+\frac{1}{2},(0)}\dt \nabla \cdot \mbf{u}^{n+\frac{1}{2},(0)}.
\end{equation}
for $\veps \rightarrow 0$. The equations (\ref{eq:ap21})-(\ref{eq:ap23}) are second order approximations to the zero Mach number problem 
(\ref{mass0})-(\ref{div0}) and the corrector step is asymptotic preserving by Theorem \ref{thm:ap}. Thus, we proved

\begin{theorem}\label{thm:ap2}
 The time-discrete scheme (\ref{eq:step1}), (\ref{eq:step2}) is asymptotic preserving in the following sense: 
 the leading order asymptotic expansion of the numerical solution is a consistent approximation of the incompressible 
 Euler equations.
\end{theorem}

\begin{remark}
An alternative to the second order Runge-Kutta and Crank-Nicolson time stepping strategies is the implicit-explicit (IMEX) Runge-Kutta schemes \cite{pareschi-russo} or the backward difference
formulae (BDF).  In our recent paper \cite{BispenArunLukacovaNoelle_2014} we have studied different time discretizations for
low Froude number shallow water equations. In particular, we have compared the RK2CN and the
IMEX BDF2 time discretizations from the accuracy, stability and efficiency point of view.
Our extensive numerical tests indicate that both approaches are comparable.
\end{remark}



\subsection{A High Order Stabilization}
\label{sec:stabilization}
In the previous subsections we have introduced the first and
second -  order scheme - (\ref{first_update}) and
(\ref{eq:step1}), (\ref{eq:step2}). Due to semi-implicit nature of
our splitting schemes the following stability condition have to be
satisfied:
\begin{equation}
    \max\left\{\frac{|u_1|+c^*}{\dx}, \frac{|u_2|+c^*}{\dy}\right\}\dt = \nuhat \leq 1,
    \quad \mbf{u} = \left( \begin{array}{c} u_1 \\ u_2 \end{array}\right),
\end{equation}
where $c^*$ is the so-called ``pseudo'' sound speed (\ref{eq:cstar}). However, in our  numerical experiment, e.g.~Section \ref{sec:pulses}, both schemes are unstable for $\nuhat > 0.02$ and $\veps = 0.01$. The reason for the unstable behaviour of scheme in the low Mach number limit is the appearance of the  checkerboard instability, which also strongly influences the convergence of the pressure equation.
The checkerboard instability is a well-known phenomenon arising in the incompressible fluid equations for approximations using
collocated grids and is generated by the decoupling of the spatial approximation. For more details we refer the reader, e.g., to the elaborate description in the book of Ferziger and Peric \cite{FerzigerPeric}.
The simplest approach to filter out the non-physical modes by modifying the discretization
error in the pressure equation is to add fourth order derivatives of the pressure multiplied by a
constant times $\Delta x^2$. We modify this slightly, and introduce the stabilization term
\begin{align}
    \cstab \frac{\dt^4}{\veps^4} \left( \frac{\partial^4 p^{n+q}}{\partial x^4}
     + \frac{\partial^4 p^{n+q}}{\partial y^4} \right)
\label{eq:cstab}
\end{align}
with $q=1/2$ or $q=1$. This is added to the elliptic pressure  equations of the first order scheme in \eqref{eq:p_mod} below, and to the predictor and corrector steps of the second order scheme in \eqref{eq:p_sec1_mod} - \eqref{eq:p_sec2_mod} below. Altogether, we replace the elliptic pressure equation \eqref{eq:elliptic_pressure} of the first order scheme (\ref{first_update}) by the stabilized
pressure equation
\begin{align}
 & \quad \frac{p^{n+1}}{\gamma-1} - \frac{(1-\veps^2)^2}{\veps^2} \, \dt^2 \, \nabla \cdot
  \left(
   \frac{(p-p_\infty)^{n+1}}{\rhohat}
   \, \nabla p^{n+1}
  \right)
 \nn &\quad
  + \frac{\cstab\dt^4}{\veps^4} \left( \frac{\partial^4 p^{n+1}}{\partial x^4} + \frac{\partial^4 p^{n+1}}{\partial y^4}\right)
\nonumber \\[1ex] &=
  - \frac{(1-\veps^2)^2}{2\veps^2} \, \frac{\dt^2}{\rhohat}
  \, \|\nabla p^{n+1}\|^2
  - (1-\veps^2) \, \dt
   (p-p_\infty)^{n+1} \, \nabla \cdot \uhat
  + \frac{\phat}{\gamma-1}.
\label{eq:p_mod}
\end{align}
For the second order scheme, we replace the pressure equation in the predictor step \eqref{eq:p_sec}
by
\begin{align}
    & \quad \frac{p^{n+1/2}}{\gamma-1} - \left( \frac{\dt (1-\veps^2)}{2 \veps} \right)^2
     \nabla\cdot \left(\frac{(p-p_{\infty})^{n+1/2}}{\rhohat}
      \nabla p^{n+1/2}\right)
 \nn &\quad
  + \frac{\cstab\dt^4}{\veps^4} \left(\frac{\partial^4 p^{n+1/2}}{\partial x^4} + \frac{\partial^4 p^{n+1/2}}{\partial y^4} \right)\nonumber \\
    &= \frac{\phat}{\gamma-1} - \left(\frac{\dt (1-\veps^2)}{2\veps}\right)^2 \frac{1}{2\rhohat}\norm{\nabla p^{n+1/2}}^2 - \frac{\dt}{2}(1-\veps^2) (p-p_{\infty})^{n+1/2} \nabla\cdot \uhat,
\label{eq:p_sec1_mod} 
\end{align}
and in the corrector step by
\begin{align}
    & \quad \frac{p^n+p^{n+1}}{\gamma-1}
     - \left( \frac{\dt (1-\veps^2)}{2 \veps} \right)^2 \nabla\cdot
    \left(\frac{(p-p_{\infty})^{n+1}}{\rhohat} \nabla (p^n+p^{n+1})\right)
  \nn &\quad
    + \frac{\cstab\dt^4}{\veps^4} \left(\frac{\partial^4 p^{n+1}}{\partial x^4} + \frac{\partial^4 p^{n+1}}{\partial y^4} \right)
 \nonumber \\ &= 
   \frac{p^n + \phat}{\gamma-1} - \left(\frac{\dt (1-\veps^2)}{2\veps^2}\right)^2 \frac{1}{2\rhohat}\norm{\nabla(p^n+p^{n+1})}^2 \label{eq:p_sec2_mod}
\\ &
   - \frac{\dt}{2}(1-\veps^2) \left\{ 
   (\mbf{u}^n-\uhat)\cdot\nabla p^n + (p-p_{\infty})^n \nabla \cdot \mbf{u}^n + (p-p_{\infty})^{n+1} \nabla\cdot \uhat \right\} \nonumber.
\end{align}

\begin{remark}
\label{remark:cstab}

(i) 
In Lemma~\ref{lemma:ap_stab} and Theorem~\ref{theorem:dtdx23}, we show that the stabilized scheme is asymptotically consistent for desired non-stiff CFL condition $\Delta t = \Delta x$, but only under the restrictive grid condition $\Delta x = O(\veps^{2/3})$. Of course, it would be most desirable to overcome this restriction.

(ii) In all one-dimensional numerical experiments, we set $\cstab = 1/6$ for the first and $\cstab = 1/12$ for the second order scheme. For the two-dimensional experiments, we had to choose substantially higher stabilization parameters, and they are listed in each example.

\end{remark}

According to extensive numerical experiments, the pressure stabilization \eqref{eq:cstab} with a suitable adapted, problem-dependent parameter $\cstab$ stabilizes the implicit velocity pressure decoupling in the low Mach number limit as $p_{\infty}^{n+1,(0)} = p^{n+1,(0)}$,
cf.~\eqref{eq:ap_1}. Hence, the whole scheme remains stable.

The modified fixed point iteration for the first order scheme reads
\begin{align}
 & \quad \frac{p_{k+1}}{\gamma-1} - \frac{(1-\veps^2)^2}{\veps^2} \, \dt^2 \, \nabla \cdot
  \left(
   \frac{(p-p_\infty)_{k}}{\rhohat} \, \nabla p_{k+1} \right)
 \nn &\quad \;
   + \frac{\cstab\dt^4}{\veps^4} \left( \frac{\partial^4 p^{n+1}}{\partial x^4} + \frac{\partial^4 p^{n+1}}{\partial y^4} \right)
\nn &=
  - \frac{(1-\veps^2)^2}{2\veps^2} \, \frac{\dt^2}{\rhohat}
  \, \|\nabla p_{k+1}\|^2
  - (1-\veps^2) \, \dt
   (p-p_\infty)_{k} \, \nabla \cdot \uhat
  + \frac{\phat}{\gamma-1}.
\label{eq:fp_mod}
\end{align}
Analogously, the modified fixed point iterations for the predictor
and corrector second order scheme read
\begin{align}
    & \quad
  \frac{p_{k+1}}{\gamma-1} - \left( \frac{\dt (1-\veps^2)}{2 \veps} \right)^2
   \nabla\cdot \left(\frac{(p-p_{\infty})_{k}}{\rhohat} \nabla p_{k+1}\right)
 \nn &\quad \;
   + \frac{\cstab\dt^4}{\veps^4} \left( \frac{\partial^4 p_{k+1}}{\partial x^4}
     +  \frac{\partial^4 p_{k+1}}{\partial y^4}    \right)
 \nn &=
  \frac{\phat}{\gamma-1} - \left(\frac{\dt (1-\veps^2)}{2\veps}\right)^2 \frac{1}{2\rhohat}\norm{\nabla p_k}^2 - \frac{\dt}{2}(1-\veps^2) (p-p_{\infty})_k \nabla\cdot \uhat \label{eq:fp_sec1_mod},
 \end{align}
 \begin{align}
    & \quad \frac{p^n+p_{k+1}}{\gamma-1} - \left( \frac{\dt (1-\veps^2)}{2 \veps} \right)^2 \nabla\cdot
    \left(\frac{(p-p_{\infty})_{k}}{\rhohat} \nabla (p^n+p_{k+1})\right)
\nn &\quad
    + \frac{\cstab\dt^4}{\veps^4}
    \left( \frac{\partial^4 p_{k+1}}{\partial x^4} 
     +  \frac{\partial^4 p_{k+1}}{\partial y^4} \right)
 \nonumber \\
    &= \frac{p^n + \phat}{\gamma-1} - \left(\frac{\dt (1-\veps^2)}{2\veps^2}\right)^2 \frac{1}{2\rhohat}\norm{\nabla(p^n+p_k)}^2 \label{eq:fp_sec2_mod}\\
    & - \frac{\dt}{2}(1-\veps^2) \left\{ (\mbf{u}^n-\uhat)\cdot\nabla p^n + (p-p_{\infty})^n \nabla \cdot \mbf{u}^n + (p-p_{\infty})_k \nabla\cdot \uhat \right\} \nonumber.
\end{align}

Analogously as in Lemma~\ref{lemma:ap_pressure} we can study the low Mach
number limit $\veps \rightarrow 0$ of the modified pressure
equation (\ref{eq:p_mod}) (divided by $\dt$. Because of the smallness of the pressure terms
derived in \eqref{eq:ap_2a} - \eqref{eq:ap_2b}, it tends towards the discrete energy equation
\eqref{eq:ap_5} plus the stabilization term \eqref{eq:cstab} (again divided by $\dt$),
\mm{
  \frac{p^{n+1,(0)}-p^{n,(0)}}{\dt}
  + \gamma \, p^{n,(0)} \nabla \cdot \mbf{u}^{n,(0)} = 
    \frac{\cstab\dt^3}{\veps^4} \left( \frac{\partial^4 p^{n+1}}{\partial x^4} +  \frac{\partial^4 p^{n+1}}{\partial y^4} \right).
}
It remains show that the stabilization term on the RHS vanishes as $\veps \to 0$. By \eqref{eq:ap_2b}, the pressure
derivatives are $\mathcal{O}(\veps^ 2)$, so the stabilization term is $\mathcal{O}(\dt^3 \veps^{-2})$. The  same holds
for the second order scheme. Consequently, we obtain the following lemma.
\begin{lemma}
\label{lemma:ap_stab}
 If $\dt = o(\veps^{2/3})$ as $\veps \rightarrow 0$, then the first
 and second order scheme with the modified pressure equations
 (\ref{eq:p_mod}), (\ref{eq:p_sec1_mod}), (\ref{eq:p_sec2_mod}) are AP.
\end{lemma}
We discuss the impact of this restriction on the asymptotic behavior of the overall scheme after introducing the fully discrete scheme
in the next section (see Remark \ref{remark:dx_eps23} after  Theorem~\ref{theorem:dtdx23} below).

%
%
 \section{Fully Discrete Scheme}
 \label{sec:discrete}

In order to get a fully discrete scheme we first discretize the
given computational domain which is assumed to a rectangle
$R=[a,b]\times[c,d]$. For simplicity, we consider mesh cells of
equal sizes $\dx$ and $\dy$ in the $x$ and $y$ directions. Let
$C_{i,j}$ be the cell centred around the point $(x_i,y_j)$, i.e.\
\begin{equation}
  \label{eq:c_ij}
  C_{ij}=\left[x_i-\frac{\dx}{2},x_i+\frac{\dx}{2}\right]\times
  \left[y_j-\frac{\dy}{2},y_j+\frac{\dy}{2}\right].
\end{equation}
The conserved variable $U$ is approximated by cell averages,
\begin{equation}
  \label{eq:cell_avg}
  \bar{U}_{i,j}(t)=\frac{1}{\abs{C_{i,j}}}\int_{C_{i,j}}U(x,y,t)dxdy.
\end{equation}
From the given cell averages at time $t^n$, a piecewise linear
interpolant is reconstructed, resulting in
\begin{equation}
  \label{eq:interpolant}
  U^{n}(x,y)=\sum_{i,j}\left(\bar{U}^{n}_{i,j}+U_{i,j}^\prime(x-x_i)
  +U_{i,j}^\backprime(y-y_j)\right)\chi_{i,j}(x,y),
\end{equation}
where $\chi_{i,j}$ is the characteristic function of the cell
$C_{i,j}$ and $U_{i,j}^\prime$ and $U_{i,j}^\backprime$ are
respectively the discrete slopes in the $x$ and $y$ directions. A
possible computation of these slopes, which results in an overall
non-oscillatory scheme is given by the family of nonlinear minmod
limiters parametrised by $\theta\in [1,2]$, i.e.\
\begin{align}
  U_{i,j}^\prime&=MM\left(\theta\frac{\bar{U}^{n}_{i+1,j}-\bar{U}^n_{i,j}}{\dx},
    \frac{\bar{U}^{n}_{i+1,j}-\bar{U}^n_{i-1,j}}{2\dx},
    \theta\frac{\bar{U}^{n}_{i,j}-\bar{U}^n_{i-1,j}}{\dx}\right), \\
  U_{i,j}^\backprime&=MM\left(\theta\frac{\bar{U}^{n}_{i,j+1}-\bar{U}^n_{i,j}}{\dy},
    \frac{\bar{U}^{n}_{i,j+1}-\bar{U}^n_{i,j-1}}{2\dy},
    \theta\frac{\bar{U}^{n}_{i,j}-\bar{U}^n_{i,j-1}}{\dy}\right),
\end{align}
where the minmod function is defined by
\begin{equation}
  \label{eq:mm}
  MM(x_1,x_2,\ldots,x_p)=
  \begin{cases}
    \min_p\{x_p\} & \mbox{if} \ x_p>0 \ \forall p,\\
    \max_p\{x_p\} & \mbox{if} \ x_p<0 \ \forall p,\\
    0 & \mbox{otherwise}.
  \end{cases}
\end{equation}
Recall that the first step of the algorithm consists of computing
the solution of the auxiliary system \eqref{eq:aux_sys}. Since
\eqref{eq:aux_sys} is hyperbolic, we use the finite volume update
\begin{equation}
  \label{eq:fv_update}
  \hat{U}_{i,j}^{n+1}=\bar{U}_{i,j}^n-\frac{\dt}{\dx}
  \left(\hat{{\mathcal{F}_1}}_{i+\frac{1}{2},j}
  -\hat{{\mathcal{F}_1}}_{i-\frac{1}{2},j}\right)
  -\frac{\dt}{\dy}\left(\hat{{\mathcal{F}_2}}_{i,j+\frac{1}{2}}-\hat{{\mathcal{F}_2}}_{i,j-\frac{1}{2}}\right),
\end{equation}
where we choose the Rusanov flux for the interface fluxes
$\hat{\mathcal{F}}_1$ and $\hat{\mathcal{F}}_2$, e.g.\ in the $x$
direction
\begin{align}
  \label{eq:kt}
  &\quad \;
   \hat{{\mathcal{F}_1}}_{i+\frac{1}{2},j}\left(U_{i+\frac{1}{2},j}^+,
    U_{i+\frac{1}{2},j}^-\right)
\nn &=
  \frac{1}{2}\left(\hat{F}_1\left(
      U_{i+\frac{1}{2},j}^+\right)+\hat{F}_1\left(U_{i+\frac{1}{2},j}^-\right)
  \right)
     -\frac{a_{i+\frac{1}{2},j}}{2}\left(U_{i+\frac{1}{2},j}^+
    -U_{i+\frac{1}{2},j}^-\right).
\end{align}
The expression for the numerical flux $\mathcal{F}_2$ in the $y$
direction is analogous. Here, $U_{i+1/2,j}^+$ and $U_{i+1/2,j}^-$
are respectively the right and left interpolated states at a right
hand vertical interface and $a_{i+1/2,j}$ is the maximal
propagation speed given by the (non-stiff) eigenvalues
$\hat\lambda$ of the flux component $\hat{F}_1$ in the $x$
direction. Based on \eqref{eq:hat_eig} we obtain
\begin{equation}
  \label{eq:a_x}
  a_{i+\frac{1}{2},j}=\max\left(\abs{u}_{i+\frac{1}{2},j}^{+}
    +{c^*}_{i+\frac{1}{2},j}^{+},\abs{u}_{i+\frac{1}{2},j}^{-}
    +{c^*}_{i+\frac{1}{2},j}^{-}\right).
\end{equation}
In an analogous way, the numerical fluxes in the $y$ direction
could be assembled.

The timestep $\dt$ is chosen by the non-stiff CFL condition
\begin{equation}
  \label{eq:cfl}
  \dt\max_{i,j}\max\left(\frac{\abs{u}_{i,j}+c_{i,j}^*}{\dx},
    \frac{\abs{v}_{i,j}+c_{i,j}^*}{\dy}\right)=\nuhat
\end{equation}
with $\nuhat$ being the given CFL number. Going back to the
dimensional variables, the CFL conditions reads
\begin{equation}
  \label{eq:cfl1}
  \dt^\prime\max_{i,j}\max\left(\frac{\abs{u^\prime}_{i,j}
  +\frac{c_{i,j}^{\prime,*}}{\veps}}{\dx^\prime},
  \frac{\abs{v^\prime}_{i,j}+\frac{c_{i,j}^{\prime,*}}
  {\veps}}{\dy^\prime}\right)=\nuhat.
\end{equation}
Hence, the effective CFL number $\nu_{eff}\sim\nuhat/\veps$.

The next step consists of solving the linearised elliptic equation
\eqref{eq:fp_mod}, (\ref{eq:fp_sec1_mod}) or (\ref{eq:fp_sec2_mod}) to obtain the pressure $p^{n+1}$. The
second order terms in in the pressure equations are discretized
using compact central differences, e.g.\
\mm{
 &\quad\;
    \left((fg_x)_x\right)_{i,j}
\nn &=
  \frac{1}{\dx}\left\{\left(fg_x\right)_{i+\frac{1}{2},j}
    -\left(fg_x\right)_{i-\frac{1}{2},j}\right\}
\nn &=
  \frac{1}{\dx}\left\{f_{i+\frac{1}{2},j}(g_x)_{i+\frac{1}{2},j}
    -f_{i-\frac{1}{2},j}(g_x)_{i-\frac{1}{2},j}\right\}
\nn &=
  \frac{1}{\dx} \bigg\{
   \frac{f_{i+1,j}+f_{i,j}}{2}
    \frac{g_{i+1,j}-g_{i,j}}{\dx}
     - \frac{f_{i,j}+f_{i-1,j}}{2}
      \frac{g_{i,j}-g_{i-1,j}}{\dx}
   \bigg\}.
}
The second order differences in the $y$ direction are treated
analogously. All the first derivatives terms are also discretized
by simple central differences. Discretizing all the terms, finally we are lead to a linear
system for the pressure $p^{n+1}$ at the new timestep. The
resulting linear system has a simple five-diagonal structure in the
1-D case, whereas it has a band matrix nature in the
multidimensional case. The linear system is solved by means of the
direct solver UMFPACK \cite{umfpack} in all the numerical test
problems reported in this paper.

\subsection{Summary of the Algorithm}
\label{sec:sum_alg}

In the following, we summarise the main steps in the algorithm.
For simplicity, we do it only for the first order case. The second
order scheme follows similar lines, except that it contains two
cycles in one timestep.

First, let us suppose that $(\rho^n,\mbf{u}^n,p^n)$ are the given
initial values.

In step 1 we solve the auxiliary system \eqref{eq:aux_sys}, i.e.\
\begin{equation*}
  \label{eq:aux_sys_rem}
  U_t+\nabla\cdot\hat{F}(U)=0
\end{equation*}
subject to the given initial data to obtain
$(\hat{\rho},\hat{\mbf{u}},\hat{p})$ at the new time $t^{n+1}$.

In a fully compressible problem, i.e.\ when $\veps=1$, we simply
set
$$
 (\rho^{n+1},\mbf{u}^{n+1},p^{n+1})=(\hat{\rho},\hat{\mbf{u}},\hat{p})
$$
and the process continues. Otherwise, we set
$\rho^{n+1}=\hat{\rho}$.

In step 2 we solve the linearised elliptic equation
\eqref{eq:fp_mod}, i.e.~the fix-point iteration\
\begin{align*}
 & \quad \frac{p_{k+1}}{\gamma-1} - \frac{(1-\veps^2)^2}{\veps^2} \, \dt^2 \, \nabla \cdot
  \left(
   \frac{(p-p_\infty)_{k}}{\rhohat}
   \, \nabla p_{k+1}
  \right) + \frac{\dt^4}{6 \veps^4} 
  \left( \frac{\partial^4 p^{n+1}}{\partial x^4} 
   + \frac{\partial^4 p^{n+1}}{\partial y^4} \right)
\nonumber \\ &=
  - \frac{(1-\veps^2)^2}{2\veps^2} \, \frac{\dt^2}{\rhohat}
  \, \|\nabla p_{k+1}\|^2
  - (1-\veps^2) \, \dt
   (p-p_\infty)_{k} \, \nabla \cdot \uhat
  + \frac{\phat}{\gamma-1}
\end{align*}
to get the pressure $p^{n+1}$ at the new time level.

In step 3 we update the velocity $\mbf{u}^{n+1}$ is explicitly
using \eqref{eq:update_mom}, i.e.
\begin{equation*}
    (\rho \mbf{u})^{n+1}
 =
  \momhat
  - \frac{1-\veps^2}{\veps^2} \, \dt \, \nabla \cdot (p^{n+1}  \, \Id)
\end{equation*}
with the aid of $p^{n+1}$ from step 2.

As shown in Lemma~\ref{lemma:ap_stab}, the stabilization term respects the AP property if $\dt = o(\veps^{2/3})$. Combining this with the definition of $\dt$ using the non-stiff CFL condition~\eqref{eq:cfl}, we obtain

\begin{theorem}\label{theorem:dtdx23}
The fully discrete stabilized scheme is AP under a time-step restriction
\mm{
  \dt = \mathcal{O}(\dx)
\label{eq:dt_eps23}  
}
and a spatial resolution
\mm{
  \dx = o(\veps^{2/3}).
\label{eq:dx_eps23}
}
\end{theorem}

\begin{remark}
\label{remark:dx_eps23}

(i)
Condition \eqref{eq:dx_eps23} shows that the AP property may not hold for an underresolved spatial grid. This restriction is due to the stabilization term \eqref{eq:cstab}, and it is an important question whether a smaller term could stabilize the present IMEX scheme.

(ii) For moderately low Mach numbers, \eqref{eq:dx_eps23} is not a severe restriction. For example, for $\varepsilon=10^{-2}$, one needs roughly 20 gridpoints per unit length, and for $\veps=10^{-3}$, roughly 100 points.
 
\end{remark}

\begin{remark}
  The following remarks are in order.
  \begin{enumerate}
  \item [(i)] It has to be noted that throughout this paper we follow
    a non-di\-men\-siona\-lisation in such a way that $0<\veps\leq 1$.
  \item [(ii)] We note that in the fully compressible case, i.e.\ when
    $\veps=1$, the auxiliary system is the Euler system itself and the
    stiff flux $\tilde{F}\equiv 0$. In this case, the algorithm just
    consists of step 1 and the overall scheme simply reduces to a shock
    capturing algorithm.
  \end{enumerate}
\end{remark}

%

\section{Numerical Experiments}
 
\label{sec:numerics}

In order to validate the proposed AP sche\-mes (\ref{first_update}) and (\ref{eq:step1}), (\ref{eq:step2}), in this section we
present the results of numerical experiments. 
First, we observe the convergence of the fixed point iterations (\ref{eq:fp_mod}) and (\ref{eq:fp_sec1_mod}), (\ref{eq:fp_sec2_mod}). As low Mach number tests we consider weakly
compressible problems with $\veps\ll 1$. In particular, we study
the propagation of long wavelength acoustic waves and their interactions
with small scale perturbations. 
Mach number.
Then, we test the performance of the scheme in the compressible regime $\veps=1$ and in this case the scheme should possess shock capturing
features. The numerical results clearly indicate that the scheme
captures discontinuous solutions containing shocks, contacts, etc.\
without any oscillations.
In many of
our numerical studies we have observed
the formation of shocks due to weakly nonlinear effects, albeit a
prescribed value of $\veps$ less than unity. Our numerical results
agree well with the benchmarks reported in the literature.
 
If not stated otherwise, the stabilization parameter in \eqref{eq:cstab}
is set to $\cstab=1/6$ for the first order and $\cstab=1/12$ for the second
order one-dimensional schemes. Note however, that it is problem-dependent and is
substantially larger for the two-dimensional problems.

 \subsection{Test Problems In The 1D Case}
 
 \subsubsection{Two Colliding Acoustic Pulses}
 \label{sec:pulses}
 
 We consider a weakly compressible test problem taken from
 \cite{klein}. The setup consists of two colliding acoustic pulses in a
 weakly compressible regime. The domain is $-L\leq x\leq L=2/\veps$ and
 the initial data are given by
 \begin{align*}
   \rho(x,0)&=\rho_0+\frac{1}{2}\veps\rho_1\left(1-\cos\left(\frac{2\pi
         x}{L}\right)\right),\ \rho_0=0.955,\ \rho_1=2.0,\\
   u(x,0)&=\frac{1}{2}u_0 \ \sign(x)\left(1-\cos\left(\frac{2\pi
         x}{L}\right)\right), \ u_0=2\sqrt{\gamma},\\
   p(x,0)&=p_0+\frac{1}{2}\veps p_1\left(1-\cos\left(\frac{2\pi
         x}{L}\right)\right),\ p_0=1.0,\ p_1=2\gamma.
 \end{align*}
 
 \begin{center} \textbf{Investigation of Fixed Point Iterations} \end{center}
 
 The value of the parameter $\veps$ will be specified later. First, we use this
 test to study the convergence of the fixed point iterations $(p_k)_k$ (\ref{eq:fp_mod}), (\ref{eq:fp_sec1_mod}) (\ref{eq:fp_sec2_mod}). Since an exact solution of the nonlinear pressure equations is not available, we determine the experimental contraction rate (ECR) as
 \begin{equation}
   \label{eq:ECR}
   \mbox{ECR}:=\frac{ \norm{p_{k}-p_{k+1}}}{\norm{p_{k-1}-p_{k}}} \approx \frac{\norm{p_{k}-p}}{\norm{p_{k-1}-p}}.
 \end{equation}
 If ECR$ < 1$ is independent of $k$, it holds
 \begin{equation}
     \norm{p_k - p_{k+1}} = \mbox{ECR}^k \norm{p_1 - p_{0}} \quad \forall \ k \geq 0
 \end{equation}
 and ECR is the contraction constant of the sequence $(p_k)_k$, i.e.~the sequence $(p_k)_k$ tends for $k\rightarrow \infty$ to its limit $p$ and
 \begin{equation}
     \norm{p-p_k} \leq \frac{ECR^k}{1-ECR} \norm{p_1-p_0}.
 \end{equation}
 In all our numerical  tests using various configurations of $\Delta t, \Delta x, \veps$ we have observed $ECR \ll 1$ i.e.~fast convergence. After less then $10$ iterations the convergence stops due to double precision arithmetic. This behaviour is demonstrated in Tables~\ref{tab:ecr_O1_0p1_w} - \ref{tab:ecr_O21_0p01_w}, where each table contains the errors $\norm{p_N-p_{N-1}}$ and ECR numbers, cf.~(\ref{eq:ECR}), for $N=k+1$  of the fixed point iteration during the first and fifth time step. The errors are computed using the discrete Sobolev norm
 \begin{equation}
     \norm{p}_{W^{1,1}}:= \Delta x \sum\limits_i \left\{|p_i|
     + \left| \delta_x p_i \right|\right\}
 \end{equation}
 and the discrete variant of the norm (\ref{eq:elliptic_6})
 \begin{equation}
     \norm{p}_{S^k}:= \frac{1-\veps^2}{\veps} \Delta t 
     \bigg( \Delta x \sum\limits_i \left( \frac{p_i^2}{\gamma-1}
     + \frac{(p_k)_i-(p_k)_{\infty}}{\hat{\rho}_i} \left(\delta_x p_i\right)^2
     \right) \bigg)^{1/2}
 \end{equation}
 Here, $p_i$ denotes the cell average of $p$ on the $i$-th cell, $\delta_x p_i$ is a central finite difference derivative,
 if the index $i$ corresponds to an inner cell, and one-sided finite difference otherwise. Tables~\ref{tab:ecr_O1_0p1_w}-\ref{tab:ecr_O22_0p1_s} present the results obtained for the first and second order scheme (\ref{first_update}) and (\ref{eq:step1}), (\ref{eq:step2}) for $\veps = 0.1$, respectively. Furthermore, the results obtained for  $\veps = 0.01$ are presented in Tables~\ref{tab:ecr_O1_0p01_w}-\ref{tab:ecr_O21_0p01_w}.
 
 We can notice that already after one iteration the error in the discrete pressure equation is  smaller then the local truncation error. Consequently, no significant differences have been observed between the numerical solution obtained by iterating the pressure once or several  times, cf.~Figure~\ref{fig:it_noit}. 
 Therefore, in the following computations we have performed just one nonlinear iteration to solve the pressure equations (\ref{eq:fp_mod}), (\ref{eq:fp_sec1_mod}), (\ref{eq:fp_sec2_mod}) numerically.
 
 \begin{figure}[htbp]
     \includegraphics[height=0.275\textheight,width=0.475\textwidth]{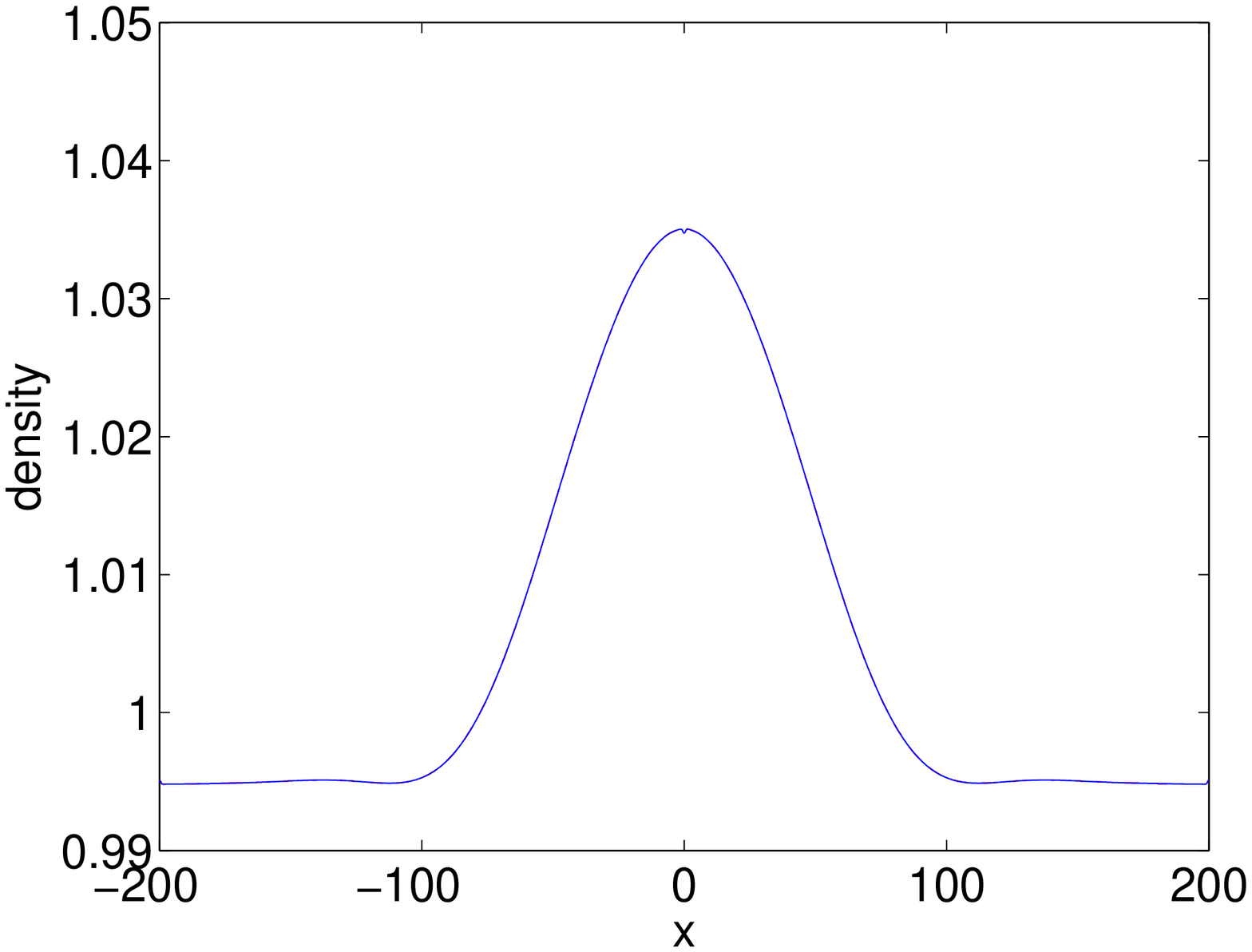}
     \includegraphics[height=0.275\textheight,width=0.475\textwidth]{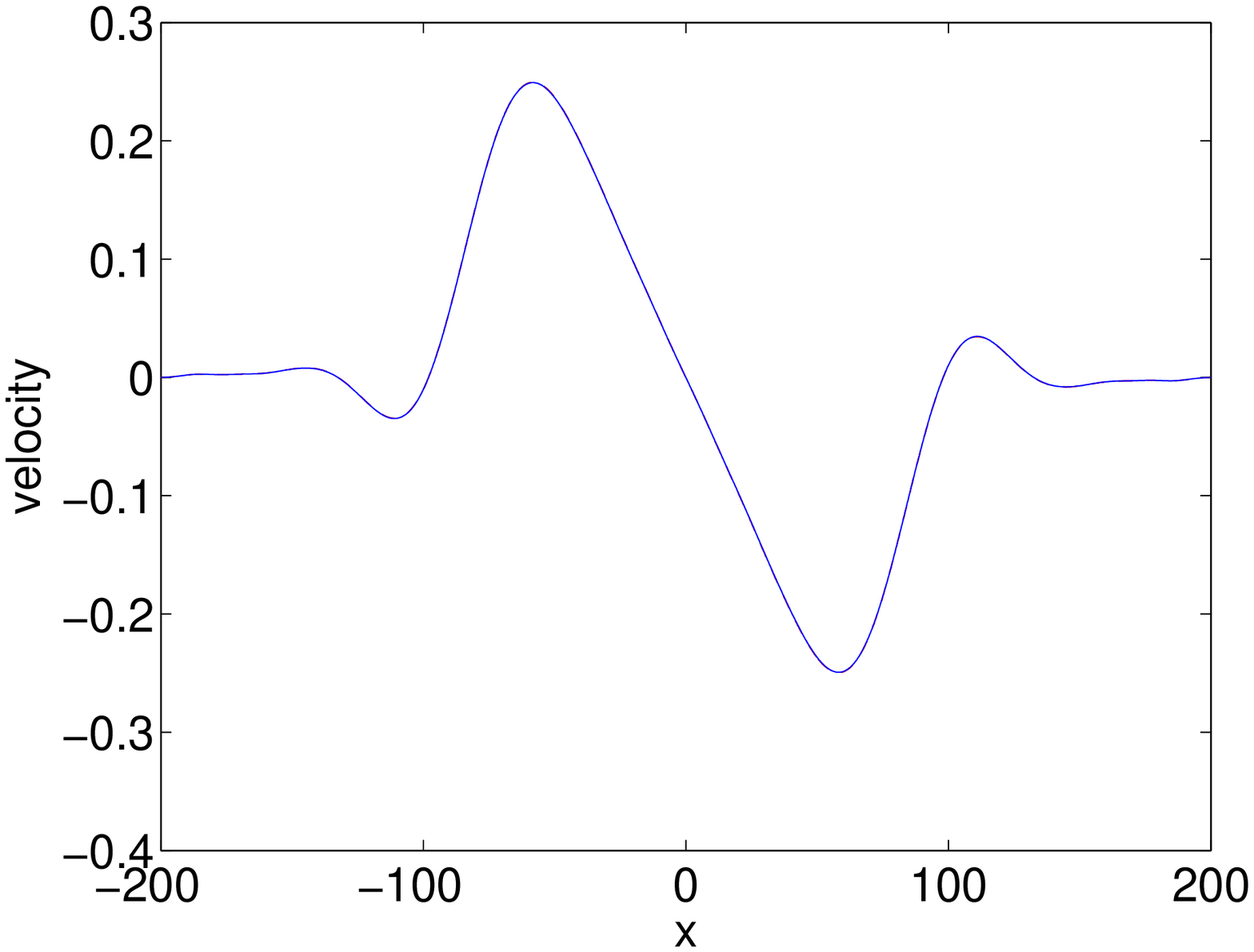}
   \\
     \includegraphics[height=0.275\textheight,width=0.475\textwidth]{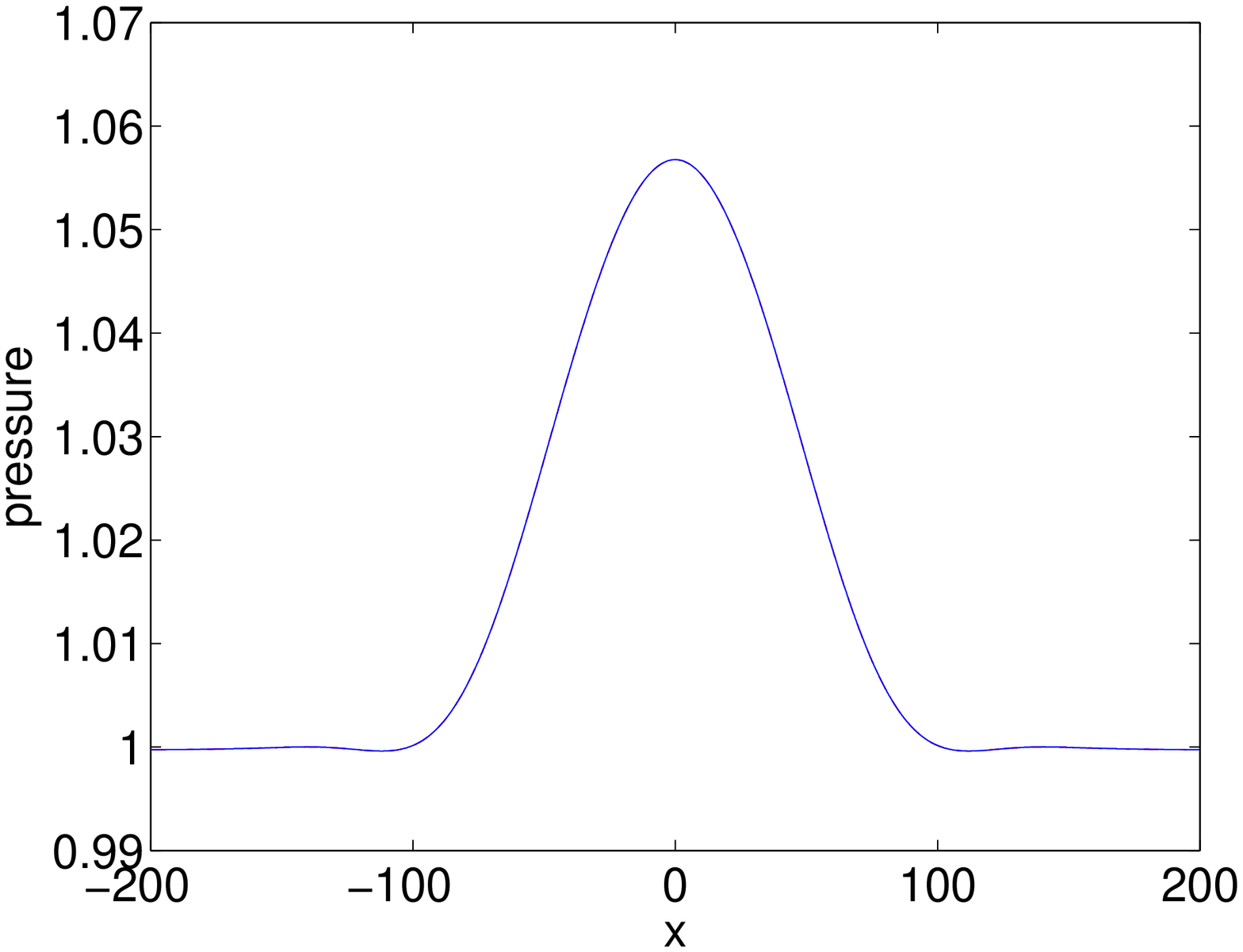}
     \includegraphics[height=0.275\textheight,width=0.475\textwidth]{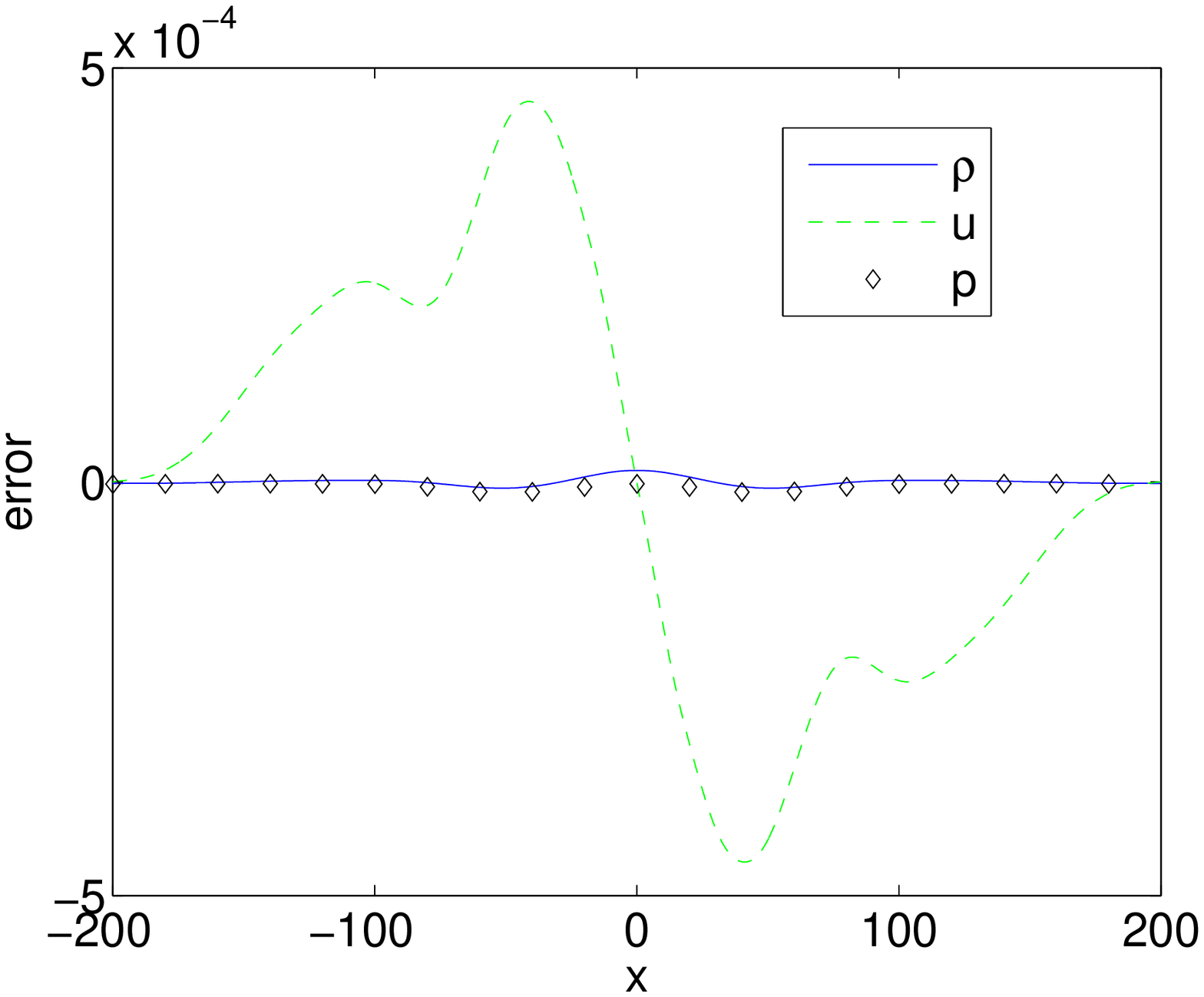}
     \caption{Two colliding acoustic pulses problem. Numerical solution for $\veps=0.01$ obtained by the second order scheme without any limiters, $\nuhat=0.9$, $\nu=90$, $\cstab=1/12$. The dashed lines represent the solution with $20$ iterations and the straight lines show the solution without iterations. Top left: density. Top right: velocity. Bottom left: pressure. Bottom right: difference between density, velocity and pressure.}
     \label{fig:it_noit}
 \end{figure}
 
 \begin{table}
   \begin{center}
     \begin{tabular}{l|ll|ll}
       \multicolumn{5}{c}{} \\ 
     & \multicolumn{2}{c}{1st time step} & \multicolumn{2}{c}{5th time step} \\
       $N$ & $W^{1,1}$-error & \mbox{ECR} & $W^{1,1}$-error & \mbox{ECR} \\
       \hline \\
       1 & 3.60e-04 &    1.000 & 7.91e-04 &    1.000 \\
       2 & 2.58e-06 &    0.007 & 1.32e-05 &    0.017 \\
       3 & 2.07e-08 &    0.008 & 2.47e-07 &    0.019 \\
       4 & 1.71e-10 &    0.008 & 4.88e-09 &    0.020 \\
       5 & 1.45e-12 &    0.009 & 9.83e-11 &    0.020 \\
       6 & 1.34e-13 &    0.092 & 2.12e-12 &    0.022 \\
       7 & 1.07e-13 &    0.799 & 1.55e-13 &    0.073 \\
       8 & 1.02e-13 &    0.955 & 1.48e-13 &    0.956 \\
       9 & 8.65e-14 &    0.848 & 1.48e-13 &    0.999 \\
       10 & 9.13e-14 &    1.055 & 1.32e-13 &    0.894 \\
       11 & 9.90e-14 &    1.084 & 1.31e-13 &    0.989 \\
       12 & 9.56e-14 &    0.966 & 1.31e-13 &    1.003 \\
       13 & 9.39e-14 &    0.982 & 1.40e-13 &    1.066 \\
     \end{tabular}
     \vspace*{1ex}
     \caption{Two colliding acoustic pulses problem.
       $W^{1,1}$-errors for pressure and ECR's of the iteration (\ref{eq:fp_mod}) during the first and 5th
       time step using the first order scheme.
        $\veps=0.1$, $\nuhat = 0.9$, $\nu_{eff} = 9$.}
     \label{tab:ecr_O1_0p1_w}
   \end{center}
 \end{table}

 \begin{table}
   \begin{center}
     \begin{tabular}{l|ll|ll}
       \multicolumn{5}{c}{}\\
        & \multicolumn{2}{c}{1st time step} & \multicolumn{2}{c}{5th time step} \\
       $N$ & $S$-error & \mbox{ECR} & $S$-error & \mbox{ECR} \\
           \hline \\
       1 & 4.86e-05 &    1.000 & 1.92e-04 &    1.000 \\
       2 & 3.65e-07 &    0.008 & 3.68e-06 &    0.019 \\
       3 & 2.99e-09 &    0.008 & 7.37e-08 &    0.020 \\
       4 & 2.50e-11 &    0.008 & 1.49e-09 &    0.020 \\
       5 & 2.10e-13 &    0.008 & 3.05e-11 &    0.020 \\
       6 & 1.03e-14 &    0.049 & 6.23e-13 &    0.020 \\
       7 & 8.04e-15 &    0.784 & 1.69e-14 &    0.027 \\
       8 & 7.80e-15 &    0.970 & 1.29e-14 &    0.765 \\
       9 & 6.87e-15 &    0.881 & 1.38e-14 &    1.071 \\
       10 & 6.95e-15 &    1.011 & 1.13e-14 &    0.821 \\
       11 & 7.96e-15 &    1.145 & 1.05e-14 &    0.924 \\
       12 & 7.80e-15 &    0.980 & 1.09e-14 &    1.035 \\
       13 & 7.49e-15 &    0.961 & 1.23e-14 &    1.134 \\
     \end{tabular}
     \vspace*{1ex}
     \caption{Two colliding acoustic pulses problem.
       $\|\cdot\|_{S}$-errors for pressure and ECR's of the iteration (\ref{eq:fp_mod}) during the first and 5th time step using the first order scheme.
        $\veps=0.1$, $\nuhat = 0.9$, $\nu_{eff} = 9$.}
     \label{tab:ecr_O1_0p1_s}
   \end{center}
 \end{table}
 
 
 \begin{table}
   \begin{center}
     \begin{tabular}{l|ll|ll}
       \multicolumn{5}{c}{}\\
        & \multicolumn{2}{c}{1st time step} & \multicolumn{2}{c}{5th time step} \\
       $N$ & $W^{1,1}$-error & \mbox{ECR} & $W^{1,1}$-error & \mbox{ECR} \\
           \hline \\
        1 & 8.99e-05 &    1.000 & 1.44e-04 &    1.000 \\
       2 & 3.53e-07 &    0.004 & 1.03e-06 &    0.007 \\
       3 & 1.48e-09 &    0.004 & 8.30e-09 &    0.008 \\
       4 & 6.36e-12 &    0.004 & 7.16e-11 &    0.009 \\
       5 & 8.51e-14 &    0.013 & 6.86e-13 &    0.010 \\
       6 & 6.41e-14 &    0.753 & 7.66e-14 &    0.112 \\
       7 & 6.71e-14 &    1.047 & 6.36e-14 &    0.830 \\
       8 & 5.75e-14 &    0.858 & 5.36e-14 &    0.843 \\
       9 & 5.69e-14 &    0.989 & 7.27e-14 &    1.358 \\
       10 & 6.14e-14 &    1.079 & 7.57e-14 &    1.040 \\
       11 & 5.74e-14 &    0.935 & 6.96e-14 &    0.920 \\
       12 & 5.77e-14 &    1.006 & 5.67e-14 &    0.814 \\
       13 & 7.15e-14 &    1.238 & 5.04e-14 &    0.890 \\
     \end{tabular}
     \vspace*{1ex}
     \caption{Two colliding acoustic pulses problem.
       $\|\cdot\|_{W^{1,1}}$-errors for pressure and ECR's of the iteration (\ref{eq:fp_sec1_mod}) during the first and 5th  time step using the second order scheme - first step (\ref{eq:step1}).
        $\veps=0.1$, $\nuhat = 0.9$, $\nu_{eff} = 9$.}
     \label{tab:ecr_O21_0p1_w}
   \end{center}
 \end{table}
 
 \begin{table}
   \begin{center}
     \begin{tabular}{l|ll|ll}
       \multicolumn{5}{c}{}\\ 
        & \multicolumn{2}{c}{1st time step} & \multicolumn{2}{c}{5th time step} \\
       $N$ & $W^{1,1}$-error & \mbox{ECR} & $W^{1,1}$-error & \mbox{ECR} \\
           \hline \\
       1 & 2.63e-03 &    1.000 & 3.88e-03 &    1.000 \\
       2 & 1.09e-05 &    0.004 & 3.00e-05 &    0.008 \\
       3 & 4.64e-08 &    0.004 & 2.66e-07 &    0.009 \\
       4 & 2.01e-10 &    0.004 & 2.52e-09 &    0.009 \\
       5 & 9.19e-13 &    0.005 & 2.46e-11 &    0.010 \\
       6 & 1.41e-13 &    0.154 & 3.73e-13 &    0.015 \\
       7 & 1.46e-13 &    1.034 & 1.44e-13 &    0.387 \\
       8 & 1.46e-13 &    1.002 & 1.56e-13 &    1.079 \\
       9 & 1.35e-13 &    0.921 & 1.50e-13 &    0.963 \\
       10 & 1.13e-13 &    0.840 & 1.39e-13 &    0.930 \\
       11 & 1.16e-13 &    1.022 & 1.42e-13 &    1.018 \\
       12 & 1.23e-13 &    1.058 & 1.66e-13 &    1.173 \\
       13 & 1.50e-13 &    1.223 & 1.35e-13 &    0.811 \\
     \end{tabular}
     \vspace*{1ex}
     \caption{Two colliding acoustic pulses problem.
       $\|\cdot\|_{W^{1,1}}$-errors for pressure and ECR's of the iteration (\ref{eq:fp_sec2_mod}) during the first and 5th time step using the second order scheme - 2nd step.
        $\veps=0.1$, $\nuhat = 0.9$, $\nu_{eff} = 9$.}
     \label{tab:ecr_O2w_0p1_w}
   \end{center}
 \end{table}
 \begin{table}
   \begin{center}
     \begin{tabular}{l|ll|ll}
       \multicolumn{5}{c}{}\\ 
        & \multicolumn{2}{c}{1st time step} & \multicolumn{2}{c}{5th time step} \\
       $N$ & $S$-error & \mbox{ECR} & $S$-error & \mbox{ECR} \\
           \hline \\
     1 & 1.83e-04 &    1.000 & 4.37e-04 &    1.000 \\
       2 & 7.71e-07 &    0.004 & 4.03e-06 &    0.009 \\
       3 & 3.33e-09 &    0.004 & 3.89e-08 &    0.010 \\
       4 & 1.46e-11 &    0.004 & 3.82e-10 &    0.010 \\
       5 & 6.44e-14 &    0.004 & 3.78e-12 &    0.010 \\
       6 & 5.38e-15 &    0.083 & 3.91e-14 &    0.010 \\
       7 & 5.59e-15 &    1.040 & 5.56e-15 &    0.142 \\
       8 & 5.14e-15 &    0.920 & 5.91e-15 &    1.062 \\
       9 & 4.96e-15 &    0.964 & 5.54e-15 &    0.938 \\
       10 & 4.36e-15 &    0.879 & 4.86e-15 &    0.876 \\
       11 & 4.17e-15 &    0.957 & 5.70e-15 &    1.174 \\
       12 & 4.77e-15 &    1.144 & 6.01e-15 &    1.053 \\
       13 & 5.65e-15 &    1.184 & 5.27e-15 &    0.878 \\
     \end{tabular}
     \vspace*{1ex}
     \caption{Two colliding acoustic pulses problem.
       $\|\cdot\|_{S}$-errors for pressure and ECR's of the iteration (\ref{eq:fp_sec2_mod}) during the first and 5th  time step using the second order scheme - first step.
        $\veps=0.1$, $\nuhat = 0.9$, $\nu_{eff} = 9$.}
     \label{tab:ecr_O21_0p1_s}
   \end{center}
 \end{table}
 
 \begin{table}
   \begin{center}
     \begin{tabular}{l|ll|ll}
       \multicolumn{5}{c}{}\\
        & \multicolumn{2}{c}{1st time step} & \multicolumn{2}{c}{5th time step} \\
       $N$ & $S$-error & \mbox{ECR} & $S$-error & \mbox{ECR} \\
           \hline \\
       1 & 6.24e-06 &    1.000 & 1.62e-05 &    1.000 \\
       2 & 2.53e-08 &    0.004 & 1.35e-07 &    0.008 \\
       3 & 1.08e-10 &    0.004 & 1.19e-09 &    0.009 \\
       4 & 4.67e-13 &    0.004 & 1.07e-11 &    0.009 \\
       5 & 3.88e-15 &    0.008 & 9.62e-14 &    0.009 \\
       6 & 2.59e-15 &    0.669 & 2.99e-15 &    0.031 \\
       7 & 2.18e-15 &    0.840 & 2.60e-15 &    0.868 \\
       8 & 2.08e-15 &    0.956 & 2.28e-15 &    0.877 \\
       9 & 2.03e-15 &    0.973 & 2.87e-15 &    1.260 \\
       10 & 2.43e-15 &    1.199 & 2.79e-15 &    0.970 \\
       11 & 2.21e-15 &    0.911 & 2.51e-15 &    0.902 \\
       12 & 2.14e-15 &    0.967 & 2.19e-15 &    0.873 \\
       13 & 2.58e-15 &    1.207 & 2.03e-15 &    0.925 \\
     \end{tabular}
     \vspace*{1ex}
     \caption{Two colliding acoustic pulses problem.
       $\|\cdot\|_{S}$-errors for pressure and ECR's of the iteration (\ref{eq:fp_sec2_mod}) during the first and 5th  time step using the second order scheme - 2nd step.
        $\veps=0.1$, $\nuhat = 0.9$, $\nu_{eff} = 9$.}
     \label{tab:ecr_O22_0p1_s}
   \end{center}
 \end{table}
 
 \begin{table}
   \begin{center}
     \begin{tabular}{l|ll|ll}
       \multicolumn{5}{c}{}\\ 
        & \multicolumn{2}{c}{1st time step} & \multicolumn{2}{c}{5th time step} \\
       $N$ & $W^{1,1}$-error & \mbox{ECR} & $W^{1,1}$-error & \mbox{ECR} \\
           \hline \\
       1 & 1.25e-05 &    1.000 & 5.84e-06 &    1.000 \\
       2 & 8.98e-09 &    0.001 & 5.96e-09 &    0.001 \\
       3 & 1.61e-09 &    0.180 & 1.34e-09 &    0.224 \\
       4 & 1.33e-09 &    0.827 & 1.77e-09 &    1.327 \\
       5 & 1.59e-09 &    1.193 & 1.78e-09 &    1.002 \\
       6 & 1.57e-09 &    0.989 & 1.07e-09 &    0.601 \\
       7 & 1.65e-09 &    1.048 & 1.33e-09 &    1.243 \\
       8 & 1.48e-09 &    0.898 & 1.57e-09 &    1.180 \\
       9 & 1.42e-09 &    0.955 & 1.63e-09 &    1.041 \\
       10 & 1.25e-09 &    0.886 & 1.37e-09 &    0.839 \\
       11 & 1.69e-09 &    1.344 & 1.36e-09 &    0.995 \\
       12 & 1.60e-09 &    0.951 & 1.56e-09 &    1.143 \\
       13 & 1.74e-09 &    1.087 & 1.29e-09 &    0.831 \\
     \end{tabular}
     \vspace*{1ex}
     \caption{Two colliding acoustic pulses problem.
       $\|\cdot\|_{W^{1,1}}$-errors for pressure and ECR's of the iteration (\ref{eq:fp_mod}) during the first and 5th time step using the first order scheme.
        $\veps=0.01$, $\nuhat = 0.9$, $\nu_{eff} = 90$.}
     \label{tab:ecr_O1_0p01_w}
   \end{center}
 \end{table}
 
 \begin{table}
   \begin{center}
     \begin{tabular}{l|ll|ll}
       \multicolumn{5}{c}{}\\ 
        & \multicolumn{2}{c}{1st time step} & \multicolumn{2}{c}{5th time step} \\
       $N$ & $S$-error & \mbox{ECR} & $S$-error & \mbox{ECR} \\
           \hline \\
       1 & 6.71e-06 &    1.000 & 3.71e-06 &    1.000 \\
       2 & 5.06e-09 &    0.001 & 4.15e-09 &    0.001 \\
       3 & 7.95e-10 &    0.157 & 7.23e-10 &    0.174 \\
       4 & 7.13e-10 &    0.898 & 9.90e-10 &    1.369 \\
       5 & 8.00e-10 &    1.121 & 9.48e-10 &    0.958 \\
       6 & 8.39e-10 &    1.048 & 5.74e-10 &    0.606 \\
       7 & 8.72e-10 &    1.040 & 7.35e-10 &    1.280 \\
       8 & 7.68e-10 &    0.881 & 8.62e-10 &    1.172 \\
       9 & 7.67e-10 &    0.998 & 8.78e-10 &    1.019 \\
       10 & 7.00e-10 &    0.913 & 7.70e-10 &    0.876 \\
       11 & 8.77e-10 &    1.253 & 7.56e-10 &    0.983 \\
       12 & 8.19e-10 &    0.934 & 8.37e-10 &    1.107 \\
       13 & 9.32e-10 &    1.137 & 6.92e-10 &    0.827 \\
     \end{tabular}
     \vspace*{1ex}
     \caption{Two colliding acoustic pulses problem.
       $\|\cdot\|_{S}$-errors for pressure and ECR's of the iteration (\ref{eq:fp_mod}) during the first and 5th  time step using the first order scheme.
        $\veps=0.01$, $\nuhat = 0.9$, $\nu_{eff} = 90$.}
     \label{tab:ecr_O1_0p01_s}
   \end{center}
 \end{table}
 
 \begin{table}
   \begin{center}
     \begin{tabular}{l|ll|ll}
       \multicolumn{5}{c}{}\\ 
        & \multicolumn{2}{c}{1st time step} & \multicolumn{2}{c}{5th time step} \\
       $N$ & $W^{1,1}$-error & \mbox{ECR} & $W^{1,1}$-error & \mbox{ECR} \\
           \hline \\
       1 & 7.71e-07 &    1.000 & 1.06e-06 &    1.000 \\
       2 & 9.53e-10 &    0.001 & 1.06e-09 &    0.001 \\
       3 & 8.74e-10 &    0.918 & 6.07e-10 &    0.575 \\
       4 & 9.89e-10 &    1.131 & 6.49e-10 &    1.068 \\
       5 & 8.16e-10 &    0.826 & 8.34e-10 &    1.286 \\
       6 & 9.36e-10 &    1.147 & 8.23e-10 &    0.987 \\
       7 & 1.01e-09 &    1.082 & 8.86e-10 &    1.076 \\
       8 & 9.42e-10 &    0.931 & 8.68e-10 &    0.981 \\
       9 & 9.34e-10 &    0.991 & 8.45e-10 &    0.973 \\
       10 & 8.56e-10 &    0.916 & 7.42e-10 &    0.879 \\
       11 & 8.32e-10 &    0.972 & 9.65e-10 &    1.299 \\
       12 & 8.82e-10 &    1.060 & 9.07e-10 &    0.940 \\
       13 & 8.92e-10 &    1.011 & 8.51e-10 &    0.939 \\
     \end{tabular}
     \vspace*{1ex}
     \caption{Two colliding acoustic pulses problem.
       $\|\cdot\|_{W^{1,1}}$-errors for pressure and ECR's of the iteration (\ref{eq:fp_sec1_mod}) during the first and 5th time step using the second order scheme - first step.
        $\veps=0.01$, $\nuhat = 0.9$, $\nu_{eff} = 90$.}
     \label{tab:ecr_O21_0p01_w}
   \end{center}
 \end{table}

 %

 \clearpage
 
 \begin{center} \textbf{Experimental Convergence Rates (EOC)} \end{center}
 
 In this subsection we study the experimental order of convergence (EOC). Since the  exact solution is not readily available, the EOC can
 be computed using numerical solutions on three grids of sizes $N_1,N_2:=N_1/2,N_3:=N_2/2$ in
 the following way
 \begin{equation}
   \label{eq:eoc}
   \mbox{EOC}:=\log_2\frac{\norm{u_{N_1}^n-u_{N_2}^n}}{\norm{u_{N_2}^n-u_{N_3}^n}}.
 \end{equation}
 Here, $u_{N}^n$ denotes the approximate solution obtained on a mesh of
 $N$ cells at time $t^n$. 
 In
 the following, we first set $\veps=0.1$ so that the problem is weakly
 compressible. In order to compute  EOC numbers, we have successively
 divided the computational domain into $40,80,\ldots,1280$ 
 cells. The final time is always set to $t=0.815$ so that the solution
 remains smooth. Hence, all the limiters are switched off and the
 slopes in the linear recovery are obtained using second order
 central differencing. In Tables~\ref{tab:eoc_rho_O1}-\ref{tab:eoc_p} we
 present the EOC numbers obtained for the $L^1,L^2$ and $L^\infty$ norms, respectively. The
 tables clearly demonstrate the first, respectively second order convergence of our schemes
 at $\veps=0.1$.
 
 
 \begin{table}
   \begin{center}
     \begin{tabular}{l|ll|ll|ll}
       \multicolumn{7}{c}{}\\
       $N$ & $L^1$ error & EOC        & $L^2$ error & EOC        &
       $L^\infty$ error & EOC \\ \hline \hline
  80 & 6.38e-05 & 1.00 & 1.27e-03 & 1.00 & 6.08e-02 & 1.00 \\
       160 & 1.63e-05 & 1.97 & 3.82e-04 & 1.74 & 1.92e-02 & 1.66 \\
       320 & 4.45e-06 & 1.87 & 1.74e-04 & 1.14 & 1.46e-02 & 0.39 \\
       640 & 1.22e-06 & 1.86 & 7.05e-05 & 1.30 & 7.65e-03 & 0.93 \\
       1280 & 3.36e-07 & 1.86 & 2.74e-05 & 1.36 & 4.02e-03 & 0.93 \\
     \end{tabular}
     \vspace*{1ex}
     \caption{Two colliding acoustic pulses problem.
       $L^1,L^2$ and $L^\infty$ errors for density and EOC for the first order scheme (\ref{first_update});
        $\veps=0.1$, $\nuhat = 0.9$, $\nu_{eff} = 9$.}
     \label{tab:eoc_rho_O1}
   \end{center}
 \end{table}
 
 \begin{table}
   \begin{center}
     \begin{tabular}{l|ll|ll|ll}
       \multicolumn{7}{c}{}\\
       $N$ & $L^1$ error & EOC        & $L^2$ error & EOC        &
       $L^\infty$ error & EOC \\ \hline \hline
     80 & 1.27e-03 & 1.00 & 1.93e-02 & 1.00 & 5.10e-01 & 1.00 \\
       160 & 3.00e-04 & 2.08 & 5.84e-03 & 1.72 & 1.72e-01 & 1.57 \\
       320 & 8.70e-05 & 1.79 & 2.59e-03 & 1.17 & 1.26e-01 & 0.45 \\
       640 & 2.14e-05 & 2.02 & 9.21e-04 & 1.49 & 6.76e-02 & 0.90 \\
       1280 & 5.25e-06 & 2.03 & 3.27e-04 & 1.49 & 3.57e-02 & 0.92 \\
       \end{tabular}
     \vspace*{1ex}
     \caption{Analogous results as in Table~\ref{tab:eoc_rho_O1}, but for velocity.}
     \label{tab:eoc_u_O1}
   \end{center}
 \end{table}
 \begin{table}
   \begin{center}
     \begin{tabular}{l|ll|ll|ll}
       \multicolumn{7}{c}{}\\
       $N$ & $L^1$ error & EOC        & $L^2$ error & EOC        &
       $L^\infty$ error & EOC \\ \hline \hline
       80 & 4.00e-05 & 1.00 & 5.92e-04 & 1.00 & 1.31e-02 & 1.00 \\
       160 & 1.97e-05 & 1.02 & 4.45e-04 & 0.41 & 1.68e-02 & -0.35 \\
       320 & 5.20e-06 & 1.92 & 1.88e-04 & 1.25 & 1.30e-02 & 0.37 \\
       640 & 1.42e-06 & 1.87 & 7.96e-05 & 1.24 & 7.86e-03 & 0.72 \\
       1280 & 3.84e-07 & 1.88 & 3.14e-05 & 1.34 & 4.30e-03 & 0.87 \\
     \end{tabular}
     \vspace*{1ex}
     \caption{Analogous results as in Table~\ref{tab:eoc_rho_O1}, but for pressure.}
     \label{tab:eoc_p_O1}
   \end{center}
 \end{table}
 \begin{table}
   \begin{center}
     \begin{tabular}{l|ll|ll|ll}
       \multicolumn{7}{c}{}\\
       $N$ & $L^1$ error & EOC        & $L^2$ error & EOC        &
       $L^\infty$ error & EOC \\ \hline \hline
       80 & 1.12e-04 & 1.00 & 1.72e-03 & 1.00 & 3.98e-02 & 1.00 \\
       160 & 1.38e-05 & 3.02 & 3.20e-04 & 2.43 & 1.15e-02 & 1.79 \\
       320 & 2.19e-06 & 2.66 & 7.61e-05 & 2.07 & 4.32e-03 & 1.41 \\
       640 & 3.74e-07 & 2.55 & 1.88e-05 & 2.02 & 1.54e-03 & 1.49 \\
       1280 & 6.15e-08 & 2.60 & 4.64e-06 & 2.02 & 5.46e-04 & 1.50 \\
     \end{tabular}
     \vspace*{1ex}
     \caption{Two colliding acoustic pulses problem.
       $L^1,L^2$ and $L^\infty$ errors for density and EOC of the second order scheme  (\ref{eq:step1}), (\ref{eq:step2});
        $\veps=0.1$, $\nuhat = 0.9$, $\nu_{eff} = 9$.}
     \label{tab:eoc_rho}
   \end{center}
 \end{table}
 \begin{table}
   \begin{center}
     \begin{tabular}{l|ll|ll|ll}
       \multicolumn{7}{c}{}\\
       $N$ & $L^1$ error & EOC        & $L^2$ error & EOC        &
       $L^\infty$ error & EOC \\ \hline \hline
       80 & 7.98e-04 & 1.00 & 1.39e-02 & 1.00 & 3.92e-01 & 1.00 \\
       160 & 1.65e-04 & 2.28 & 4.37e-03 & 1.67 & 2.06e-01 & 0.93 \\
       320 & 2.66e-05 & 2.63 & 1.03e-03 & 2.09 & 7.22e-02 & 1.51 \\
       640 & 3.99e-06 & 2.74 & 2.22e-04 & 2.21 & 2.17e-02 & 1.74 \\
       1280 & 5.97e-07 & 2.74 & 4.85e-05 & 2.19 & 6.89e-03 & 1.65 \\
       \end{tabular}
     \vspace*{1ex}
     \caption{Analogous results as in Table~\ref{tab:eoc_rho}, but for velocity.}
     \label{tab:eoc_u}
   \end{center}
 \end{table}
 \begin{table}
   \begin{center}
     \begin{tabular}{l|ll|ll|ll}
       \multicolumn{7}{c}{}\\
       $N$ & $L^1$ error & EOC        & $L^2$ error & EOC        &
       $L^\infty$ error & EOC \\ \hline \hline
       80 & 1.63e-04 & 1.00 & 2.67e-03 & 1.00 & 6.08e-02 & 1.00 \\
       160 & 1.80e-05 & 3.17 & 4.35e-04 & 2.62 & 1.58e-02 & 1.94 \\
       320 & 2.69e-06 & 2.74 & 9.54e-05 & 2.19 & 5.67e-03 & 1.48 \\
       640 & 4.67e-07 & 2.53 & 2.39e-05 & 2.00 & 2.24e-03 & 1.34 \\
       1280 & 7.68e-08 & 2.60 & 5.90e-06 & 2.02 & 8.25e-04 & 1.44 \\
     \end{tabular}
     \vspace*{1ex}
     \caption{Analogous results as in Table~\ref{tab:eoc_rho}, but for pressure.}
     \label{tab:eoc_p}
   \end{center}
 \end{table}
 
 Next, we choose a very small value for $\veps$, $\veps = 0.01$.  The results are presented in
 Tables~\ref{tab:eoc_rho1_O1}-\ref{tab:eoc_p1}. Note that we can observe clearly the first and the second order experimental order of convergence of our schemes, despite a small value for $\veps$.
 
 \begin{table}
   \begin{center}
     \begin{tabular}{l|ll|ll|ll}
       \multicolumn{7}{c}{}\\
       $N$ & $L^1$ error & EOC        & $L^2$ error & EOC        &
       $L^\infty$ error & EOC \\ \hline \hline
       80 & 1.63e-06 & 1.00 & 2.64e-05 & 1.00 & 8.49e-04 & 1.00 \\
       160 & 2.59e-05 & -3.99 & 5.11e-04 & -4.27 & 1.44e-02 & -4.08 \\
       320 & 1.42e-05 & 0.86 & 4.30e-04 & 0.25 & 2.54e-02 & -0.82 \\
       640 & 1.05e-06 & 3.76 & 4.63e-05 & 3.21 & 5.03e-03 & 2.34 \\
       1280 & 7.41e-08 & 3.83 & 4.65e-06 & 3.32 & 4.91e-04 & 3.36 \\
       2560 & 1.67e-08 & 2.15 & 1.48e-06 & 1.65 & 2.45e-04 & 1.00 \\
       5120 & 2.86e-09 & 2.54 & 3.67e-07 & 2.02 & 8.69e-05 & 1.50 \\
     \end{tabular}
     \vspace*{1ex}
     \caption{Two colliding acoustic pulses problem.
       $L^1,L^2$ and $L^\infty$ errors for deinsity and EOC of the first order scheme (\ref{first_update});
        $\veps=0.01$, $\nuhat = 0.9$, $\nu_{eff} = 90$.}
     \label{tab:eoc_rho1_O1}
 
   \end{center}
 \end{table}
 \begin{table}
   \begin{center}
     \begin{tabular}{l|ll|ll|ll}
       \multicolumn{7}{c}{}\\
       $N$ & $L^1$ error & EOC        & $L^2$ error & EOC        &
       $L^\infty$ error & EOC \\ \hline \hline
       80 & 9.75e-05 & 1.00 & 1.44e-03 & 1.00 & 3.10e-02 & 1.00 \\
       160 & 1.21e-03 & -3.64 & 2.33e-02 & -4.02 & 5.51e-01 & -4.15 \\
       320 & 1.33e-03 & -0.13 & 3.85e-02 & -0.73 & 1.61e+00 & -1.55 \\
       640 & 1.59e-04 & 3.06 & 6.46e-03 & 2.58 & 4.14e-01 & 1.96 \\
       1280 & 4.78e-05 & 1.74 & 2.78e-03 & 1.21 & 2.26e-01 & 0.88 \\
       2560 & 1.27e-05 & 1.91 & 1.05e-03 & 1.40 & 1.24e-01 & 0.87 \\
       5120 & 3.18e-06 & 2.00 & 3.70e-04 & 1.51 & 6.32e-02 & 0.97 \\
       \end{tabular}
     \vspace*{1ex}
     \caption{Analogous results as in Table~\ref{tab:eoc_rho1_O1}, but for velocity.}
     \label{tab:eoc_u1_O1}
   \end{center}
 \end{table}
 
 \clearpage
 
 \begin{table}
   \begin{center}
     \begin{tabular}{l|ll|ll|ll}
       \multicolumn{7}{c}{}\\
       $N$ & $L^1$ error & EOC        & $L^2$ error & EOC        &
       $L^\infty$ error & EOC \\ \hline \hline
       80 & 9.44e-07 & 1.00 & 1.21e-05 & 1.00 & 2.11e-04 & 1.00 \\
       160 & 2.15e-05 & -4.51 & 4.32e-04 & -5.16 & 1.17e-02 & -5.80 \\
       320 & 7.73e-06 & 1.47 & 2.26e-04 & 0.93 & 1.09e-02 & 0.11 \\
       640 & 9.84e-07 & 2.97 & 4.35e-05 & 2.38 & 3.21e-03 & 1.76 \\
       1280 & 1.00e-07 & 3.30 & 5.89e-06 & 2.88 & 6.01e-04 & 2.42 \\
       2560 & 2.46e-08 & 2.02 & 2.19e-06 & 1.43 & 3.88e-04 & 0.63 \\
       5120 & 3.85e-09 & 2.67 & 5.00e-07 & 2.13 & 1.25e-04 & 1.63 \\
     \end{tabular}
     \vspace*{1ex}
     \caption{Analogous results as in Table~\ref{tab:eoc_rho1_O1}, but for pressure.}
     \label{tab:eoc_p1_O1}
   \end{center}
 \end{table}
 
 \begin{table}
   \begin{center}
     \begin{tabular}{l|ll|ll|ll}
       \multicolumn{7}{c}{}\\
       $N$ & $L^1$ error & EOC        & $L^2$ error & EOC        &
       $L^\infty$ error & EOC \\ \hline \hline
       80 & 4.73e-07 & 1.00 & 7.81e-06 & 1.00 & 2.51e-04 & 1.00 \\
       160 & 8.41e-06 & -4.15 & 2.03e-04 & -4.70 & 9.83e-03 & -5.29 \\
       320 & 2.60e-06 & 1.69 & 9.19e-05 & 1.15 & 6.98e-03 & 0.49 \\
       640 & 1.70e-06 & 0.62 & 7.37e-05 & 0.32 & 7.50e-03 & -0.10 \\
       1280 & 2.46e-07 & 2.78 & 1.54e-05 & 2.26 & 2.10e-03 & 1.83 \\
       2560 & 2.23e-08 & 3.47 & 2.22e-06 & 2.79 & 4.52e-04 & 2.22 \\
       5120 & 2.65e-09 & 3.07 & 4.56e-07 & 2.29 & 1.42e-04 & 1.67 \\
     \end{tabular}
     \vspace*{1ex}
     \caption{Two colliding acoustic pulses problem.
       $L^1,L^2$ and $L^\infty$ errors for density and EOC of the second order scheme  (\ref{eq:step1}), (\ref{eq:step2});
        $\veps=0.01$, $\nuhat = 0.9$, $\nu_{eff} = 90$.}
     \label{tab:eoc_rho1}
   \end{center}
 \end{table}
 
 \clearpage
 
 \begin{table}
   \begin{center}
     \begin{tabular}{l|ll|ll|ll}
       \multicolumn{7}{c}{}\\
       $N$ & $L^1$ error & EOC        & $L^2$ error & EOC        &
       $L^\infty$ error & EOC \\ \hline \hline
       80 & 1.13e-05 & 1.00 & 1.60e-04 & 1.00 & 3.39e-03 & 1.00 \\
       160 & 1.34e-03 & -6.89 & 2.89e-02 & -7.50 & 9.13e-01 & -8.08 \\
       320 & 8.25e-04 & 0.70 & 2.38e-02 & 0.28 & 9.87e-01 & -0.11 \\
       640 & 9.01e-05 & 3.20 & 3.77e-03 & 2.66 & 2.38e-01 & 2.05 \\
       1280 & 2.11e-05 & 2.09 & 1.42e-03 & 1.41 & 1.64e-01 & 0.54 \\
       2560 & 3.40e-06 & 2.64 & 3.24e-04 & 2.13 & 5.04e-02 & 1.70 \\
       5120 & 5.58e-07 & 2.61 & 7.41e-05 & 2.13 & 1.36e-02 & 1.89 \\
       \end{tabular}
     \vspace*{1ex}
     \caption{Analogous results as in Table~\ref{tab:eoc_rho1}, but for velocity.}
     \label{tab:eoc_u1}
   \end{center}
 \end{table}
 \begin{table}
   \begin{center}
     \begin{tabular}{l|ll|ll|ll}
       \multicolumn{7}{c}{}\\
       $N$ & $L^1$ error & EOC        & $L^2$ error & EOC        &
       $L^\infty$ error & EOC \\ \hline \hline
       80 & 3.16e-07 & 1.00 & 4.47e-06 & 1.00 & 8.40e-05 & 1.00 \\
       160 & 1.34e-05 & -5.40 & 2.65e-04 & -5.89 & 6.81e-03 & -6.34 \\
       320 & 3.05e-06 & 2.14 & 9.41e-05 & 1.49 & 4.63e-03 & 0.56 \\
       640 & 2.07e-06 & 0.55 & 8.86e-05 & 0.09 & 6.86e-03 & -0.57 \\
       1280 & 3.24e-07 & 2.68 & 2.05e-05 & 2.11 & 1.99e-03 & 1.78 \\
       2560 & 3.01e-08 & 3.43 & 2.99e-06 & 2.78 & 5.40e-04 & 1.88 \\
       5120 & 3.70e-09 & 3.03 & 6.27e-07 & 2.25 & 1.95e-04 & 1.47 \\
     \end{tabular}
     \vspace*{1ex}
     \caption{Analogous results as in Table~\ref{tab:eoc_rho1}, but for pressure.}
     \label{tab:eoc_p1}
   \end{center}
 \end{table}
 \begin{center} \textbf{Weakly Compressible Flow} \end{center}
 In the third experiment we measure the efficacy of our newly developed scheme to capture weakly
 compressible flow features. We set $\veps=1/11$ as  in
 \cite{klein}. The computational domain is divided into $440$ equal
 mesh points and the CFL number is set to $\nuhat=0.9$ so that
 $\nu_{eff}=9.9$. The boundary conditions are periodic and the slopes
 in the reconstruction are computed using the minmod recovery with
 $\theta=2$. The plots of the pressure obtained using the second order
 scheme at times $t=0.815$ and $t=1.63$ are given in
 Figure~\ref{fig:pulses}, where the initial pressure distributions are
 also plotted for comparison. Note that the initial data represent two
 pulses, where the one on the left moves to the left and the one on the
 right moves to the right. The data being symmetric about $x=0$, the
 periodic boundary conditions act like reflecting boundary
 conditions. Therefore, the pulses reflect back and they superimpose at
 time $t=0.815$, which produces the maximum pressure at $x=0$. The
 pulses separate and move apart and at time $t=1.63$ they assume almost
 like the initial configuration. However, as a result of the weakly
 nonlinear effects, the pulses start to steepen and two shocks are
 about to form at $x=-18.5$ and $x=18.5$ which can be seen from the
 plot at $t=1.63$.
 
 \begin{figure}[htb]
   \centering
   \includegraphics[height=0.275\textheight]{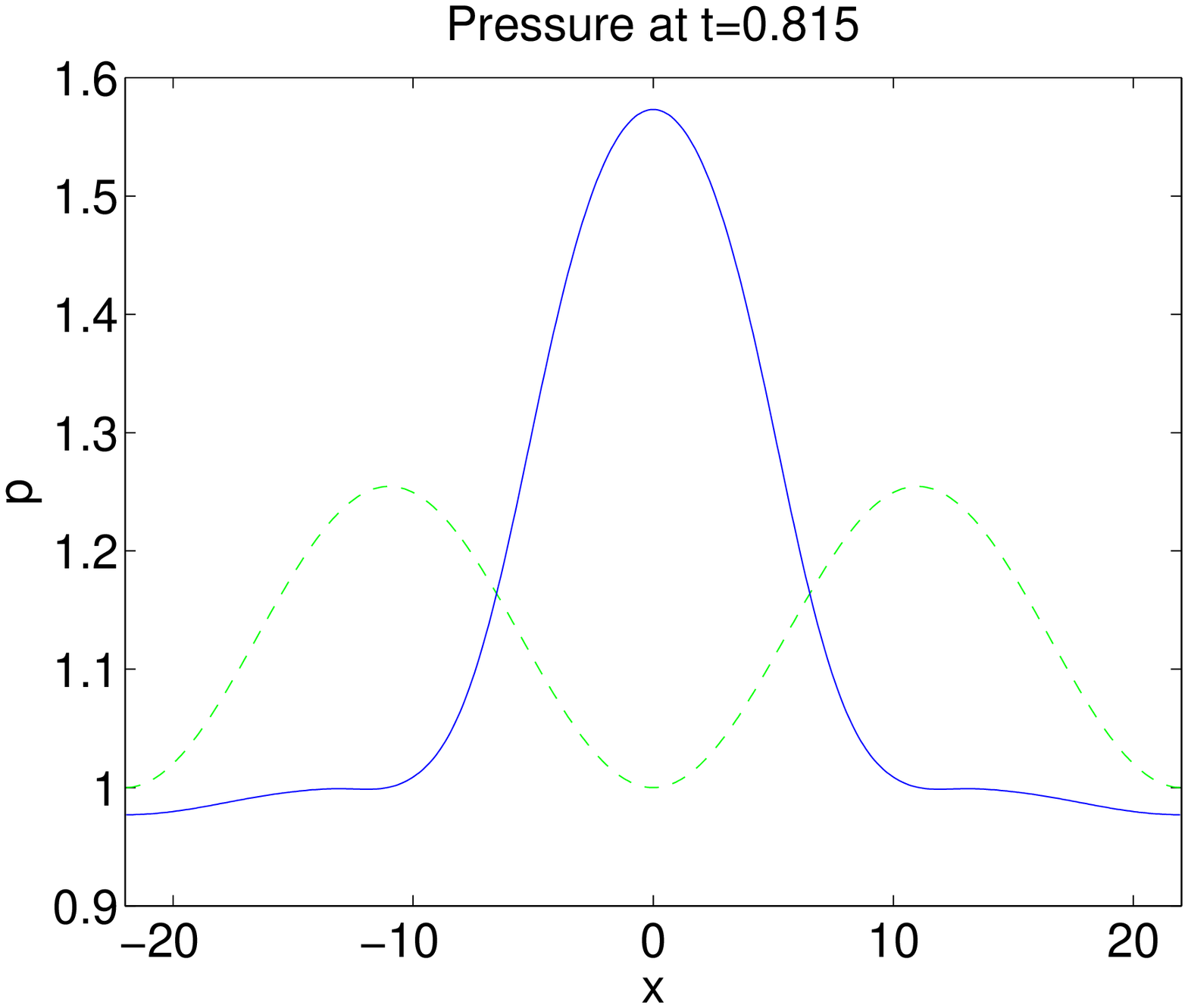}
   \includegraphics[height=0.275\textheight]{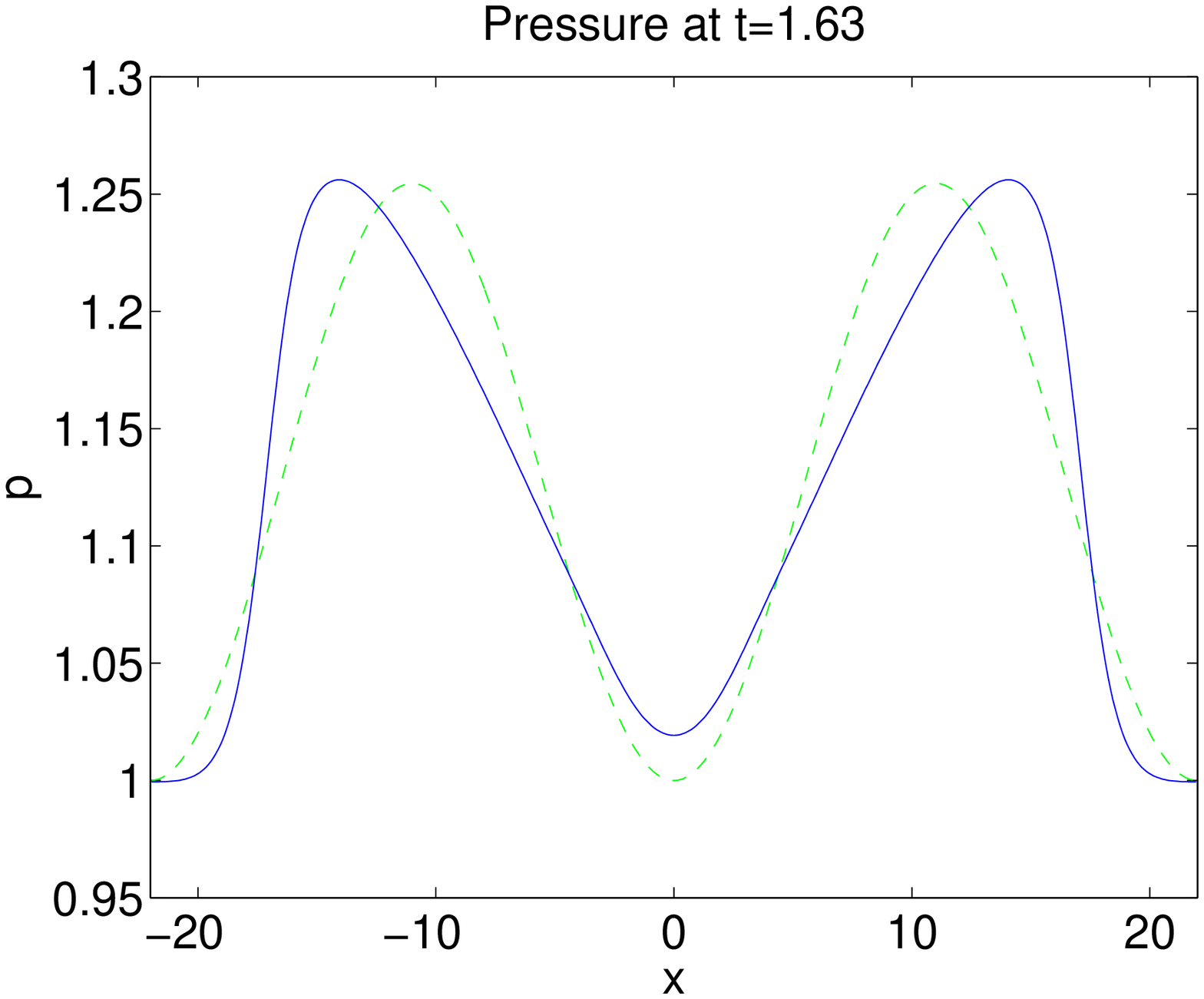}
   \caption{Two colliding acoustic pulses problem computed with the
     second order method. On the left: pressure distribution at
     $t=0.815$. On the right: pressure at $t=1.63$. Dotted line is the
     initial pressure distribution. Here, $\veps=1/11$ and
     $\hat{\nu}=0.9,\nu_{eff}=9.9$, $\cstab=1/12$.}
   \label{fig:pulses}
 \end{figure}
 
 \subsubsection{Density Layering Problem}
 \label{sec:layering}
 
 This test problem is also taken from \cite{klein} and it models the
 propagation of a density fluctuation superimposed in a large
 amplitude, short wavelength pulse. The initial data read
 \begin{align*}
   \rho(x,0)&=\rho_0+\frac{1}{2}\veps\rho_1\left(1+\cos\left(\frac{\pi
         x}{L}\right)\right)+\rho_2\Phi(x)\sin\left(\frac{40\pi
       x}{L}\right),
   \nn
   \rho_0&=1.0, \ \rho_1=2.0, \ \rho_2=0.5,\\
   u(x,0)&=\frac{1}{2}u_0\left(1+\cos\left(\frac{\pi
         x}{L}\right)\right),\ u_0=2\sqrt{\gamma},\\
   p(x,0)&=p_0+\frac{1}{2}\veps p_1\left(1+\cos\left(\frac{\pi
         x}{L}\right)\right),\ p_0=1.0, \ p_1=2\gamma.
 \end{align*}
 The domain is $-L\leq x\leq L=1/\veps$ with $\veps=1/51$. The function
 $\Phi(x)$ is defined by
 \begin{equation*}
   \Phi(x)=
   \begin{cases}
     0, & -\frac{1}{L}\leq x\leq 0,\\
     \frac{1}{2}\left(1-\cos\left(\frac{5\pi
           x}{L}\right)\right), & 0\leq x \leq \frac{2L}{5}, \\
     0, & x>\frac{2L}{5}.
   \end{cases}
 \end{equation*}
 The computation domain is divided into $1020$ equal mesh cells and the
 CFL number is $\nuhat=0.45$, hence $\nu_{eff}=22.95$ . The linear
 reconstruction is performed using the minmod limiter with $\theta=2$
 and the boundary conditions are set to periodic. Note that the initial
 data describe a density layering of large amplitude and small
 wavelengths, which is driven by the motion of a right-going periodic
 acoustic wave with long wavelength. The main aspect of this test case
 is the advection of the density distribution and its nonlinear
 interaction with the acoustic waves. Figure~\ref{fig:layering} show
 the solutions obtained using the second order scheme at time $t=5.071$
 and the initial distributions. Due to the high stifness, we multiplied the pressure correcture with the additional factor $1.4$. It can be noted that up to this time
 the acoustic wave transport the density layer about $2.5$ units and
 the shape of the layer is undistorted. As in the previous problem, due
 to the weakly nonlinear effects, the pulse starts to steepen, leading
 to shock formation.
 
 \begin{figure}[htb]
   \centering
   \includegraphics[height=0.3\textheight]{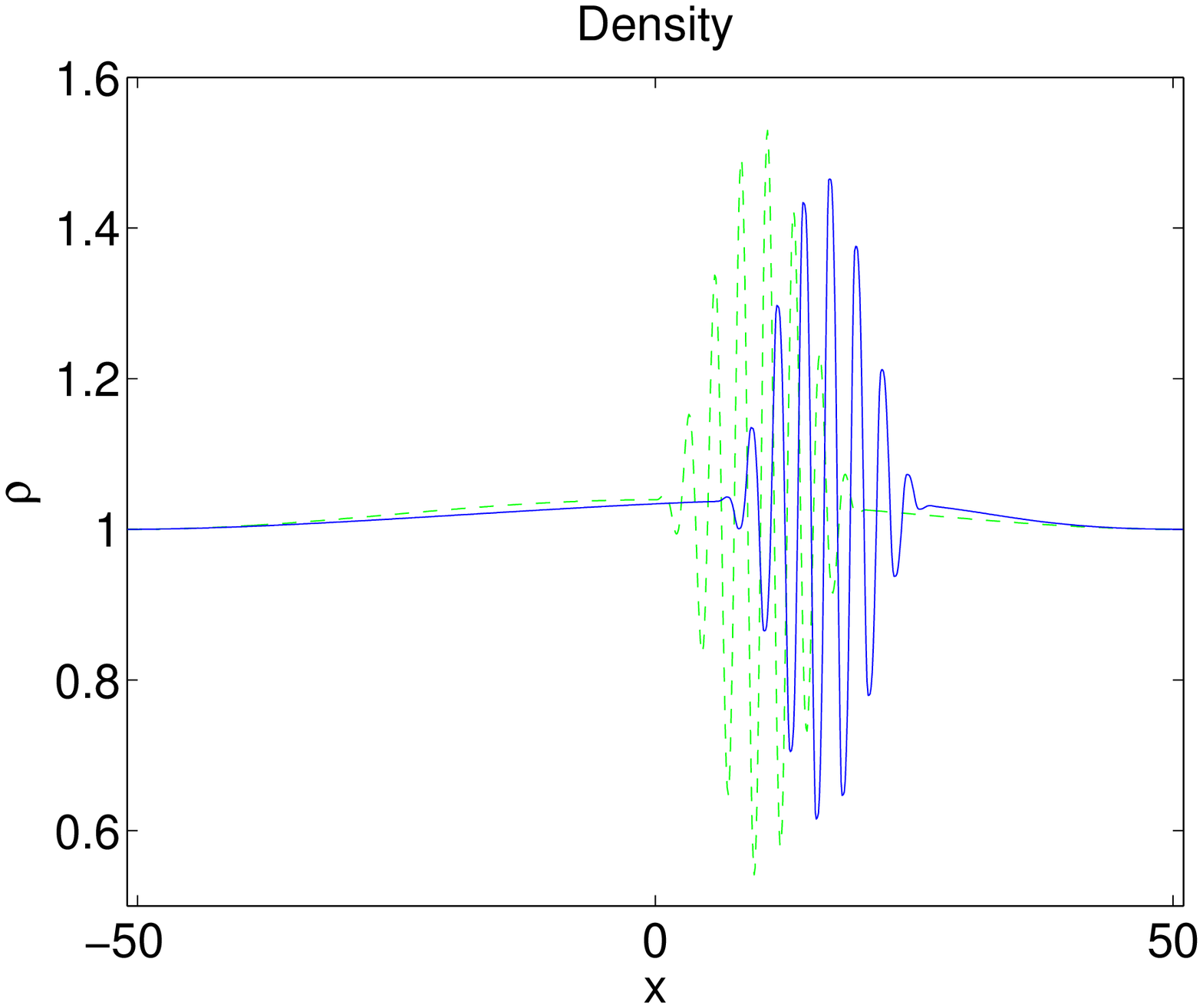}
   \includegraphics[height=0.3\textheight]{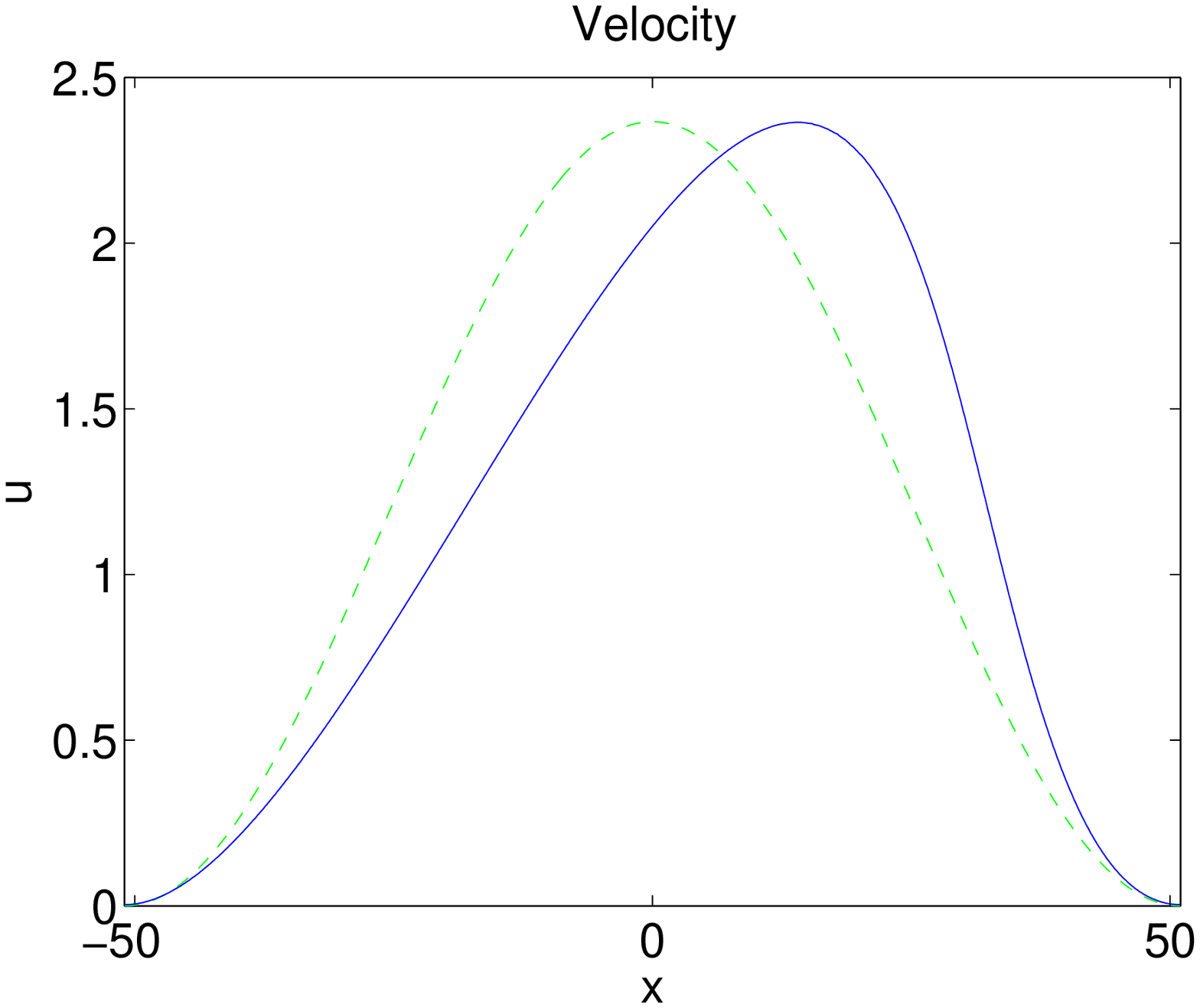}
   \includegraphics[height=0.3\textheight]{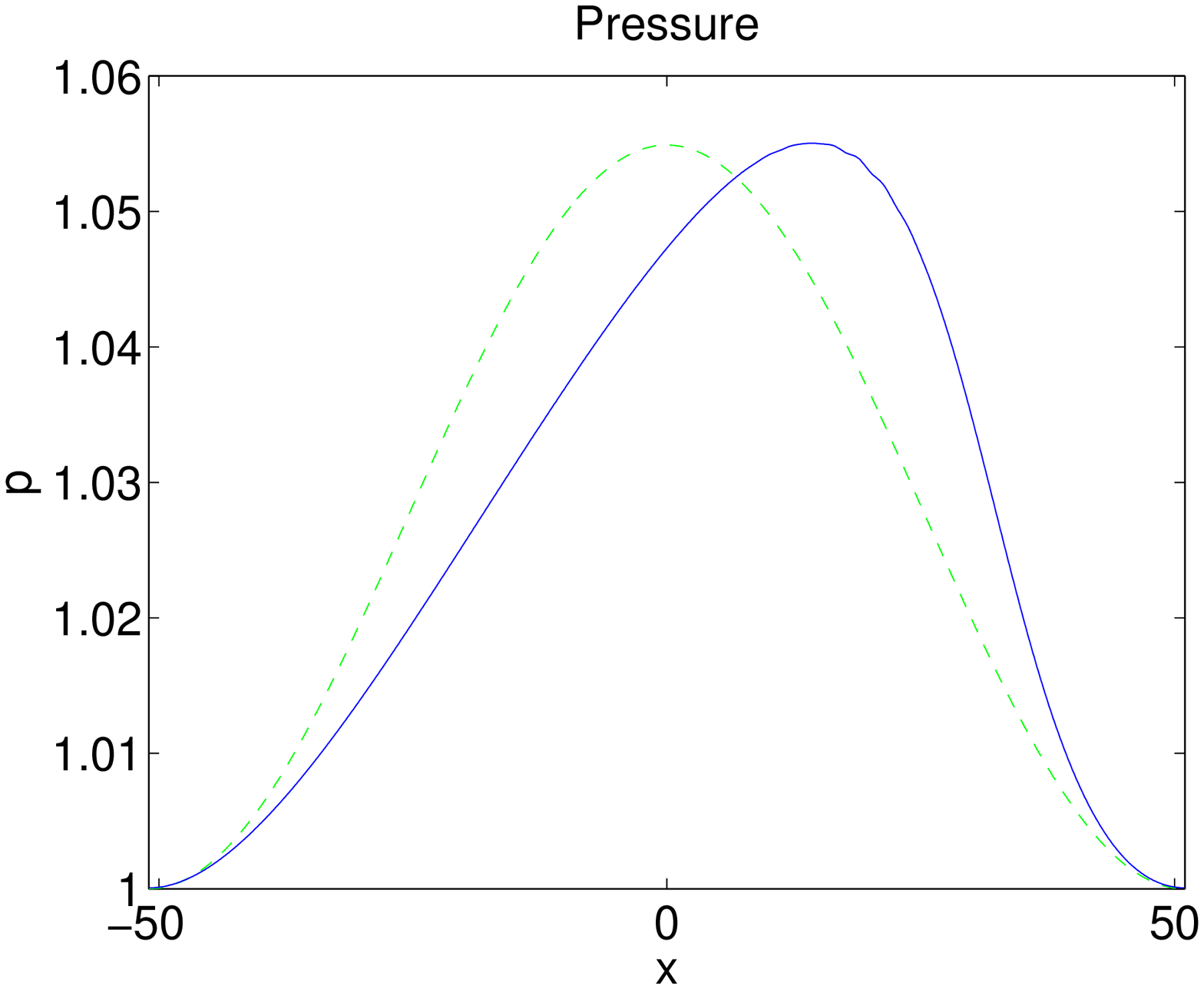}
   \includegraphics[height=0.3\textheight]{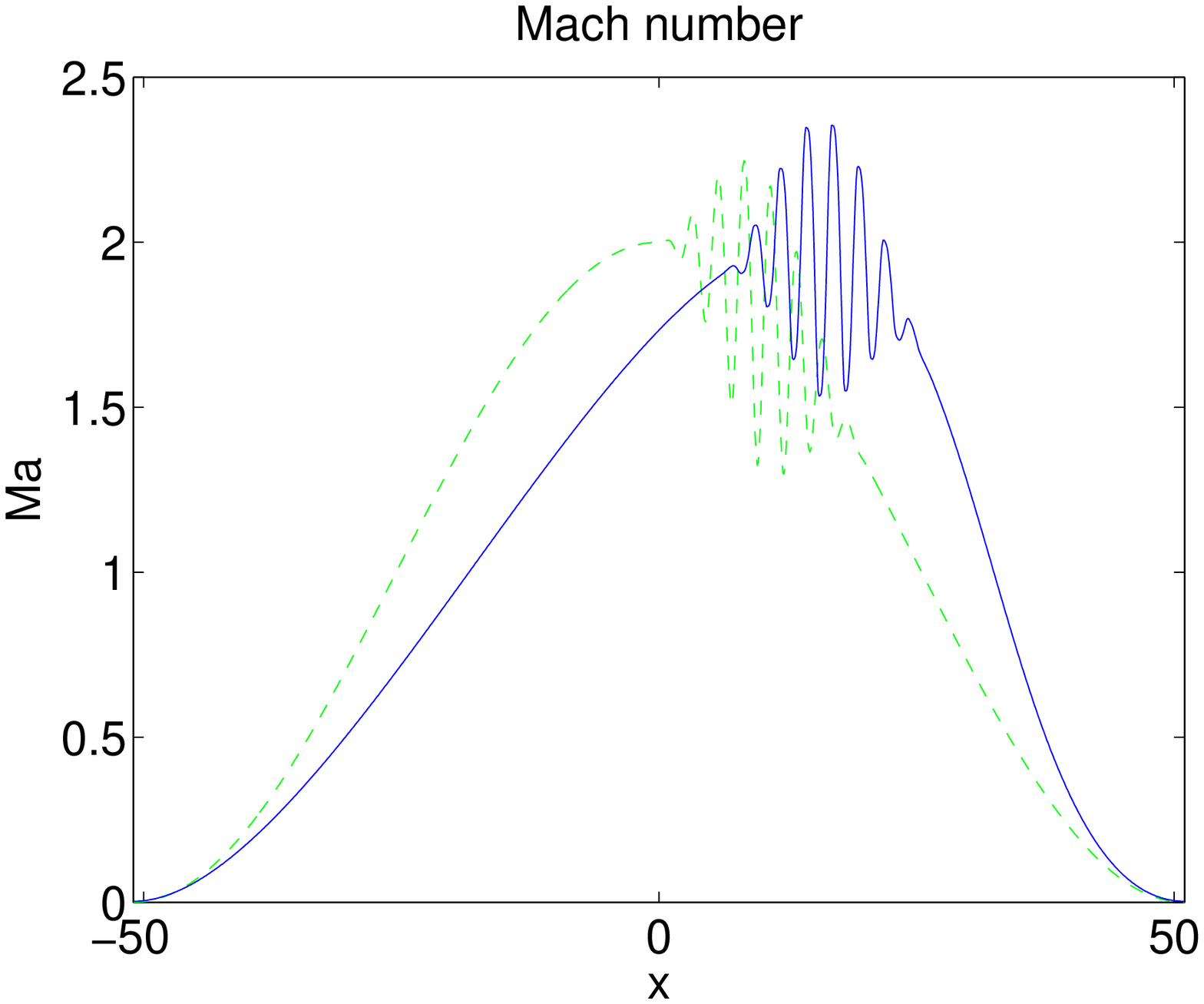}
   \caption{The results of density layering problem at
     $t=5.071$. Dotted lines: initial solution. Solid lines: solution
     obtained by the second order scheme. The CFL numbers are
     $\hat{\nu}=0.45,\nu_{eff}=22.95$, $\cstab=1/12$ and $\veps=1/51$.}
   \label{fig:layering}
 \end{figure}
 
 We use this test also to make a comparison of the first and second
 order schemes. It can be noted from Figure~\ref{fig:layering} that the
 second order scheme preserves the amplitude of the layer, without much
 damping. In order to make the comparison, in
 Figure~\ref{fig:layering_comp}, a zoomed portion of the layer obtained
 by the second order scheme and the first order scheme with the same
 problem parameters is plotted. In the figure the results of first
 order scheme is barely visible due to the excess diffusion.
 
 \begin{figure}[htb]
   \centering
     \includegraphics[height=0.3\textheight]{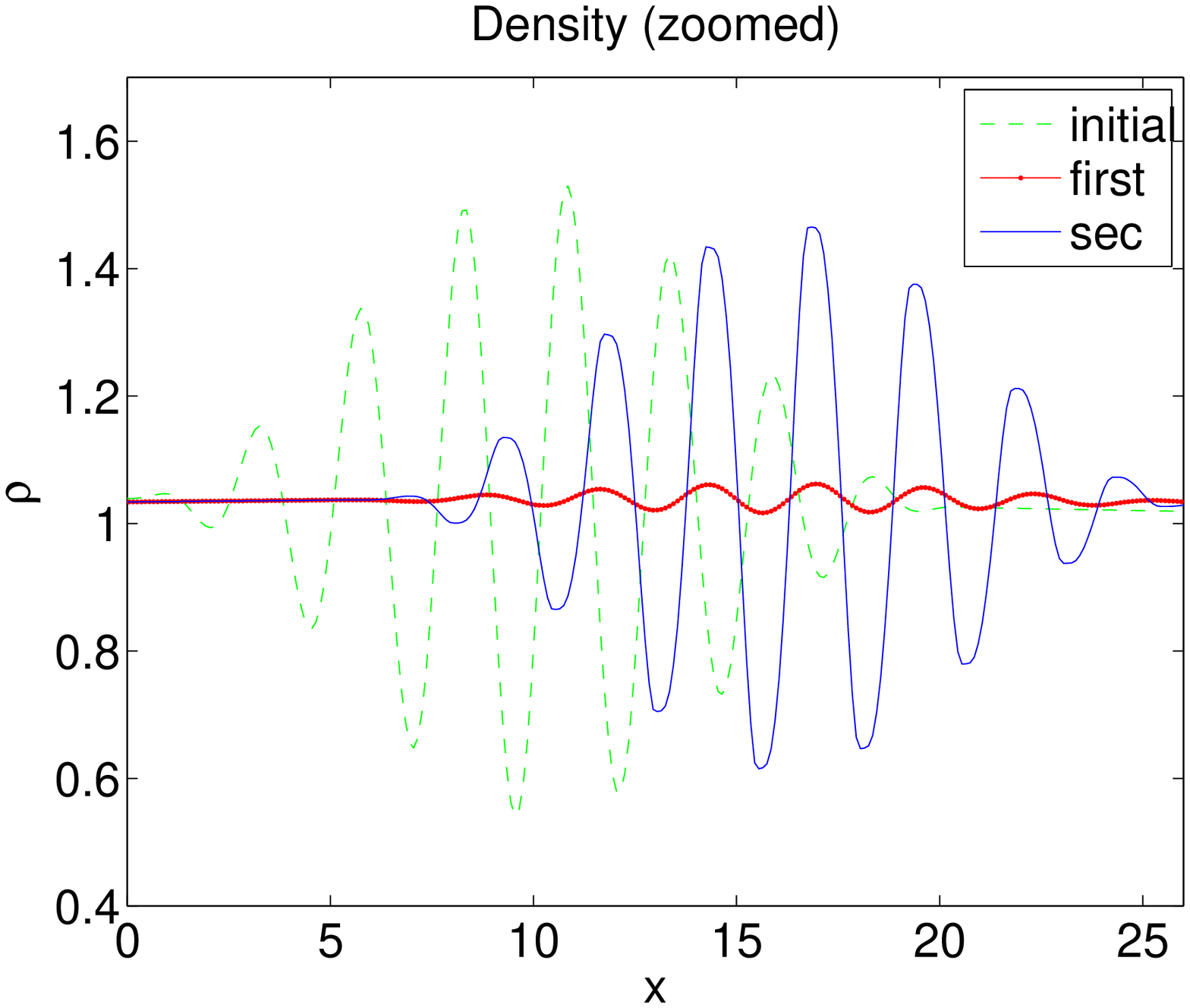}
   \caption{Density layering problem. Comparison of first and second
     order methods at $t=5.071$. $\veps=1/51$, $\hat{\nu}=0.45$,
     $\nu_{eff}=22.95$. $\cstab=1/6$ (first order) respectively $\cstab=1/12$
     (second order method).}
   \label{fig:layering_comp}
 \end{figure}
 
 \subsection{Test Problems in the 2-D Case}
 \label{sec:2-d}
 
 \subsubsection{Gresho Vortex Problem}
 \label{sec:vortex}
 
 In this problem we perform the simulation of the Gresho vortex
 \cite{gresho,liska}. A rotating vortex is positioned at the centre
 $(0.5,0.5)$ of the computational domain $[0,1]\times [0,1]$. The
 initial conditions are specified in terms of the radial distance
 $r=\sqrt{(x-0.5)^2+(y-0.5)^2}$ in the form
 \begin{align*}
   \rho(x,y,0)&=1.0,\\
   u(x,y,0)&=-u_{\phi}(r)\sin\phi,\\
   v(x,y,0)&=u_{\phi}(r)\cos\phi,\\
   p(x,y,0)&=
   \begin{cases}
     p_0+12.5r^2, & \mbox{if} \ 0\leq r\leq 0.2, \\
     p_0+4-4.0\log(0.2)+12.5r^2-20r+4\log r, & \mbox{if} \ 0.2\leq
     r\leq 0.4,\\
     p_0-2+4\log 2, & \mbox{otherwise}.
   \end{cases}
 \end{align*}
 Here, $\tan\phi=(y-0.5)/(x-0.5)$ and the angular velocity $u_\phi$ is
 defined by
 \begin{equation*}
   u_\phi(r)=
   \begin{cases}
     5r, & \mbox{if} \ 0\leq r\leq 0.2, \\
     2-5r, & \mbox{if} \ 0.2\leq r\leq 0.4,\\
     0, & \mbox{otherwise}.
   \end{cases}
 \end{equation*}
 Note that the above data is a low Mach variant of the one proposed in
 \cite{liska} by changing the parameter $p_0=\rho/(\gamma\veps^2)$ as
 given in \cite{happenhofer} to obtain a low Mach number flow. In this
 problem we have set $\veps=0.1$.
 
 The computational domain is divided into $80\times80$ equal mesh cells.
 The boundary conditions to the left and right side are periodic. To the top and bottom sides we apply wall boundary conditions. 
 The slopes in the recovery procedure are obtained by simple
 central differences without using any limiters. 
 In order to stabilize the scheme, we increased the stabilization parameter in
 \eqref{eq:cstab} to $\cstab= \frac1{12}10^4$.
 In Figure~\ref{fig:vortex} the pseudo-colour plots of the Mach numbers are
 plotted which shows that very low Mach numbers of the order $0.01$ are
 developed in the problem. We have set the CFL number $\nuhat=0.45$ and
 therefore, the effective CFL number is $4.5$. In the figure we have
 plotted the results at times $t=1.0,2.0$ and $3.0$ so that the vortex
 completes one, two and three rotations completely. A comparison with
 the initial Mach number distribution shows that the scheme preserves
 the shape of the vortex very well.
 
 \begin{figure}[htbp]
   \centering
   \includegraphics[height=0.275\textheight]{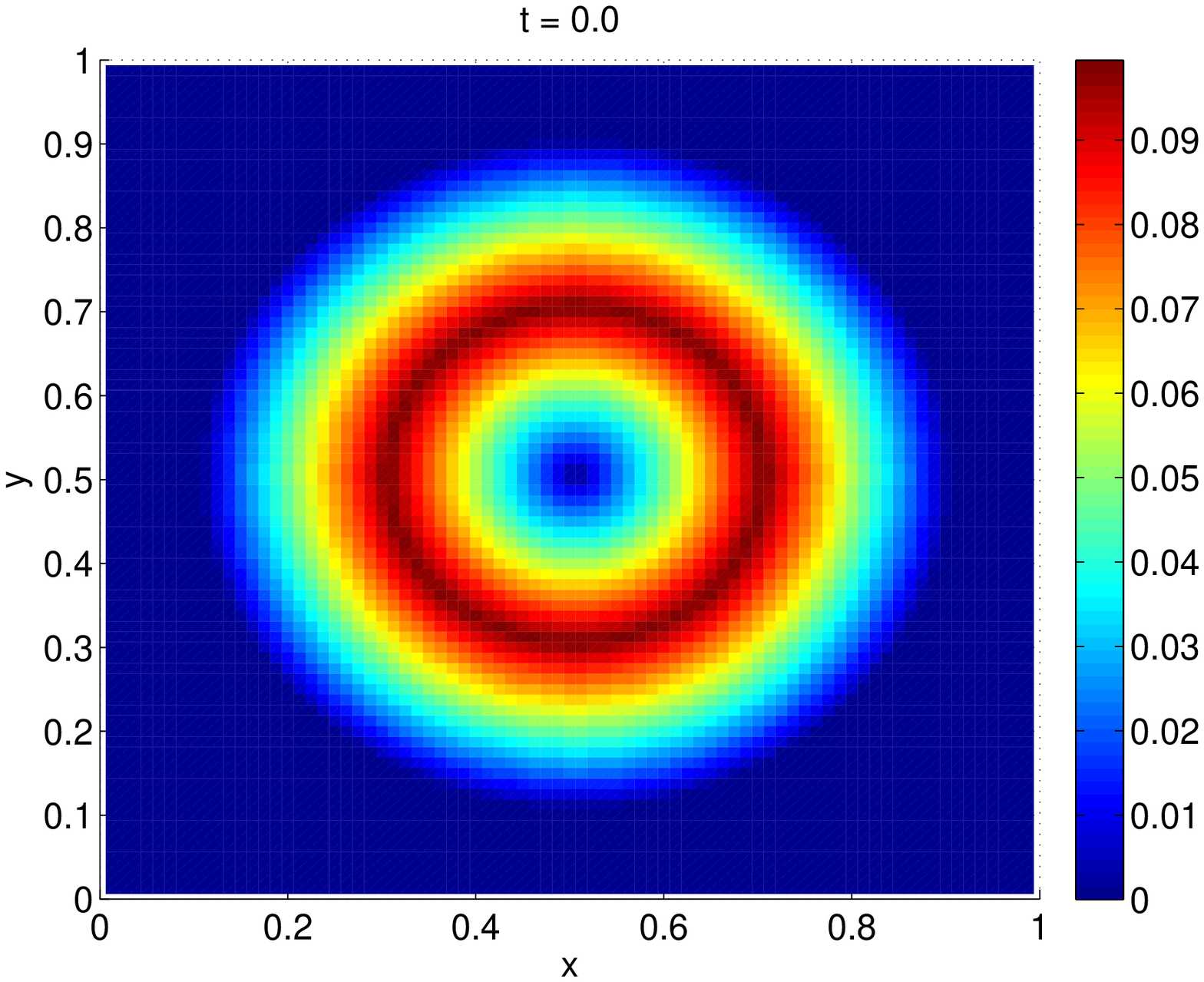}
   \includegraphics[height=0.275\textheight]{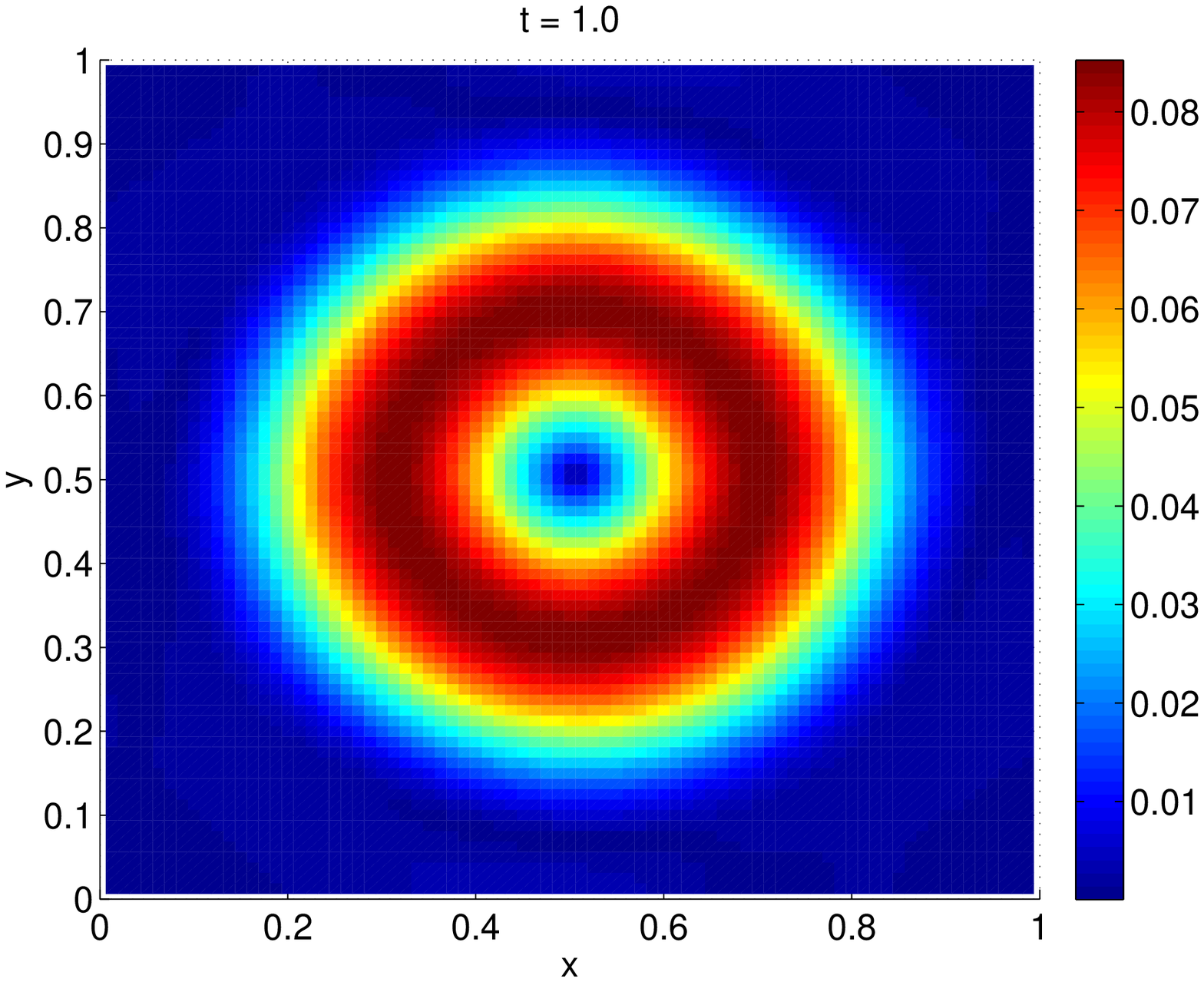}
   \includegraphics[height=0.275\textheight]{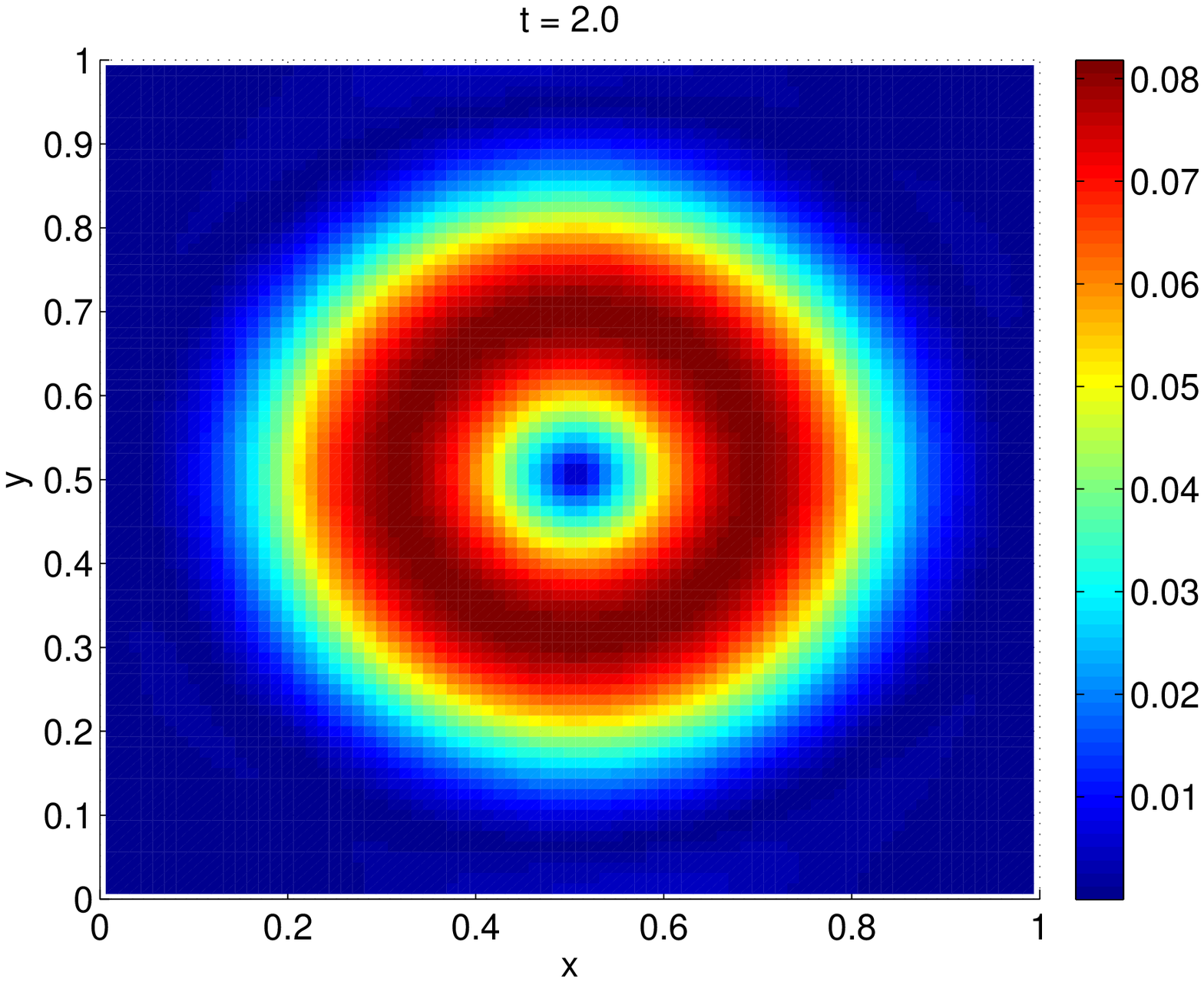}
   \includegraphics[height=0.275\textheight]{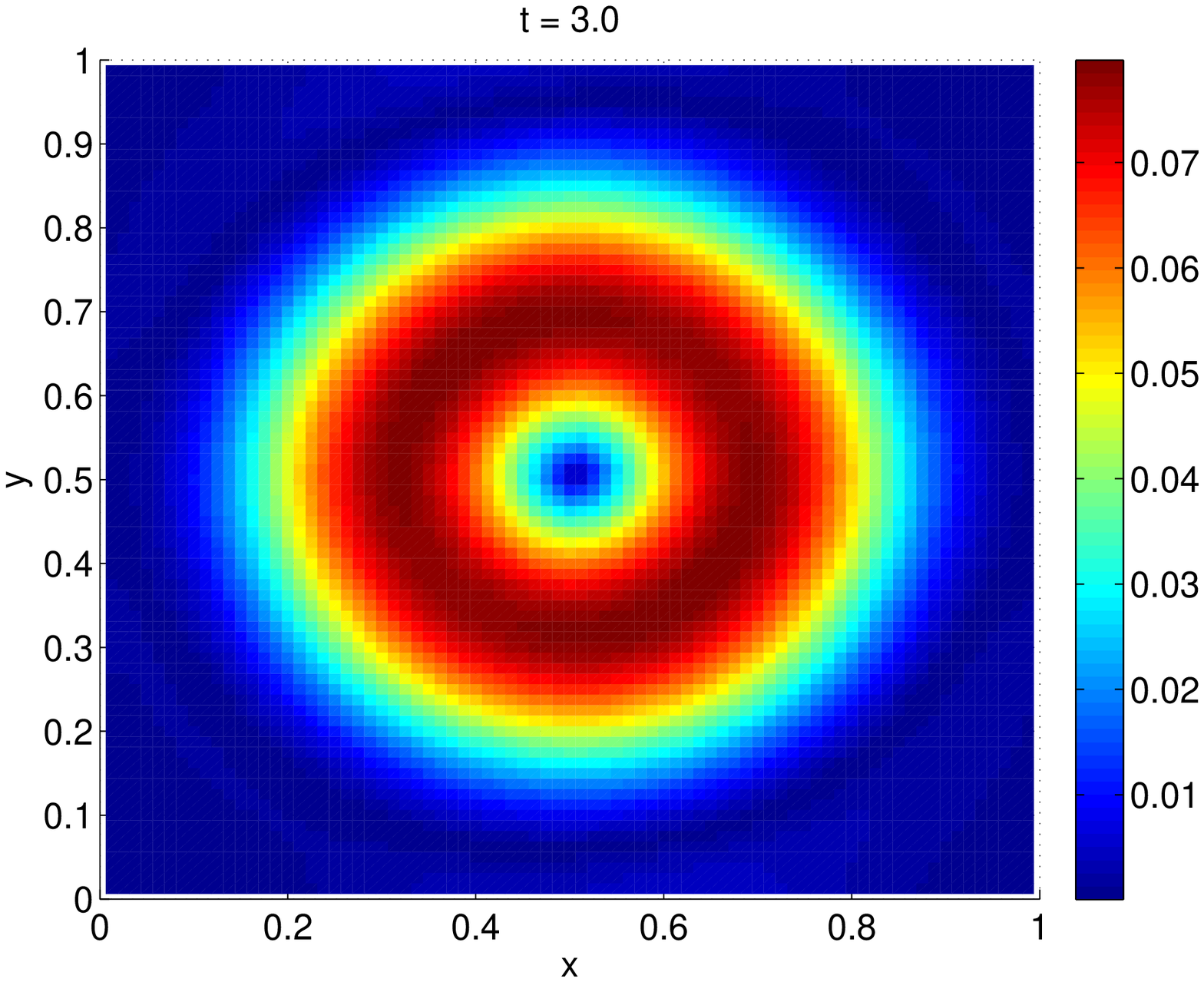}
   \caption{Gresho vortex problem: pseudo-colour plots of the Mach
     number at times $t=0,1,2,3$. In this problem $\veps=0.1$ and the
     CFL numbers are $\hat{\nu}=0.45$ and $\nu_{eff}=4.5$. $\cstab=\frac1{12}10^4$.}
   \label{fig:vortex}
 \end{figure}
 
 \subsubsection{Baroclinic Vorticity Generation Problem}
 \label{sec:baroclinic}
 
 The motivation for this problem is an analogous test studied in
 \cite{geratz}. The setup contains a right-going acoustic wave,
 crossing a wavy density fluctuation in the vertical
 direction. Specifically, the initial data read
 \begin{align*}
   \rho(x,y,0)&=\rho_0+\frac{1}{2}\veps\rho_1\left(1+\cos\left(\frac{\pi
       x}{L}\right)\right)+\Phi(y), \ \rho_0=1.0, \ \rho_1=0.001,\\
 u(x,y,0)&=\frac{1}{2}u_0\left(1+\cos\left(\frac{\pi
       x}{L}\right)\right), \ u_0=\sqrt{\gamma},\\
 v(x,y,0)&=0,\\
 p(x,y,0)&=p_0+\frac{1}{2}\veps p_1\left(1+\cos\left(\frac{\pi
       x}{L}\right)\right), \ p_0=1.0, \ p_1=\gamma.
 \end{align*}
 Here, the problem domain is $-L\leq x\leq L=1/\veps,0\leq y\leq
 L_y=2L/5$ with $\veps=0.05$. The function $\Phi$ is defined by
 \begin{equation*}
   \Phi(y)=
   \begin{cases}
     \rho_2\frac{y}{L_y}, & \mbox{if} \ 0\leq y \leq \frac{L_y}{2},\\
     \rho_2\left(\frac{y}{L_y}-1\right), & \mbox{otherwise},
   \end{cases}
 \end{equation*}
 where $\rho_2=1.8$.
 
The computational domain is divided into $400\times 80$ cells and the CFL
number is $\nuhat=0.45$ so that $\nu_{eff}=9$. The boundary conditions are
periodic and the slope limiting is performed with the minmod recovery with
$\theta=2$. In order to stabilize the scheme, we increased the stabilization
parameter in \eqref{eq:cstab} to $\cstab= \frac1{12}10^2$. The isolines of the
density at times $t=0.0,2.0,4.0,6.0$ and $8.0$ are given in
Figure~\ref{fig:baroclinic}.

 \begin{figure}[h!]
   \centering
   \includegraphics[height=0.18\textheight]{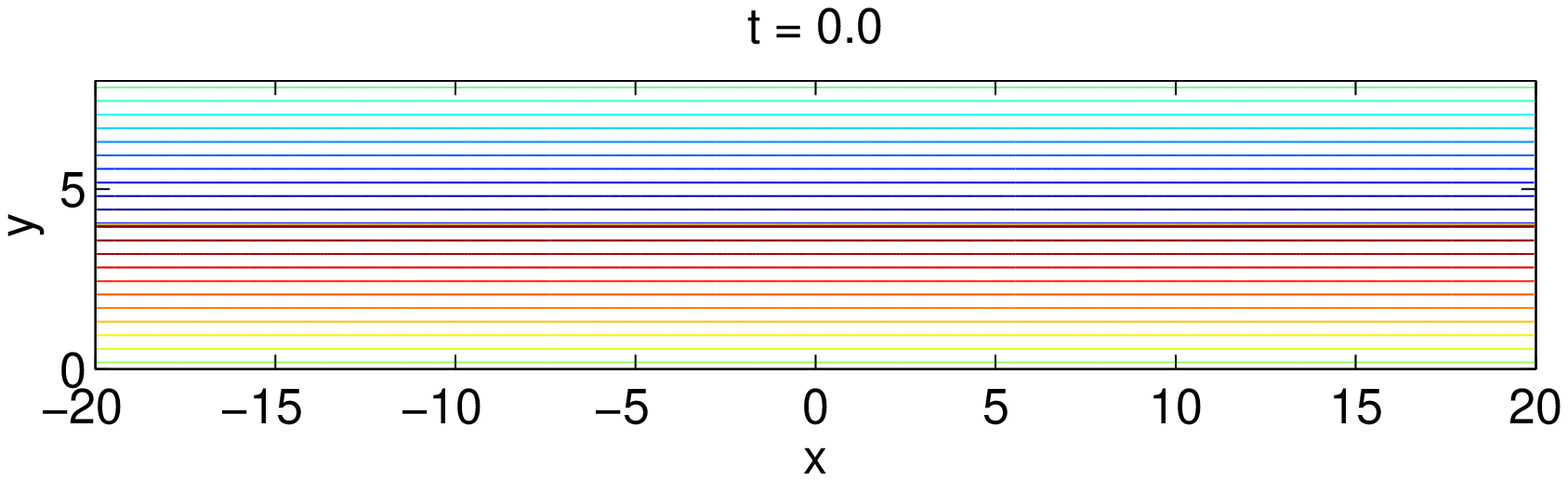} \\
   \includegraphics[height=0.18\textheight]{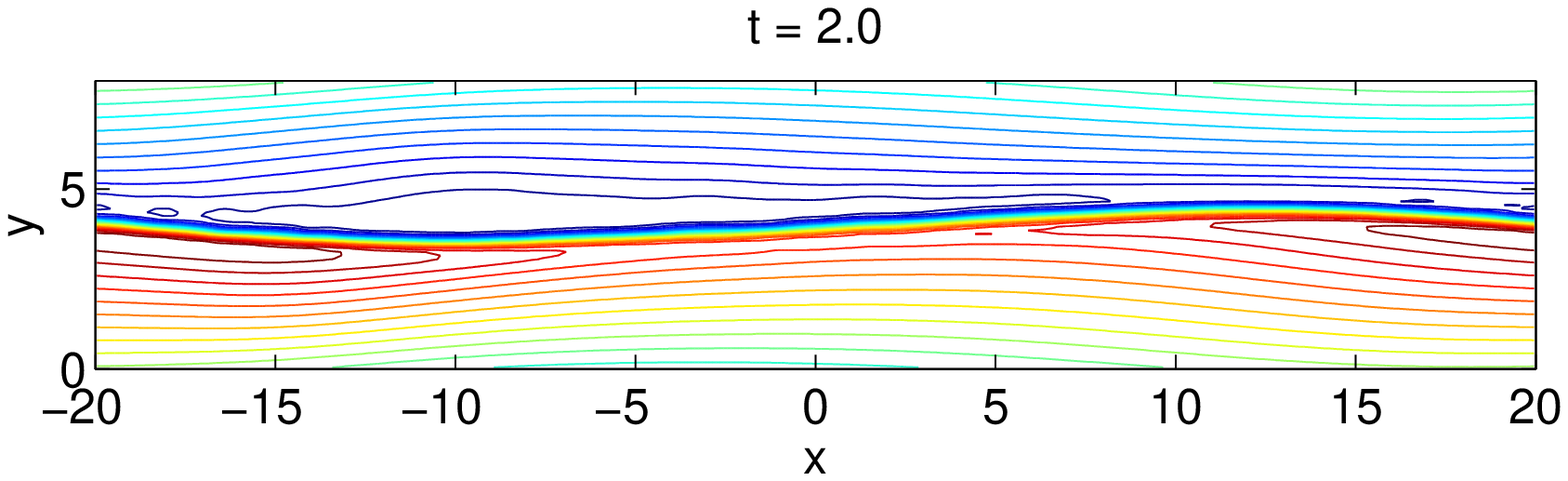} \\
   \includegraphics[height=0.18\textheight]{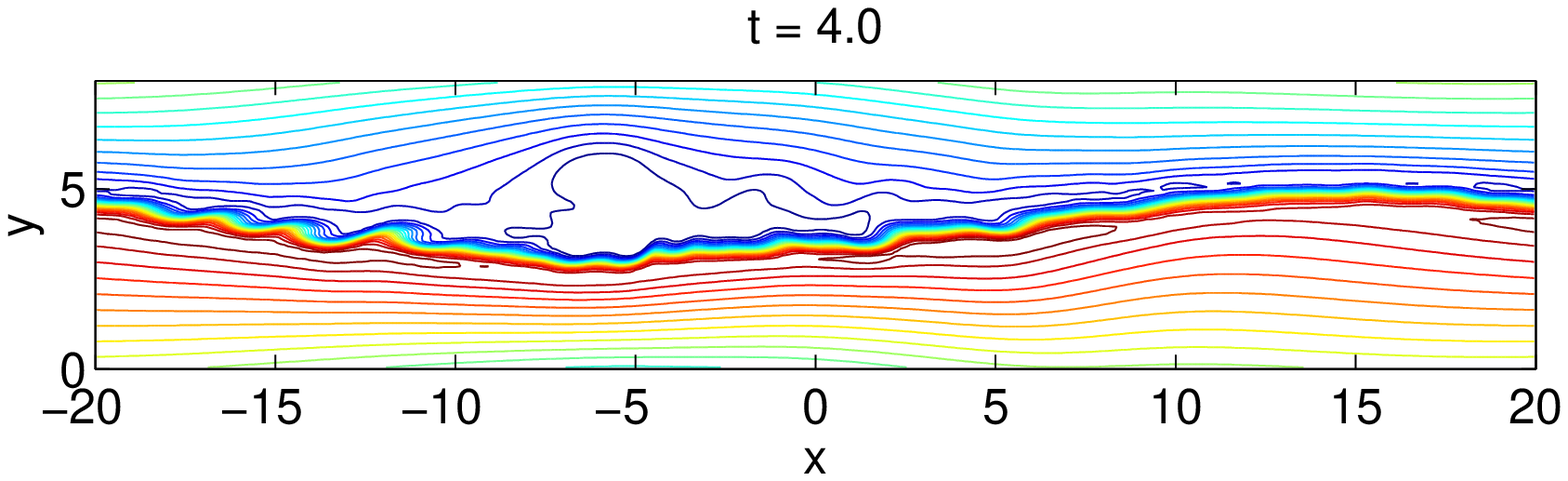} \\
   \includegraphics[height=0.18\textheight]{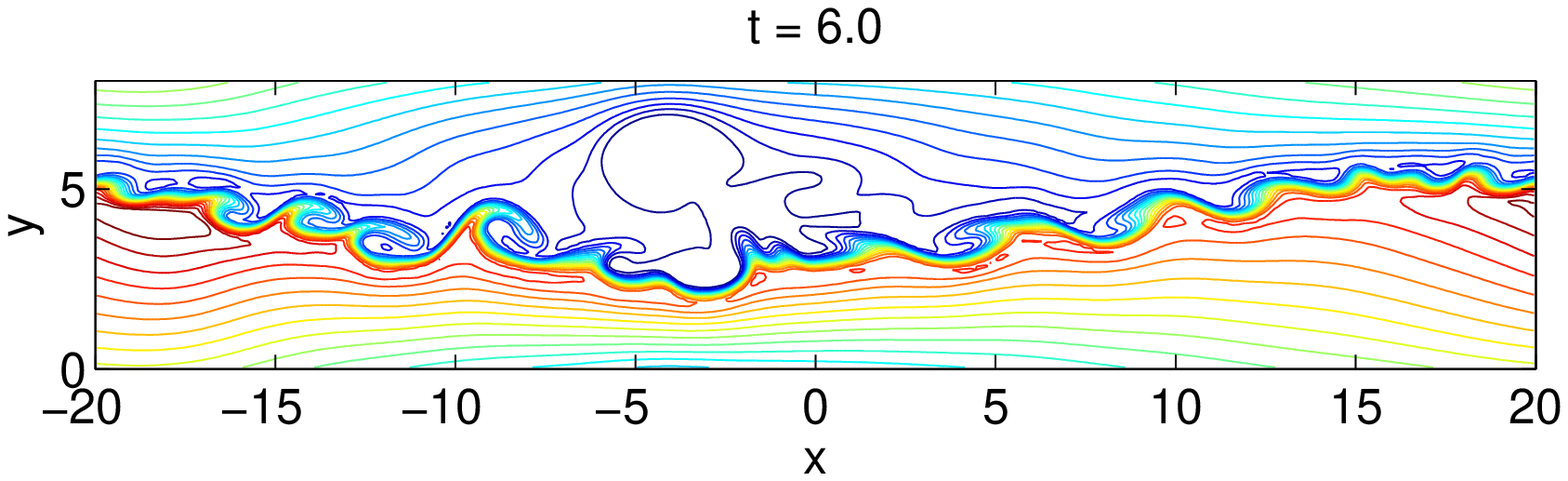} \\
   \includegraphics[height=0.18\textheight]{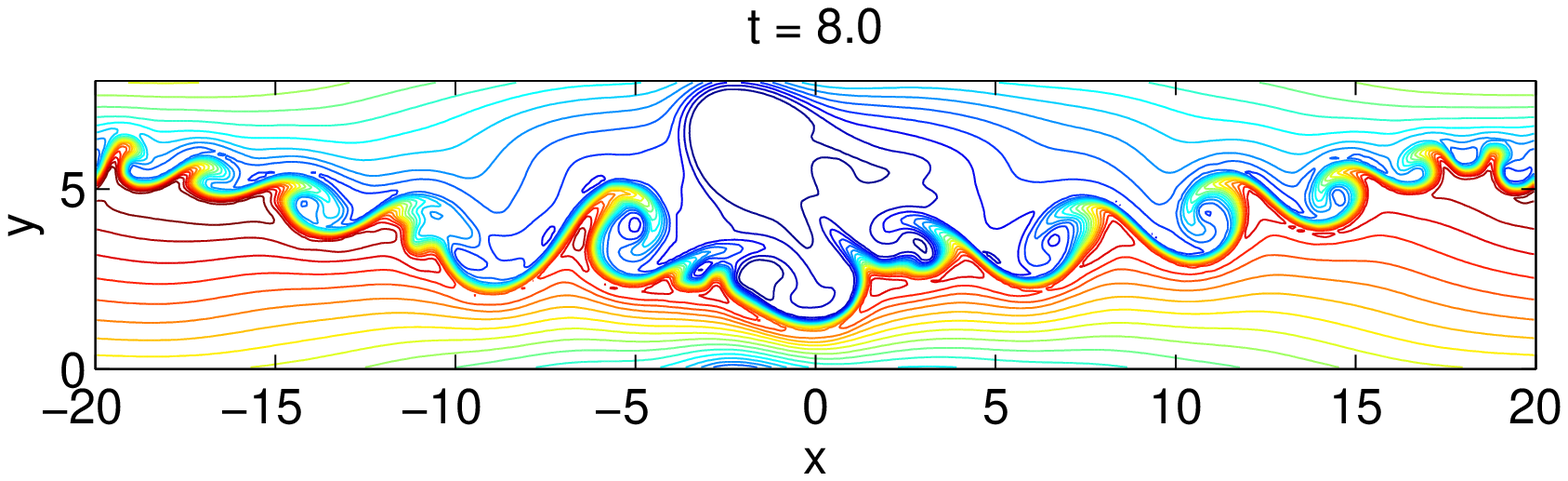} \\
   \caption{Baroclinic vorticity generation in weakly compressible
     flows. The isolines of the density at times $t=0,2,4,6$. The CFL
     numbers are $\nuhat=0.45,\nu_{eff}=9$ and $\veps=0.05$.
     $\cstab=\frac1{12}10^2$.}
   \label{fig:baroclinic}
 \end{figure}
 
 Note that there are two layers in the initial density distribution and
 these two layers are separated by a vertical fluctuation. These
 layers have different accelerations from the acoustic density
 perturbation. Hence, a rotational motion is induced along the
 separating layer. As a result, the well known Kelvin-Helmholtz
 instability develops and long wave sinusoidal shear layers are
 formed. The acoustic waves continue to move and the sinusoidal shear
 layers will now change and become unstable at the edges because of the
 larger density. As a result, they create small vortex structures and
 grow very fast. The problem illustrates how long-wave acoustic
 waves produce small scale flow features.

 
 
 \section{Concluding Remarks}
 \label{sec:conclusion}

 We presented a low Mach number finite volume scheme for the Euler
 equations of gas dynamics. The scheme is based on Klein's
 non-stiff/stiff decomposition of the fluxes \cite{klein},
 and an explicit/implicit time discretization due to Cordier
 et al.\ \cite{cordier}. Inspired by the latter, we replace the stiff energy equation
 by a nonlinear elliptic pressure equation.
 A crucial part of the
 method is a choice of reference pressure which ensures that both the non-stiff and
 the stiff subsystems are hyperbolic, and the second order PDE for the
 pressure is indeed elliptic. The CFL number is only related to the non-stiff
 characteristic speeds, independently of the Mach number.
 
 The second order accuracy of the scheme is based on MUSCL-type reconstructions
 and an appropriate combination of Runge-Kutta and Crank-Nicolson time stepping
 strategies due to Park and Munz \cite{park}. We have proven in
 Theorem~\ref{thm:ap} and \ref{thm:ap2} that our first and second order time
 discretizations are asymptotically consistent, uniformly with respect to
 $\veps$.

Since the Jacobian of the stiff flux function degenerates in the limit, as it
is well-known for finite volume schemes on collocated grids in the
incompressible case, we add a classical fourth order pressure derivative as    
stabilization \cite{FerzigerPeric} to the energy equation. The stabilization
contains a problem-dependent parameter, which is $O(10^{-1})$ for  our
one-dimensional test cases, but substantially higher for the two-dimensional
test cases.  However, for each example, this parameter is independent of the
Mach number. Given the stabilization, we can still prove asymptotic consistency
for a ratio $\dt/\dx$ which is independent of the Mach number, but we have to
reduce the spatial and the temporal grid sizes simultaneously as the Mach number
goes to zero.

According to the results presented here, it seems to be important to
study the effect of flux splittings, time discretizations, stabilization terms
and their interplay upon asymptotic consistency and stability. As a step in this
direction, Sch\"utz and Noelle \cite{NoelleSchuetz_2014} began to study the
modified equation of IMEX time discretizations of linear hyperbolic systems by
Fourier analysis, identifying both stable and unstable splittings. For nonlinear
problems, linearly implicit splittings such as those studied by  Bispen et al. 
\cite{Bispen_diss,BispenArunLukacovaNoelle_2014} show very good asymptotic
stability properties for the shallow water equations in one an two
space-dimensions. Since these correspond to isentropic Euler equations, it is
not clear how this splitting would perform for the non-isentropic gas dynamics,
which we studied in the present paper.

While we have proven uniform asymptotic consistency for our
time-discrete schemes in Theorem~\ref{thm:ap} and \ref{thm:ap2}, an analogous
theoretical result for a fully space-time discrete scheme is missing (for some
recent results in this direction, see ~\cite{haack, Bispen_diss}). Indeed, in
order to obtain a stable space discretization, the stabilization term
\eqref{eq:cstab} had to be introduced. Consequently, our results show that also
in the collocated case, for which it is well-known that it generates
checkerboard instabilities in the incompressible limit, some of the AP
properties can be preserved by a proper stabilization.

The results of some benchmark problems are presented in Section~\ref{sec:numerics}, which
validate the convergence, robustness and the efficiency of the scheme to capture the
weakly compressible flow features accurately.


 


\end{document}